\documentclass[hidelinks,onefignum,onetabnum]{siamart220329}

\usepackage{lipsum}
\usepackage{makecell}
\usepackage{rotating}
\usepackage{graphicx}%
\usepackage{multirow}%
\usepackage{amsmath,amssymb,amsfonts}%
\usepackage{mathrsfs}%
\usepackage[title]{appendix}%
\usepackage{xcolor}%
\usepackage{textcomp}%
\usepackage{manyfoot}%
\usepackage{booktabs}%
\usepackage{algorithm}%
\usepackage{algorithmicx}%
\usepackage{algpseudocode}%
\usepackage{listings}%
\usepackage{subcaption}
\usepackage{setspace}
\usepackage{microtype}
\usepackage[inline]{enumitem}
\usepackage{relsize}
\usepackage{extarrows}
\usepackage{enumitem}
\usepackage{ifthen}
\usepackage{stmaryrd}
\usepackage{boxedminipage}
\usepackage{tikz}
\usepackage{soul}
\usepackage{cancel}
\usepackage{enumitem}
\usepackage{arydshln}
\usepackage[export]{adjustbox}
\usetikzlibrary{
    positioning,    
    arrows.meta,    
    calc           
}

\usepackage{xcolor}          
\usepackage[markup=underlined]{changes} 
\definechangesauthor[name={cc}, color=red]{cc}

\definecolor{lightblue}{rgb}{0.8,0.9,1}

\definecolor{lightgray}{gray}{0.7}

\definecolor{mygreen}{rgb}{0,0.6,0}
\definecolor{mygray}{rgb}{0.5,0.5,0.5}
\definecolor{mymauve}{rgb}{0.58,0,0.82}

\newtheorem{example}{Example}[section]

\newcommand{\normSpectral}[1]{\ensuremath{\left\| #1\right\|}}
\newcommand{\normSpectralnoleftright}[1]{\ensuremath{ \| #1 \|}}
\newcommand{\normFSquare}[1]{\ensuremath{\left\| #1\right\|_F^2}}

\newcommand{\normF}[1]{\ensuremath{\left\| #1\right\|_F}}
\newcommand{\normFnoleftright}[1]{\ensuremath{ \| #1 \|_F}}

\newcommand{\probleftright}[1]{ \ensuremath{ \mathbb{P} \left( #1 \right) } }
 
 \makeatletter
 \newcommand{\ExpectationNoBracket}[2][]{%
     \ifx\\#1\\%
         \ensuremath{ \mathbb{E}{#2} }
     \else
         \ensuremath{ \mathbb{E}_{#1}{#2} }
     \fi
 }
 \makeatother
 \makeatletter
 \newcommand{\Expectation}[2][]{%
     \ifx\\#1\\%
         \ensuremath{ \mathbb{E}\left[#2\right] }
     \else
         \ensuremath{ \mathbb{E}_{#1}\left[#2\right] }
     \fi
 }
 \makeatother
 \makeatletter
 \newcommand{\normP}[2][]{
     \ifx\\#1\\
     \ensuremath{\left\| #2\right\|_p}
     \else
         \ensuremath{ \left\| #2\right\|_{#1} }
     \fi
 }
 \makeatother

\newcommand{\qmat}[1]{\ensuremath{\mathbf{ #1}}}
\newcommand{\diag}{{\rm diag}}

 \newcommand{\rangemat}[1]{ \ensuremath{{\rm range}\left(#1 \right)}  }

\newcommand{\rankrTrun}[1]{ \ensuremath{  \llbracket \qmat{#1}\rrbracket_{r}  } }

 \newcommand{\bdOmega}{ \ensuremath{ \boldsymbol{\Omega}  } }
 \newcommand{\bdtildeOmega}{ \ensuremath{ \tilde{\bdOmega}  } }

 \newcommand{\bdPsi}{ \ensuremath{ \boldsymbol{\Psi}  } }
 \newcommand{\bdPhi}{ \ensuremath{ \boldsymbol{\Phi}  } }

 \newcommand{\bdGamma}{ \ensuremath{ \boldsymbol{\Gamma}  } }

 \newcommand{\bbR}{ \mathbb{R} }
 \newcommand{\bbC}{ \mathbb{C} }

 \newcommand{\bigxiaokuohao}[1]{\ensuremath{ \left(  #1 \right) }}      
    
 \newcommand{\bigdakuohao}[1]{\ensuremath{ \left\{  #1 \right\} }}         
 \newcommand{\bigzhongkuohao}[1]{\ensuremath{ \left[   #1 \right] }}

 \newcommand{\bignorm}[1]{\ensuremath{ \left\|   #1 \right\|  }}

 \newcommand{\trace}[1]{\ensuremath{\rm{trace}\left( #1 \right)}}

	\definecolor{darkgray}{rgb}{0.66, 0.66, 0.66}
  \definecolor{darkgray}{rgb}{0.66, 0.66, 0.66}

  \newenvironment{mytabular2}{\bgroup\tiny\tabular}{\endtabular\egroup}

\ifpdf
  \DeclareGraphicsExtensions{.pdf,.png,.jpg}
\else
  \DeclareGraphicsExtensions{.eps}
\fi

\newsiamremark{remark}{Remark}
\newsiamremark{hypothesis}{Hypothesis}
\crefname{hypothesis}{Hypothesis}{Hypotheses}
\newsiamthm{claim}{Claim}

\headers{Sketch-Power Iterations}{C. Chang \and Y. Yang}

\title{Improving Sketching Algorithms for Low-Rank Matrix Approximation via Sketch-Power Iterations\thanks{Submitted to the editors DATE.
 \funding{This work was funded by the    National Natural Science Foundation of China No. 12171105 and the special foundation for Guangxi Ba Gui
 Scholars No. GXR-6BG2424008.}
}}

\author{Chao Chang\thanks{College of Mathematics and Information Science, Guangxi University, Nanning, 530004, China}
\and Yuning Yang$^\dagger$\thanks{Corresponding author: Yuning Yang, 
  (\email{yyang@gxu.edu.cn}).}}

\usepackage{amsopn}

\ifpdf
\hypersetup{
  pdftitle={Sketch-Power Iteration for One-Pass Low-Rank Approximation},
  pdfauthor={Chao Chang and Yuning Yang}
}
\fi

\begin{document}

\maketitle

\begin{abstract}
 Power iteration can improve the accuracy of randomized SVD,  but requires multiple data passes, making it impractical in streaming or memory-constrained settings. 
 We  introduce a  lightweight yet effective sketch-power iteration,  allowing   power-like iterations with only a single pass of the data, which can be incorporated into one-pass algorithms for low-rank approximation. 
  As an example, we integrate the sketch-power iteration into a one-pass algorithm proposed by Tropp et al., and introduce   strategies to reduce its   storage cost. {\color{black}We   establish meaningful error bounds: given a fixed storage budget, the     sketch sizes derived from the     bounds   closely match   the optimal ones observed in reality. This allows one to preselect reasonable parameters.} Numerical experiments on both synthetic and real-world datasets indicate that, under the same storage constraints, applying one or two sketch-power iterations   can substantially improve the approximation accuracy of the considered one-pass algorithms. \color{black} In particular, experiments on real data with flat spectrum    show that   the method   can approximate the dominant  singular vectors well.   \color{black}
\end{abstract}

\section{Introduction}

Randomized algorithms have emerged as a popular tool for   low-rank matrix approximation. 
Of particular interest is randomized SVD (RSVD), which was proposed in \cite{martinsson2006randomizedalg,libertyRandomizedAlgorithmsLowrank2007} and later systematically refined and analyzed by Halko, Martinsson, and Tropp  \cite{FindingStructureHalko}. Roughly speaking, given a data matrix $\qmat{A}$,  
RSVD   first generates a   sketch $\qmat{Y}=\qmat{A}\bdOmega$    approximating   the range of $\qmat{A}$, with  $\bdOmega$ a random test matrix, orthonormalizes $\qmat{Y}$, and then performs a QB decomposition, followed by an SVD on a much smaller 
  matrix.        RSVD is more efficient than SVD, with detailed theoretical error bounds   provided. 
RSVD requires two passes of the data matrix. 
In another thread, algorithms using one-pass of the data matrix have been developed \cite{woolfe2008Fast,clarkson2009NumericalLinear,Practical_Sketching_Algorithms_Tropp,troppStreamingLowRankMatrix2019,upadhyay2018price,nakatsukasaFastStableRandomized2020}. For example, Tropp et al. \cite{Practical_Sketching_Algorithms_Tropp} considered  taking two sketches of $\qmat{A}$, catching its range and co-range, which is able to compute a QB approximation without revisiting   $\qmat{A}$  again. 
  One-pass algorithms   are   effective for managing data with limited storage, arithmetic, and communication capabilities.  
 Apart from these works, randomized algorithms have been studied extensively in the literature; see, e.g., \cite{FastMontecarlo,woolfe2008Fast,clarkson2009NumericalLinear,mahoney2011RandomizedAlgorithms,woodruff2014SketchingTool,cohen2015dimensionalityreduction,boutsidis2016OptimalPrincipal,nakatsukasaFastStableRandomized2020}, and the recent monographs \cite{tropp2023RandomizedAlgorithms,kannan2017RandomizedAlgorithms,kireeva2024RandomizedMatrix,murray2023RandomizedNumerical,martinsson2020RandomizedNumerical}.

The performance of RSVD is influenced by the decay rate of the singular spectrum. Power iterations $(\qmat{A}\qmat{A}^{T})^q\qmat{A}\bdOmega$ have been incorporated to amplify the gap between singular values, hence  enhancing the decay and improving approximation accuracy  \cite{rokhlin2010randomized,FindingStructureHalko}. 
  Error bounds   have been established, with more refined ones obtained in \cite{gu2015subspace,tropp2023RandomizedAlgorithms}. 
  The   power scheme needs $2q+1$ passes of $\qmat{A}$, while an alternative approach allows for a flexible number of passes \cite{bjarkason2019PassEfficientRandomizedAlgorithms}.


However,   power scheme typically requires multiple passes over the data matrix, limiting its applicability in scenarios with limited  memory or streaming data. 
 \color{black}In such settings, the data  cannot be fully retained in memory, precluding repeated access. \color{black}   Additionally, in privacy-sensitive contexts  where the raw data cannot be exposed \cite{upadhyay2018price},   power iteration is not feasible. Existing literature on incorporating   power iterations into one-pass algorithms to improve accuracy     seems sparse. The only relevant instance   found is  in \cite[Remark 4.2]{yu2018efficient} (see also \cite[Sect. 6.8]{bjarkason2019PassEfficientRandomizedAlgorithms}), which briefly   suggests a ``half'' power iteration: $\qmat{A}\qmat{A}^{T}\bdOmega$, provided column-wise streaming of  $\qmat{A}$; the details will appear   later. 
 These observations motivate the development of   alternative methods that approximate the benefits of power iteration without necessitating multiple data passes.

Motivated by sketching, we replace     $\qmat{A}$ in power iterations with a much smaller sketch $\qmat{Z}$.
Specifically, instead of iterating  $(\qmat{A}\qmat{A}^{T})^q\qmat{A}\bdOmega=(\qmat{A}\qmat{A}^{T})^q\qmat{Y}$, we   perform
\begin{align}
    \label{eq:sps_introduction}
    \hat{\qmat{Y}} = (\qmat{Z}\qmat{Z}^{T})^q\qmat{Y},
\end{align}
where $\qmat{Y}=\qmat{A}\bdOmega$   and $\qmat{Z}=\qmat{A}\bdPhi$ are both sketches of $\qmat{A}$ with $\bdOmega,\bdPhi$ oblivious random  test matrices, so   $\qmat{Y},\qmat{Z}$ can be formed in a single pass of $\qmat{A}$. 
 To be reasonable when storage is constrained,   $\qmat{Z}$ should be significantly smaller than $\qmat{A}$---typically, it is   a few times larger than $\qmat{Y}$, depending on the   storage budget. 
It also allows cheap re-orthonormalization in the process to   mitigate round-off errors. Subsequently, $\hat{\qmat{Y}}$ can be computed    in a single pass of    $\qmat{A}$. We refer to \eqref{eq:sps_introduction} as Sketch-Power Iteration (SPI).  



{The intuition behind SPI can be explained by two scenarios. First, consider when $\qmat{A}$ has insignificant  tail energy (e.g. $\qmat{A}$ has a clear low-rank structure)    corresponding to the sketch size of    $\qmat{Z}$. In this case, $\qmat{Z}$ is a good representation of $\qmat{A}$'s  main subspace, which is especially effective when the tail energy of $\qmat{A}$ is meaningful corresponding to the sketch size of $\qmat{Y}$. 
Alternatively, if the tail energy is non-trivial, $\qmat{Z}$ provides only a partial approximation. Nevertheless, $\qmat{Z}$ is   constructed with a larger sketch size than $\qmat{Y}$, ensuring it captures substantially more   information than $\qmat{Y}$.  In both scenarios, the operation $(\qmat{Z}\qmat{Z}^T)^q\qmat{Y}$ amplifies the approximate leading directions by concentrating the energy of the captured subspace.
In this way, SPI mimics the effect of   power iteration while only requires one pass over $\qmat{A}$. }

SPI can be incorporated into various one-pass algorithms for low-rank approximation; here we integrate SPI into the sketching algorithm proposed by Tropp et al. \cite{Practical_Sketching_Algorithms_Tropp} (TYUC17 for the  original algorithm   and TYUC17-SPI for its SPI enhanced version) to investigate it practically and theoretically. In practice, while a naive implementation of  TYUC17-SPI requires additional storage for  $\qmat{Z}$, we propose a mixed-precision and two hardware-dependent strategies to ensure that it maintains comparable   storage cost to TYUC17. 
 Numerical evidence   demonstrates that TYUC17-SPI achieves a substantial improvement in accuracy over TYUC17,  with   one or two sketch-power iterations. \color{black}In particular, experiments on real data   exhibiting flat decay spectrum   show that enhanced with one   iteration, TYUC17 can approximate the dominant  singular vectors well. In the supplementary material, we also integrate SPI into the streaming algorithm of Tropp et al. \cite{troppStreamingLowRankMatrix2019}.\color{black}    

  Two error bounds of TYUC17-SPI have been    established,     depending   on the sketch sizes, $q$, and tail energy.   The first   is an expectation bound in the Frobenius norm for the $q=1$ case, and the second one is a probabilistic deviation bound  in the  spectral norm for general $q$ cases. 
  \color{black} 
   Given a fixed   storage budget,  we have used the first bound  to derive      sketch sizes for different singular spectrum decay types and  rates. Empirically, the resulting sketch sizes are closely aligned with the   optimal choices (oracle ones). This allows us to preselect sketching parameters to achieve higher accuracy under theoretical guidance.  The bounds in \cite{Practical_Sketching_Algorithms_Tropp,troppStreamingLowRankMatrix2019} have also been employed for this purpose, and our work extends their results in this direction.   
  \color{black}

 The remainder is organized as follows. Sect. \ref{sec:SPI} formally introduces the sketch-power method and   its integration with TYUC17. Sect. \ref{sec:error_bound} provides theoretical error bounds and related implications, with proofs deferred to Sect. \ref{sec:analysis} and \ref{sec:error_anal_q>1}. Sect. \ref{sec:numer_exp} provides  numerical experiments to validate its accuracy.     
 Some related approaches  are discussed in Sect. \ref{sec:related_approaches}. 
 Sect. \ref{sec:conclusion} draws conclusions.  \color{black}

\emph{Notation.} $\rankrTrun{A}$ means the best rank-$r$ approximation of a matrix $\qmat{A}$. $\cdot^\dagger$ represents the Moore-Penrose   inverse. $\|\cdot\|_p$ standards for the matrix Schatten-$p$ norm; in particular, $\|\cdot\|_F$ and $\|\cdot\|$ respectively means the Frobenius and spectral norms. $\tau_{k}(\qmat{A})$ represents the tail energy of $\qmat{A}$, i.e., $\tau^2_{k}(\qmat{A}) = \sum_{i\geq k}\sigma_i^2$, where $\sigma_i$ is the $i$-th singular value of $\qmat{A}$, arranged in a descending order.  $\kappa(\cdot)$ stands for the condition number.

 \emph{Code.}  The code   is available at \url{github.com/Mitchell-Cxyk/SketchPowerIteration}.










\section{Sketch-Power Iteration}\label{sec:SPI}
This section begins with a review of RSVD  \cite{FindingStructureHalko}, power iteration, and the TYUC17 algorithm  \cite{Practical_Sketching_Algorithms_Tropp}. Next, the sketch-power iteration is proposed and integrated into TYUC17, along with strategies to reduce its storage cost. Finally, we discuss how to choose the sketch sizes for different spectrum decays.
\subsection{RSVD and power iteration} Randomized SVD (RSVD) was proposed in \cite{martinsson2006randomizedalg,libertyRandomizedAlgorithmsLowrank2007} and then systematically refined and analyzed in \cite{FindingStructureHalko}.
Given a data matrix $\qmat{A}\in\bbR^{m\times n}$ and the target rank $r$,     RSVD takes the following steps:

\begin{small}
    1. Draw a random test matrix $\bdOmega\in\bbR^{n\times s}$ ($s > r $ with $p=s-r$ the oversampling parameter) and compute        $\qmat{Y} = \qmat{A}\bdOmega$;  

2. Compute the orthonormal basis of $\qmat{Y}$ by thin QR: $[\qmat{Q},\sim] = \texttt{QR}(\qmat{Y},0)$;

3. Compute  $ \qmat{B} = \qmat{Q}^{T}\qmat{A}$ and its best rank-$r$ approximation  $\rankrTrun{B}$;

4. Output $\hat{\qmat{A}}:= \qmat{Q} \rankrTrun{B}$ as the rank-$r$ approximation of $\qmat{A}$. 
\end{small} 


RSVD works well when the spectrum of $\qmat{A}$ has certain decay; otherwise, the small singular values may influence     $\qmat{Y}$ to produce a poor range  \cite{FindingStructureHalko}. To overcome this, the power iteration has been   employed to compute $\qmat{Y}$ \cite{rokhlin2010randomized,FindingStructureHalko,gu2015subspace,tropp2023RandomizedAlgorithms}: 
\begin{align}
    \label{eq:standard_power_scheme}
    \qmat{Y} = (\qmat{A}\qmat{A}^{T})^q\qmat{A}\bdOmega;
\end{align}
 usually a few iterations are enough. 
To avoid   numerical issues, instead of \eqref{eq:standard_power_scheme}, a re-orthonormalization procedure can be employed: 
\begin{align}
    \label{eq:standard_power_scheme_reorth}
    {\rm Repeat~}q~{\rm times}:~ [\qmat{X} ,\sim] = \texttt{QR}(\qmat{A}^{T}\qmat{Y},0),~[\qmat{Y},\sim]=\texttt{QR}(\qmat{A}\qmat{X} ,0).
\end{align}
\eqref{eq:standard_power_scheme_reorth} gives the same basis as that of \eqref{eq:standard_power_scheme} but is more numerically stable. Additional flops for orthonormalizations are $O(q\cdot(ms^2+ns^2))$. 

In practice, there is no need to perform two QRs in each iteration \cite{voronin2015rsvdpack}. An option is to replace QR by LU before the last iteration \cite{li2017algorithm}. \eqref{eq:standard_power_scheme} and \eqref{eq:standard_power_scheme_reorth} require  $2q$ passes of $\qmat{A}$
, while \cite{bjarkason2019PassEfficientRandomizedAlgorithms} proposed a method that allows any number of passes.

\subsection{The TYUC17 algorithm}\label{sec:practical_sketching_alg} RSVD needs two passes of $\qmat{A}$, which is  prohibitive in some applications. Several   one-pass algorithms have been proposed and studied   \cite{woolfe2008Fast,clarkson2009NumericalLinear,woodruff2014SketchingTool,cohen2015dimensionalityreduction,boutsidis2016OptimalPrincipal,troppStreamingLowRankMatrix2019,nakatsukasaFastStableRandomized2020,upadhyay2018price}.
We mainly recall that of Tropp et al. \cite{Practical_Sketching_Algorithms_Tropp}, which is a numerically stable version of the proposal of \cite[Thm. 4.9]{clarkson2009NumericalLinear}. 
  \cite{Practical_Sketching_Algorithms_Tropp} can be regarded as a modification of RSVD to  the one-pass scenario. 
First draw   random matrices   $\bdOmega \in\mathbb{F}^{n\times s}$ and $\boldsymbol{\bdPsi}\in\mathbb{F}^{l\times m} $ independently with $r\leq s\leq d \ll \min\{m,n\}$. Construct two sketches:
	\begin{align}\label{eq:Y=AOmegaW=PsiA}
		\qmat{Y}=\qmat{A}\bdOmega\in\mathbb{F}^{m\times s} \quad\text{and}\quad \qmat{W}=\bdPsi \qmat{A}\in\mathbb{F}^{d\times n}.
	\end{align}
	Here   $\qmat{W}$ captures the co-range of $\qmat{A}$. As RSVD, $\qmat{Q}$ is   computed from the thin QR of $\qmat{Y}$. 
Then the   QB approximation to $\qmat{A}$ is obtained as   
	\begin{align}\label{eq:reconstructionOfTwoSketch}
		\qmat{A}\approx\qmat{Q}\qmat{B},\quad \text{with}\quad \qmat{B}=\left(\bdPsi \qmat{Q}\right)^\dagger \qmat{W}\in\mathbb F^{s\times n} \Leftrightarrow \qmat{B}\in \arg\min_{\qmat{X}}\nolimits{\|\bdPsi(\qmat{A}-\qmat{Q}\qmat{X})\|_F}.
	\end{align}
    Finally the rank-$r$ approximation is given by $\hat{\qmat{A}} = \qmat{Q}\rankrTrun{B}$. The   algorithm only needs a single pass of $\qmat{A}$, 
   and $\qmat{Y},\qmat{W}$ can   be generated in parallel for   efficiency.   It is a stabilized   oblique projection method  \cite{nakatsukasaFastStableRandomized2020}. \color{black} 
    In the sequel, we call the algorithm   \emph{TYUC17}.   
    \subsection{Sketch-power iteration} \label{sec:subsection_SPS}
    TYUC17 exhibits strong numerical performance under rapidly decaying spectra; yet, as RSVD, its performance tends to deteriorate when the spectral decay is   flat.  
    However, the power iteration is prohibitive, 
    as it needs extra $2q$   passes of $\qmat{A}$, violating the nature of one-pass algorithms.

To address this limitation, motivated by sketching algorithms, we consider replacing $\qmat{A}$ in the power iterations    by a sketch,   i.e.,  
we   perform the following to obtain the rangefinder matrix:

\begin{boxedminipage}{0.92\textwidth}
    \begin{center}
        Sketch-Power Iteration
   \end{center}
\begin{align}
    \label{eq:sketch_power_scheme}
    &\hat{\qmat{Y}} = (\qmat{Z}\qmat{Z}^{T})^q\qmat{Y} = \qmat{Z}(\qmat{Z}^T\qmat{Z})^{q-1}\qmat{Z}^T\qmat{A}\bdOmega,\\
    &\qmat{Z}=\qmat{A}\bdPhi\in\bbR^{m\times l};~\bdPhi\in\bbR^{n\times l}~{\rm oblivious~random~test~matrix}.\nonumber
\end{align}
\end{boxedminipage}

\noindent   
Here $l\ll \min\{m,n\}$. In both practice and theory,   $l>s$ is necessary; otherwise $\hat{\qmat{Y}}$ cannot retain more information than $\qmat{Y}$. 
We call \eqref{eq:sketch_power_scheme} the Sketch-Power Iteration (SPI). As both $\qmat{Z}$ and $\qmat{Y}$ can be simultaneously sketched through a single pass of $\qmat{A}$, SPI is   well-suited for integration into one-pass algorithms. 

  

\subsubsection{Intuition} A natural intuition of SPI comes from the subspace embedding perspective.
Consider the case  $\qmat{A} = \qmat{A}_1 + \qmat{A}_2$, where $\qmat{A}_1$ is a rank-$k$ matrix carrying the principal information of $\qmat{A}$, and $\qmat{A}_2$ is a perturbation with relatively small magnitude. In particular,   assume that the   spectrum of 
$\qmat{A}_1$ decays not so rapidly, while there is a clear spectral gap between 
$\qmat{A}_1$ and $\qmat{A}_2$. This corresponds to the ``clear low-rank'' scenario   
in the introduction. \color{black} The goal is to find a rank-$r$ approximation ($r < k$). 

  If one simply takes $\qmat{Y} = \qmat{A}\bdOmega$ (with $k  > s > r$) as a rangefinder, then $\qmat{Y}$ may deviate considerably from the     subspace in interest.  By contrast, if we   construct $\qmat{Z} = \qmat{A}\bdPhi$ of an appropriately large size (e.g., $l > 2k$), we can include enough of $\qmat{A}_1$'s main spectral directions. 
 Denote   $\qmat{Z}_1=\qmat{A}_1\bdPhi$; one can see this by noting
  \[
    \qmat{Z}\,\qmat{Z}^{T}\qmat{Y}= \qmat{Z}_1\qmat{Z}_1^{T}\;\qmat{A}_1\,\bdOmega 
    + f({\qmat{A}_1},\qmat{A}_2)\,\bdOmega,
  \]
  
where $f(\qmat{A}_1,\qmat{A}_2)$ collects all terms involving $\qmat{A}_2$, including cross terms with $\qmat{A}_1$.  
For the first term, by the Johnson-Lindenstrauss property, the subspace embedding $\qmat{A}_1\bdPhi$   preserves the action of operator $\qmat{A}_1$  that helps correct  the direction of $\qmat{Y}$, making   $\qmat{Z}_1 $ act like $\qmat{A}_1$, as well as $\qmat{A}$ (since $\qmat{A}_1$ well approximates $\qmat{A}$). For Gaussian embedding,   $l>2k$ is usually enough to ensure this. 
\color{black}
 For the second term, $\qmat{A}_2$ carries insignificant energy, so $f(\qmat{A}_1,\color{black}\qmat{A}_2)\,\bdOmega$ contributes   a relatively small perturbation. Thus, by choosing suitable $l$ and  $s$, we preserve the key spectral information of $\qmat{A}_1$ while limiting overhead. Consequently,   SPI   approximately replicates the subspace in interest typically revealed by a standard power iteration, yet it does so with only a single pass over $\qmat{A}$.
  
  \color{black}

\subsubsection{Stabilization} Instead of directly performing \eqref{eq:sketch_power_scheme},  
re-orthonormalization   can be employed to avoid round-off issues. 
Additional cost for stabilization of SPI only takes $O(q\cdot l s^2)$ flops, independent of $m,n$. We shortly introduce it as follows.  
For $\qmat{Y}$ at hand, first denote $\hat{\qmat{Y}}=\qmat{Y}$; the procedure takes the following form:

\begin{boxedminipage}{0.92\textwidth}
\begin{center}
     Sketch-Power Iteration with Re-orthonormalization
\end{center}
\begin{align}
    \label{eq:sketch_power_scheme_reorth}
   {\rm Repeat~}q~{\rm times}:~ [\qmat{X},\sim]=\texttt{QR}(\qmat{Z}^{T}\hat{\qmat{Y}},0),~\hat{\qmat{Y}}=\qmat{Z}\qmat{X}. \quad\quad \textcolor{lightgray}{\% \qmat{Z}^{T}\hat{\qmat{Y}}\in\bbR^{l\times s}}
\end{align}
\end{boxedminipage}

\noindent The range of $\hat{\qmat{Y}}$ given by \eqref{eq:sketch_power_scheme_reorth} is equivalent to that of \eqref{eq:sketch_power_scheme}, provided that both $\qmat{Z}$ and $\qmat{Y}$ have full column rank. To see this, just mark $\qmat{X}$ and $\hat{\qmat{Y}}$ above with their iteration number (particularly $\hat{\qmat{Y}}_0=\qmat{Y}$), and explicitly write $\qmat{Z}^{T}\hat{\qmat{Y}}_p = \qmat{X}_p\qmat{R}_p$; then $\qmat{Z}^{T}\hat{\qmat{Y}}_p\in\bbR^{l \times s}$ and $\hat{\qmat{Y}}_p$ all have full column rank and $\qmat{R}_p$ are invertible.  Iteratively express $\hat{\qmat{Y}}_q$ to see that 
\[
\hat{\qmat{Y}}_q = \qmat{Z}\qmat{Z}^{T}\hat{\qmat{Y}}_{q-1}\qmat{R}_q^{-1} = \cdots =(\qmat{Z}\qmat{Z}^{T})^q\hat{\qmat{Y}}_0\qmat{R}_q^{-1}\cdots \qmat{R}_{1}^{-1}. 
\]
Comparing \eqref{eq:sketch_power_scheme_reorth} with \eqref{eq:sketch_power_scheme}, only additional $q$ thin QR on matrices $\qmat{Z}^{T}\hat{\qmat{Y}}$ sizing $l \times s$ are required. This extra  cost is independent of the size of $\qmat{A}$, 
such that stabilization in SPI is both computation- and storage-efficient   for large-scale data. 


  

\subsubsection{Test matrices} Just as   sketching algorithms, various oblivious random test matrices are available for SPI. Frequently used examples include  Gaussian, Rademacher \cite{clarkson2009NumericalLinear}, sparse \cite{woodruff2014SketchingTool}, subsampled randomized Fourier transform (SRFT) \cite{boutsidis2013improved,FindingStructureHalko,woolfe2008Fast}, and sub-Gaussian \cite{saibaba2023randomized}.  
Gaussian embedding is the most widely used in the theoretical analysis. On one hand, it usually provides optimal guarantees and often gives   sharper bounds. On  the other hand, other randomized embeddings may behave similarly to a Gaussian embedding in practice due to the  central limit theorem \cite{martinsson2020RandomizedNumerical}. Thus, the results from Gaussian theory are often valuable for the general behavior. 

\subsubsection{Complexity} 
For dense $\qmat{A}$  
and $\bdPhi$ being   unstructured distributions,  one first forms $\qmat{Z}\in\bbR^{m\times l}$, needing $mnl$ flops; an iteration of SPI computes $\qmat{X} = \qmat{Z}^{T}\qmat{Y}\in\bbR^{l\times s}$ and $\qmat{Z}\qmat{X}\in\bbR^{m\times s}$, taking $2mls$ flops. Thus SPI with re-orthonormalization takes about $mnl + qml s+O(qls^2)$ flops. 

For structured   embeddings  such as SRFT or sparse test matrices, computing   $\qmat{Z}$ costs much less flops. For SRFT, the flops of computing $\qmat{Z}$ are $O(mn\log(l ))$, and for sparse $\bdPhi$, computing $\qmat{Z}$ requires $m\cdot {\rm nnz}(\bdPhi)$ flops, where ${\rm nnz}(\cdot)$ denotes the number of non-zero entries of a matrix. Thus the main cost of SPI with SRFT or sparse $\bdPhi$ is $ O(mn\log(l )) + 2qml s+O(qls^2)$ and $m\cdot {\rm nnz}(\bdPhi)+ 2qml s+O(qls^2)$, respectively.



\subsection{TYUC17 with sketch-power iterations}
  SPI only acts on the rangefinder, and so it is easily incorporated in various one-pass algorithms.   We mainly investigate SPI with the TYUC17 algorithm  revisited in Sect. \ref{sec:practical_sketching_alg}.   This is because 1) TYUC17 is easily implemented; 2) it is numerically stable  \cite{nakatsukasaFastStableRandomized2020}; 3) it has an accurate error bound that can guide the selection of sketch parameters \cite[Sect. 4.5]{Practical_Sketching_Algorithms_Tropp}. In the supplementary material, we also integrate SPI into the streaming algorithm of Tropp et al. \cite{troppStreamingLowRankMatrix2019}.   

The modified algorithm (TYUC17-SPI for short) is depicted in Algorithm \ref{alg:psa-sps} and its differences with     TYUC17   are highlighted in light blue.   

\begin{algorithm}\small
    \caption{TYUC17 Algorithm \cite{Practical_Sketching_Algorithms_Tropp} with Sketch-Power Iteration}\label{alg:psa-sps}
    \textcolor{lightgray}{\textbf{\% Sketching stage (support linear update):}}
    \begin{algorithmic}[1]
        \Require Data matrix $\qmat{A}\in\bbR^{m\times n}$, random  test matrices  $\bdOmega\in\mathbb{R}^{n\times s}$,      $\bdPsi\in\bbR^{d\times n}$,  \colorbox{lightblue}{$\bdPhi\in\bbR^{m\times l }$}, $l ,d ,s\ll \min\{m,n\}$, 
        $\min\{l ,d\}>s$ 
        \Ensure Sketches $\qmat{Y}\in\bbR^{m\times s},\qmat{W}\in\bbR^{d\times n},\qmat{Z}\in\bbR^{m\times l}$
        \State $\qmat{Y}=\qmat{A}\bdOmega $,   $\qmat{W}=\bdPsi\qmat{A} $, \colorbox{lightblue}{$\qmat{Z}=\qmat{A}\bdPhi $}
    \end{algorithmic}
    \textcolor{lightgray}{\% \textbf{QB  approximation stage:}}
    \begin{algorithmic}[1]
        \Require Test matrix $\bdPsi\in\bbR^{d\times n}$, sketches $\qmat{Y}\in\bbR^{m\times s}$,       $\qmat{W}\in\bbR^{d\times n}$, \colorbox{lightblue}{$\qmat{Z}\in\bbR^{m\times l }$}, \colorbox{lightblue}{and    parameter $q$}
        \Ensure Rank-$s$ approximation of   $\qmat{A}\approx\hat{\qmat{A}}=\qmat{Q}\qmat{B} $ with   $\qmat{Q}\in\bbR^{m\times s}$ and $\qmat{B}\in\bbR^{s\times n}$:
        \State 
         Set $\hat{\qmat{Y}}=\qmat{Y}$, \colorbox{lightblue}{repeat $q$ times}:
        \[
            \begin{tikzpicture}
            \node[fill=blue!14, anchor=base] {$[\qmat{X},\sim]=\texttt{QR}(\qmat{Z}^{T}\hat{\qmat{Y}},0),~\hat{\qmat{Y}}=\qmat{Z}\qmat{X}\quad\quad\quad\quad\quad \textcolor{lightgray}{\% ~\text{QR costs } O(ls^2) } $};
            \end{tikzpicture}
\]
        \State Compute $[\qmat{Q},\sim] = \texttt{QR}(\hat{\qmat{Y}},0)$
        \State Compute $\qmat{B}=\left(\bdPsi \qmat{Q}\right)^\dagger \qmat{W}$
    \end{algorithmic}
 
    \textcolor{lightgray}{\% \textbf{Truncation stage:}}
    \begin{algorithmic}[1]
        \Require $\qmat{Q}\in\bbR^{m\times s},\qmat{B}\in\bbR^{s\times n}$ from \textbf{QB stage},   target rank $r<s$
        \Ensure Rank-$r$ approximation triple $\qmat{U}\in\bbR^{m\times r}$, $\qmat{V}\in\bbR^{n\times r}$, and   diagonal   $\boldsymbol{\Sigma}\in\bbR^{r\times r}$,  such that $\hat{\qmat{A}}=\qmat{U}\boldsymbol{\Sigma}\qmat{V}^{T}\approx \qmat{A}$
        \State $\tilde{\qmat{U}} \boldsymbol{\Sigma} \qmat{V}^{T}=\rankrTrun{\qmat{B}}=\texttt{SVD}(\qmat{B},r)$
        \State $\qmat{U}=\qmat{Q}\tilde{\qmat{U}}$
    \end{algorithmic}
\end{algorithm}

\color{black}

A naive implementation of TYUC17-SPI requires $\qmat{Y},\qmat{W},\qmat{Z}$ simultaneously, which is memory-unfriendly. Note that $\qmat{Z}$ no longer participates in the computation once $\hat{\qmat{Y}}$ is obtained. This   allows us to design strategies such that   TYUC17-SPI has a comparable storage cost as that of TYUC17. We discuss it in Sect. \ref{sec:same_storage_as_PSA}.


 \begin{remark}
      
Throughout this work, we present SPI in its most transparent form with two sketches 
$\qmat{Z},\qmat{Y}$, treating $\bdPhi,\bdOmega$ as  independent. This setting 
highlights the distinct roles of $\qmat{Y}$ as the initial rangefinder (as in TYUC17) and $\qmat{Z}$ 
as an amplifier. In particular, 
the independence assumption is needed for the proof of our first bound (Theorem \ref{thm:oblique_proj_error_q=1}), 
which is further used to derive a priori sketch sizes  under memory constraints (Appendix \ref{sec:sketch_size_guidance}). 
Nevertheless, SPI does not   require $\qmat{Z},\qmat{Y}$   be independent as well: 
for example, choosing $\bdOmega=\bdPhi\tilde\bdOmega$  leads to $\qmat{Y}=\qmat{Z}\tilde\bdOmega$, which can further reduce memory cost in practice (see Sect. \ref{sec:reduce_storage_cost_q_gt_1}).
 \end{remark}


\color{black}

\subsection{Limiting   storage cost of TYUC17-SPI} \label{sec:same_storage_as_PSA} \color{black}
We propose three strategies to reduce   the   storage cost  of TYUC17-SPI   to $dn + ml$,  such that it is    comparable to that of TYUC17.\footnote{Following \cite{troppStreamingLowRankMatrix2019}, we disregard       the storage cost of test matrices, as they can be structured, sparse, or   be generated using random seeds.} The first  uses mixed-precision, while the latter two are hardware-dependent. 
For the hardware-dependent approaches, we may consider the   storage cost in the sketching stage and the approximation  stage (including QB and truncation) separately, since in  many architectures,   the   memory of the former  stage  is less constrained. 
\color{black}

\subsubsection{Mixed-precision} \label{sec:mixed_precision} 
Mixed-precision techniques have shown promise in reducing both computational and storage costs in various algorithms.      It was recently introduced in randomized algorithms \cite{connolly2022randomized,carson2024single} to obtain the   sketches in lower-precision, and implement subsequent computations in higher-precision. 

  This motivates   implementing TYUC17-SPI in a mixed-precision fashion.    
  We first sketch $\qmat{W},\qmat{Z},\qmat{Y}$ in \emph{single-precision}, and then compute and store $\hat{\qmat{Y}}$ also in single-precision, overwriting the space of $\qmat{Y}$. {\color{black} Recent advances in mixed‐precision algorithms   effectively alleviate the cumulative propagation of rounding errors that typically occurs in single‐precision arithmetic \cite{higham2022mixed,higham2019new}}. 
   Next, before steps 2 and 3 of Algorithm \ref{alg:psa-sps}, we  can convert $\hat{\qmat{Y}}$ and $\qmat{W}$ to \emph{double-precision} by reusing the space previously held by $\qmat{Z}$ if its size equals those  of $\hat{\qmat{Y}}$ and $\qmat{W}$; see Algorithm \ref{alg:single2double} in the supplementary material. Higher precision  is   important for avoiding loss of orthonormality  in steps 2 and keeping  the accuracy of solving linear systems  in step 3. 
   All the subsequent computations are executed in double-precision as well.
  
   \color{black}
  We count the   storage cost in terms of double-precision words.  Storing the     single-precision   $\qmat{W},\qmat{Z},\qmat{Y}$ costs $(d\cdot n + m\cdot s + m\cdot l )/2$   ($1/2$ represents single-precision);   $\hat{\qmat{Y}}$ takes the space of $\qmat{Y}$. As the space of $\qmat{Z}$ can be reused for converting and storing $\qmat{W}$ and $\hat{\qmat{Y}}$ into double-precision, we have $ml=ms+dn$. Thus the storage cost is 
  \begin{align}\label{eq:parameter_l_mixed_precision}
    (d\cdot n + m\cdot s + m\cdot l )/2 =  d\cdot n + m\cdot s, 
  \end{align}
  the same as the storage cost of    $dn+sm$ of TYUC17 (in double-precision).
\color{black}

\color{black}The above single-double precision model allows for larger sketch sizes, trading an additional round-off error  for  better approximation accuracy, which is effective especially when the data has a flat spectrum; see Sect. \ref{sec:numer_exp} for example. \color{black}
The single/double precisions   can be replaced by any lower precision $\boldsymbol{\mu}_l$ and higher precision $\boldsymbol{\mu}_h$, respectively; e.g., we can set $\boldsymbol{\mu}_l$ as half-precision to allow larger $l$.




\subsubsection{Multi-level caches system} Consider a setup consisting of a Hard Drive, CPU (with RAM), and GPU (with VRAM). In many GPU-accelerated tasks, VRAM is the real bottleneck due to its limited capacity, whereas   RAM is typically larger\footnote{For instance, 256GB of RAM is relatively affordable, while a high-end NVIDIA GeForce RTX 5090 graphics card offers only 32GB of VRAM and is much more costly.}. The goal is to leverage   GPU for speed while limiting its VRAM usage.

Concretely, we     first compute the sketches $\qmat{Y},\qmat{Z},\qmat{W}$ in a   pass of $\qmat{A}$ using  CPU and store   them temporarily    in the more redundant RAM.  Next we   transfer $\qmat{Z}$ and $\qmat{Y}$ to   GPU to   compute $\hat{\qmat{Y}}$. Once finishing, we delete $\qmat{Z}$ and  transfer $\qmat{W}$ into VRAM,   occupying  the space previously held by $\qmat{Z}$ (assume that $d\cdot n = m\cdot l $).     Thus   the storage needed in VRAM remains  
$d\cdot n + m\cdot s$, i.e., that for $\qmat{W}$ and $\hat{\qmat{Y}}$, matching the storage cost of TYUC17 if also computed by GPU.

Although   TYUC17-SPI   requires extra $m\cdot l$ storage in RAM (to hold $\qmat{Z}$), the crucial point is that in the more constrained VRAM, its usage stays the same as TYUC17. Therefore, comparing VRAM usage between the two methods is fair, as both fit within the same hardware limits and maintain the same  storage cost in GPU.

\subsubsection{Distributed system}  Suppose we are    in a parallel or distributed system and   for simplicity, the system   has three computing nodes (labeled as $a_{\qmat{W}},b_{\qmat{Z}},c_{\qmat{Y}}$), each with independent RAM. First,  $\qmat{W},\qmat{Z},\qmat{Y}$ are parallelly computed in node $a_{\qmat{W}},b_{\qmat{Z}}$, and $c_{\qmat{Y}}$, using a pass of $\qmat{A}$. Then,  $\qmat{Y}$ is transferred from node $c_{\qmat{Y}}$ to node $b_{\qmat{Z}}$ to compute $\hat{\qmat{Y}}$, and finally, $\hat{\qmat{Y}}$ is transferred from node $b_{\qmat{Z}}$ to $a_{\qmat{W}}$ to execute the rest steps. Thus the   storage cost in node $a_{\qmat{W}}$ is still
$d\cdot n + m\cdot s$,    
 and the transferring cost is only $2ms$ ($\qmat{Y}$ from   $c_{\qmat{Y}}$ to $b_{\qmat{Z}}$ and $\hat{\qmat{Y}}$ from   $b_{\qmat{Z}}$ to $a_{\qmat{W}}$).

\subsection{Sketch size guidance}\label{sec:sketch_size}\color{black}
\color{black}

It is necessary to a priori choose sketch sizes for one-pass algorithms.
 We only consider TYUC17-SPI in the mixed-precision setting. 
In the sequel,  we express storage in terms of double-precision words.
For the storage of the three single-precision sketches sizing $(m(l+s)+nd)/2$, it can be parameterized    as $T=(c(l+s)+d)/2$  (suppose $m=c n$ for  some $c>0$).  \color{black}Determining $l$ is easy: the discussion in Sect. \ref{sec:mixed_precision} shows that $\qmat{Z}$ takes half of the storage, i.e., $l=T\cdot n/m = T/c=s+d/c$.  For $s$, we present the $c=1$ case in Table \ref{tab:sketch_size_guidance}, and put the $c\neq 1$ case in Sect. \ref{SMsec:para_guid_m_neq_n} in the supplementary material. 

\setlength{\cmidrulewidth}{0.5pt}
\begin{table}[htbp] 
    \centering
    \caption{Sketch size $s$ selection in different spectrum decay;   $r\leq s\leq T/2$}
      \begin{mytabular2}{c!{\vline width 1pt}c|c|c!{\vline width 1pt}c|c}
        \toprule
        \multicolumn{1}{c}{{Flat}} & \multicolumn{3}{c}{Poly ($\sigma_i=i^{-\alpha}$)} & \multicolumn{2}{c}{Exp ($\sigma_i = e^{-\alpha\cdot i}$)} \\
        \cmidrule(lr){1-1} \cmidrule(lr){2-4} \cmidrule(lr){5-6} 
    &  {$\alpha<\frac{1}{2}$} & {$\alpha\approx\frac{1}{2}$} & {$\alpha>\frac{1}{2}$} & {$\alpha<\frac{1}{2T}$} & {$\alpha\geq \frac{1}{2T}$}     \\
      \hline
      $r$ & $r$ &   ${\rm P}_{[r,T/2]}(-\frac{T+1 }{2W(- {(T+1) }/{2ne})}-1)$ & $\max\{r, \frac{(2\alpha-1)({T +3})-2}{4\alpha} \}$ & $r$ & $  \frac{T}{2}$\\
      \bottomrule
      \end{mytabular2}%
      \label{tab:sketch_size_guidance}%
  \end{table}%

In the table,  ``Flat'' means flat spectrum, which occurs for example when $\qmat{A}$ is of low-rank plus noise; ``Poly'' means polynomial spectrum decay, i.e., singular values $\sigma_i=i^{-\alpha}$ for $\alpha>0$; ``Exp'' means exponential   decay, i.e., $\sigma_i = e^{-\alpha\cdot i}$ for $\alpha>0$;   ${\rm P}_{[a,b]}(\cdot)$ means   projection  onto   $[a,b]$.   $W(\cdot)$ refers to the Lambert W function, i.e., $W(a)$ is the solution of $xe^x = a$. In Matlab, one needs to use the command $\texttt{lambertw}(-1,a)$ to compute the value for a negative $a$. Detailed derivations  will be provided in   Appendix \ref{sec:sketch_size_guidance} by minimizing the error bound with respect to the sketch sizes.

\subsection{A variant for further reducing the storage cost   of TYUC17-SPI}\label{sec:reduce_storage_cost_q_gt_1} \color{black}
We can obtain a variant of SPI without sketching $\qmat{Y}$   to further save storage.    This can be achieved by using two test matrices $\bdPhi\in\bbR^{n\times l}$ and $\bdtildeOmega\in\bbR^{l\times s}$; then set $\qmat{Z}=\qmat{A}\bdPhi$ as that in SPI, while   compute $\hat{\qmat{Y}} = (\qmat{Z}\qmat{Z}^{T})^q\qmat{Z}\bdtildeOmega \in\bbR^{m\times s}$ instead. 

If denoting    $\bdOmega=\bdPhi\bdtildeOmega$ and $\qmat{Y}=\qmat{A}\bdOmega$, then $\qmat{Z}\bdtildeOmega=\qmat{A}\bdPhi\bdtildeOmega=\qmat{A}\bdOmega=\qmat{Y}$, and so   $\hat{\qmat{Y}}$ above can still be written as $\hat{\qmat{Y}}= (\qmat{Z}\qmat{Z}^{T})^q\qmat{Y}$, taking the same form as SPI. Thus this variant can still be regarded as  SPI, but without explicitly generating $\qmat{Y}$. 

We count the storage cost. Storing $\qmat{Z},\qmat{W}$ in single-precision requires $(ml+dn)/2$ storage. To compute $\hat{\qmat{Y}}$, which can be expressed as $\hat{\qmat{Y}}=\qmat{Z}(\qmat{Z}^{T}\qmat{Z})^q\bdtildeOmega$, we allocate an extra $l^2$ storage to form $\qmat{Z}^{T}\qmat{Z}$ and reuse this space to compute   $(\qmat{Z}^{T}\qmat{Z})^q\bdtildeOmega$.  Subsequently, we compute  $\hat{\qmat{Y}}\in\bbR^{m\times s}$ by overwriting the first $s$ columns of $\qmat{Z}$, using extra $s$ storage for buffering. Thereafter, $\hat{\qmat{Y}}$ is converted to double-precision (using the space of $\qmat{Z}$) for maintaining the orthonormal accuracy of $\qmat{Q}$,  while $\qmat{W}$ remains in single-precision. When $s\leq l/2$,   the double-precision $\hat{\qmat{Y}}$ can replace the space previously occupied by $\qmat{Z}$, so that the overall storage cost is bounded by  $(ml+dn)/2 + l^2 + s$.

\section{Error Bounds} \label{sec:error_bound}  
This section states two error bounds for TYUC17-SPI: one tailored for   $q=1$ and one valid  for general $q\geq 1$. Proofs appear in Sect. \ref{sec:analysis} and \ref{sec:error_anal_q>1}, respectively.  We include both because the first facilitates a priori sketch parameter selection (Sect. \ref{sec:sketch_size_guidance}) and because their proofs rely on different techniques.  
\subsection{Error bound when $q=1$}\label{sec:error_bound_q=1_stated}
We begin with the $q=1$ case, giving an expected Frobenius norm bound.  The test matrices $\bdPhi,\bdOmega,\bdPsi$ are assumed independent and standard Gaussian, i.e., each entry is i.i.d.   $N(0,1)$. Recall that $l,s,d$ respectively are the sketch size for $\qmat{Z},\qmat{Y},\qmat{W}$. 
\begin{theorem}\label{thm:oblique_proj_error_q=1}
    Let $ \qmat{Q}\in\bbR^{m\times s},\qmat{B}\in\bbR^{s\times n}$ be generated by Algorithm \ref{alg:psa-sps} with $q=1$ and       $\bdPhi \in\bbR^{n\times l},\bdOmega \in\bbR^{n\times s},\bdPsi \in\bbR^{d\times m}$ be independent standard Gaussian. If $l>s\geq \varrho + 4$ and $d>s+1$, then we have 
    \begin{small}
    \begin{align*}
        \Expectation{\normF{\qmat{A}-\qmat{Q}\qmat{B}}^2\!|E_F}\!\leq \!\!\frac{d}{d-s-1}\!\!\left[\!(1+\epsilon\hat{\xi}^2)\tau_{\varrho +1}^2\!(\qmat{A})\!+\!\delta\hat{\xi}^2\frac{\sum_{j=\varrho +1}^{n}\sigma_j^6}{\sigma_{\varrho}^4}\!+\!\mu\hat{\xi}^2\!\tau_{\varrho +1}^2\!(\qmat{A})\!\frac{ \sum_{j=\varrho +1}^{n}\sigma_j^4}{\sigma_{\varrho}^4}\right],
    \end{align*}
\end{small}
where $   \epsilon=\frac{2\varrho   }{(l-\varrho -1)}$, 
\begin{small}
\begin{equation}\label{eq:parameters_eps_delta_mu_1}
     \begin{split}
    \delta =\frac{e^2(s+\varrho )\varrho (l-1)(l^2+2l)\cdot }{ (s-\varrho )^2(l-\varrho )(l-\varrho -1)(l-\varrho -3)} ,~
    \mu =\frac{ e^2(s+\varrho )\varrho (l-1)l\cdot  }{(s-\varrho )^2(l-\varrho )(l-\varrho -1)(l-\varrho -3)},
\end{split}
\end{equation}
\end{small}
   and 
  $  E_F=\left\{\bdOmega,\bdPhi:\normSpectral{\left(\qmat{I}+\bdPhi_1\bdPhi_2^{T}\boldsymbol{\Sigma}_2^2\bdOmega_2\bdOmega_1^\dagger\boldsymbol{\Sigma}_1^{-2}(\bdPhi_1\bdPhi_1^{T})^{-1}\right)^{-1}}<\xi \right\}$ for some given  $\xi  >1$, and $\hat{\xi}=\xi/\sqrt{\mathbb P(E_F)}\approx \xi $ (see the discussion below). 
\end{theorem}

The matrices involved in $E_F$   will be detailed in Sect. \ref{sec:deter_qb_err}.  $\varrho$  is a  parameter that varies between $r$ and $s$; one can minimize the bound with respect to $\varrho$ to get a tighter bound.  The following gives some discussions on the bound. 

\paragraph{On the   event $E_F$}  $E_F$ is introduced to  characterize the   event that the matrix
$\qmat{I}+\bdPhi_1\bdPhi_2^{T}\boldsymbol{\Sigma}_2^2\bdOmega_2\bdOmega_1^\dagger\boldsymbol{\Sigma}_1^{-2}(\bdPhi_1\bdPhi_1^{T})^{-1}$ stays close to the identity, with spectral norm of its inverse bounded by a moderate constant $\xi $. A relatively larger   SPI parameter $l$   reduces   its spectral norm and    increases the likelihood of $E_F$. Specifically, if 
   $l$ is relatively large such that the term $ (e_1e_4+e_2e_3+e_2e_4u)e_5^2\sigma_{\varrho}^{-2} < 1- \xi^{-1}$, then $\mathbb P(E_F)>1-p_f$,   where  $e_i$'s are defined in \eqref{eq:parametersE}, failure probability  $p_f = 5e^{-u^2/2}+2t^{-(l-\varrho )}+2t^{-(s-\varrho )}$, and $u>0,t>1$.      See   Theorem \ref{thm:prob_inverseTerm} and its discussions in  Sect. \ref{sec:prob_qb_err} for more details. The underlying reason is the favorable embedding properties of $\bdPhi_1$. $p_f$ can be small such that $\hat{\xi}\approx \xi $. 


\paragraph{Simplification under spectral gap} Assume   a  large gap between $\sigma_{\varrho}$ and $\sigma_{\varrho +1}$. Then the last two terms in the bound can be viewed as higher-order terms of $\tau^2_{\varrho+1}(\qmat{A})$. Meanwile, the coefficient of the first term can be simplified: $1+\epsilon \xi^2 \leq 2\xi^2\cdot \frac{l}{l-\varrho -1} = O(\frac{l}{l-\varrho-1})$.  Up to constants, 
the bound reduces to the following: 
\begin{small}
\begin{align}\label{eq:error_bound_q1_simplified_1}
\Expectation{\normFSquare{\qmat{A}-\qmat{Q}\qmat{B}  } \!\mid \! E_F  }  \leq  O\bigxiaokuohao{\!\frac{d}{d-s-1} }  \cdot{  O\bigxiaokuohao{\frac{l }{l-\varrho-1} }}\tau_{\varrho +1}^2(\qmat{A}) + o(\tau_{\varrho +1}^2(\qmat{A})).
\end{align}
\end{small}
 While   heuristic, this simplification provides effective sketch sizes guidance   (Sect.  \ref{sec:sketch_size_guidance}). 

The higher-order terms arise from one iteration of SPI. The coefficient $d/(d-s-1)$ matches that in the TYUC17 bound (see \eqref{eq:tyuc17_bound_f_norm}), due to both sketching the same corange $\qmat{W}=\bdPsi\qmat{A}$.
  We also see that a larger SPI parameter $l$   yields smaller bound (also observe from \eqref{eq:parameters_eps_delta_mu} that a larger $l$ does not increase $\delta,\mu$,  the coefficients of $o(\tau_{\varrho +1}^2(\qmat{A}))$). 
This aligns with the intuition that a larger $\qmat{Z}$ captures more information of $\qmat{A}$, and so does the new rangefinder $\hat{\qmat{Y}} = \qmat{Z}\qmat{Z}^T\qmat{Y}$ produces a better approximation.  

\paragraph{Comparison with TYUC17 bound} Comparing \eqref{eq:error_bound_q1_simplified_1} with   \cite[Thm. 4.3]{Practical_Sketching_Algorithms_Tropp}\footnote{There was a typo in the bound of \cite[Thm. 4.3]{Practical_Sketching_Algorithms_Tropp}: the factor $k/(l-k-1)$ should be $l/(l-k-1)$, as $1+k/(l-k-1)\leq l/(l-k-1)$. Note that $l,k$ in \cite{Practical_Sketching_Algorithms_Tropp} correspond to $d,s$ in our notation.}  
\begin{align}\label{eq:tyuc17_bound_f_norm}\mathbb E\|\qmat{A} - \qmat{Q}\qmat{B}\|_F^2 \leq \frac{d}{d-s-1}\frac{s}{s-\varrho-1}\tau_{\varrho+1}^2(\qmat{A}),\end{align}
 we see that SPI replaces the relatively small $s$ in ${s}/{(s-\varrho-1)}$ with a larger $l$, which   lowers the    TYUC17 bound when higher-order terms are negligible. This helps explain why TYUC17, when equipped with SPI, achieves better performance in practice.


\paragraph{On orthogonal projection} Bounding the orthogonal projection error $\normF{\qmat{A}-P_{\hat{\qmat{Y}}}\qmat{A}}$ is the key to establish the main bound, where $P_{\hat{\qmat{Y}}}$ is the projection onto the column space of $\hat{\qmat{Y}}$.  With $\normF{\qmat{A}-P_{\hat{\qmat{Y}}}\qmat{A}}$ obtained,  the relation between oblique and orthogonal projection error\footnote{``Orthogonal'' is emphasized to distinguish it from the oblique projection error $\normF{\qmat{A}-\qmat{Q}\qmat{B}}$. Unless stated otherwise, ``projection'' hereafter refers to orthogonal projection.}  follows by the argument of \cite{Practical_Sketching_Algorithms_Tropp}: conditioned on $\bdOmega,\bdPhi$,  
\begin{align*}
    \mathbb E_{\bdPsi}\normF{\qmat{A}-\qmat{Q}\qmat{B}}^2=\frac{d}{d-s-1}\normF{\qmat{A}-P_{\hat{\qmat{Y}}}\qmat{A}}^2.
\end{align*}
The orthogonal projection is the main difference between TYUC17 and TYUC17-SPI.

We   do not state the final   error   $\|\qmat{A}-\qmat{Q}\tilde{\qmat{U}}\boldsymbol{\Sigma}\qmat{V}^{T}\|_F$; it  can be  derived directly from      Theorem \ref{thm:oblique_proj_error_q=1} and \cite[Prop. 6.1]{Practical_Sketching_Algorithms_Tropp}, and is therefore omitted.

\subsection{Error bound for general $q\geq 1$}\label{sec:error_bound_q>=1_stated}
We now present a probabilistic bound in spectral norm.  Here
   $\bdPsi\in\bbR^{d\times n},\bdPhi\in\bbR^{m\times l}$ are   independent standard Gaussian, while $\bdOmega=\bdPhi\bdtildeOmega \in\bbR^{m\times s}$, where $\bdtildeOmega\in\bbR^{l\times s}$ is standard Gaussian independent of $\bdPsi,\bdPhi$.  This setting connects to   Sect. \ref{sec:reduce_storage_cost_q_gt_1}. 
\begin{theorem} \label{thm:error_oblique_bound_q_gt_1}
  Let $\qmat{Q},\qmat{B}$ be generated by Algorithm \ref{alg:psa-sps}, with $\bdPhi,\bdPsi$ being independent standard Gaussian, and $\bdOmega=\bdPhi\bdtildeOmega$ where $\bdtildeOmega$ is standard Gaussian independent of $\bdPhi,\bdPsi$. Assume that $d\geq s+4$, and for any natural numbers $\varrho,k$ with $\varrho<k\leq l-4$  and $\varrho\leq s-4$. Then   $\normSpectral{\qmat{A}-\qmat{Q}\qmat{B}}$ is upper bounded by  
    \begin{align*}
        &\left((\eta_1+u\cdot\eta_2)\eta_3+\eta_2\eta_4\right)\tau_{k+1}(\qmat{A})+\left((\eta_1+u\cdot\eta_2)(\eta_3+\eta_4)+u\eta_3\right)\sigma_{k+1}(\qmat{A})\\
        &+\gamma_1\gamma_2\gamma_3\left(\left(1+\sqrt{\frac{k}{l}}+\frac{\beta}{\sqrt{l}}\right)\sigma_{\varrho+1}\left(\qmat{A}\right)+\left(1+\frac{\beta}{\sqrt{l}}\right)\sigma_{k+1}\left(\qmat{A}\right)
        +\tau_{k+1}\left(\qmat{A}\right)/\sqrt{l}
        \right),
    \end{align*}
    where 
    \begin{gather*}
        \eta_1=1+t\sqrt{\frac{3s}{d-s+1}};~ \eta_2=t\frac{e\sqrt{d}}{d-s+1};\\
        \eta_3=\frac{e\sqrt{l}}{l-k+1}t;~
        \eta_4=1+t\sqrt{\frac{3k}{l-k+1}};\\
        \gamma_1=1+t\sqrt{\frac{3s}{d-s+1}}+t\frac{e\sqrt{dl}}{d-s+1}+ut\frac{e\sqrt{d}}{d-s+1};~
    \gamma_2=\frac{el}{l-k+1};\\
    \gamma_3=\left( 1 + t \cdot \sqrt{\frac{3\varrho}{s-\varrho + 1}}+t \cdot \frac{e\cdot \left(\sqrt{sl}+u\sqrt{s}\right)}{s-\varrho+1}\right)^{\frac{1}{2q+1}},
    \end{gather*}
    with failure  probability at most $p_e=3e^{-u^2/2}+2t^{-(d-s)}+2t^{-(l-k)}+2t^{-(s-\varrho)}+e^{-\beta^2/2}$, where $u,\beta> 0$, and $t\geq 1$. 
\end{theorem}

As $\varrho,k$ are relatively flexible, the bound can be improved by suitable choices.

\paragraph{Simplification under parameter choices} For simplicity, let $d=2s$, $l=2k$, which are reasonable. We also assume  $l=O(d)$ (c.f. Sect. \ref{sec:sketch_size}),   In this case, $\eta_1,\eta_4,\gamma_1,\gamma_2=O(1)$,  $\eta_2=O(1/\sqrt{d})$, and $\eta_3 = O(1/\sqrt{l})$. The bound then simplifies to
\begin{align*}
   & O\bigxiaokuohao{\frac{1}{\sqrt{d}} + \frac{1}{\sqrt{l}} } \tau_{k+1}(\qmat{A}) + O\bigxiaokuohao{1+ O(\frac{1}{\sqrt{d}})+ O(\frac{1}{\sqrt{l}})}\sigma_{k+1}(\qmat{A})\\
   & + O(1)\gamma_3\bigxiaokuohao{ \bigxiaokuohao{ O(1) + O\bigxiaokuohao{\frac{1}{\sqrt{l}}}  }\sigma_{\varrho+1}(\qmat{A}) + (1 + O(1/\sqrt{l}))\sigma_{k+1}(\qmat{A}) + \frac{\tau_{k+1}(\qmat{A})}{\sqrt{l}}  }\\
    \approx & \gamma_3 O\bigxiaokuohao{1+ \frac{1}{\sqrt{l}  }}\sigma_{\varrho+1}(\qmat{A}) \\
   & + \bigxiaokuohao{O(1)+ O(1)\gamma_3  }O\bigxiaokuohao{1+  {\frac{1}{\sqrt{l}}} }\sigma_{k+1}(\qmat{A}) + O\bigxiaokuohao{\frac{\gamma_3}{\sqrt{l}}  + \frac{1}{\sqrt{l}} }\tau_{k+1}(\qmat{A}).
\end{align*}

Two   observations follow. First,   $\gamma_3$, which depends on the number of SPI steps, decreases exponentially  
with $q$, thereby reducing the bound, acting a similar role as the standard power iterations \cite{FindingStructureHalko}. Second,    the bound decreases further as $l$  increases.

\paragraph{Tail energy improvement} The tail energy   is the price   for obtaining a spectral norm bound \cite{FindingStructureHalko,tropp2023RandomizedAlgorithms,gu2015subspace}. A key distinction of our result is that  the tail   does not start  from a small $\varrho+1$ \cite{FindingStructureHalko} nor $s+1$ \cite{tropp2023RandomizedAlgorithms,gu2015subspace}, but from a {considerably} larger index $k+1$ (as $k$ scales with $l$, it can be several times larger than $s$). Consequently,   $\tau_{k+1}(\qmat{A})$ may be far smaller than $\tau_{\varrho+1}(\qmat{A})$, especially when the spectrum decays slowly before the   first $k$ singular values but rapidly thereafter. 
A special case is   ${\rm rank}(\qmat{A})=k$, where $\sigma_{k+1}(\qmat{A})=\tau_{k+1}(\qmat{A})=0$, so that   for large   $q$,   $\normSpectralnoleftright{\qmat{A}-\qmat{Q}\qmat{B}} \lesssim O(1)\sigma_{\varrho+1}(\qmat{A})   $ precisely.   This illustrates the benefit of SPI: whereas a small rangefinder   $\qmat{Y}=\qmat{A}\bdOmega$ may fail to capture sufficient range information due to storage limits, SPI   drives the sketch size from $s$ to a considerably large $l$ to capture much more information  (with feasible storage),  powers it to enhance the rangefinder,  yielding improved results.





\paragraph{Comparison with TYUC17 bound} To   compare   Theorem \ref{thm:error_oblique_bound_q_gt_1} with the bound of  TYUC17, we also need a spectral norm bound for it (as \cite{Practical_Sketching_Algorithms_Tropp} only provides a Frobenius norm one in expectation).  
Similar to the $q=1$ case, the key is to convert   oblique  to  orthogonal projection errors. By     Lemma \ref{lem:GNError_deviation_bound}, with high probability,  
\begin{align*}
    \normSpectralnoleftright{\qmat{A}-\qmat{Q}\qmat{B}}&\leq \beta_1 \normFnoleftright{\qmat{A} - P_{ \hat{\qmat{Y}}}\qmat{A}  }    +\beta_2\normSpectralnoleftright{\qmat{A} - P_{ \hat{\qmat{Y}} }\qmat{A}  }.
\end{align*}
where $\beta_1=\frac{e\sqrt d t}{d-s+1},\beta_2=1+t\sqrt{\frac{3s}{d-s+1}}+u \frac{e\sqrt{d}t}{d-s+1}$.  
For TYUC17, similarly, 
\begin{align*}
    \normSpectralnoleftright{\qmat{A}-\qmat{Q}\qmat{B}}&\leq \beta_1\normFnoleftright{\qmat{A} - P_{  \qmat{A}\bdOmega }\qmat{A}  } + \beta_2\normSpectralnoleftright{\qmat{A} - P_{  \qmat{A}\bdOmega }\qmat{A}  }.
\end{align*}
Since both decompositions are analogous, it suffices to compare their orthogonal projection errors. For clarity, we only consider the spectral norm case. Theorem \ref{thm:prob_projection_error_q_gt_1} shows that with high probability, 
\begin{small}
\begin{align*}
     \normSpectral{\qmat{A}-P_{ \hat{\qmat{Y}}}\qmat{A}} 
    &\leq 2 \left( 1 + t \cdot \sqrt{\frac{3k}{l-k + 1}}+  \frac{ut\cdot e \sqrt{l}}{l-k+1} \right) \sigma_{k+1}\left(\qmat{A}\right) +  \frac{2t \cdot e \sqrt{l}}{l-k+1} \tau_{k+1}\left(\qmat{A}\right)\\
        +&
        \frac{t e \sqrt{l}}{l-k+1}\cdot \gamma_3\cdot\left(\sigma_{\varrho+1}\left(\qmat{A}\right) (\sqrt{l}+\sqrt{k}+\beta )+\sigma_{k+1} (\qmat{A} ) (\sqrt{l}+\beta )+\tau_{k+1}\left(\qmat{A}\right)\right).
\end{align*}
\end{small}
On the other hand, \cite{FindingStructureHalko} shows that with high probability, 
\begin{small}
\begin{align*}
    \bignorm{\qmat{A}- P_{\qmat{A}\bdOmega}\qmat{A}}\leq  \left(\! 1 + t \cdot \sqrt{\frac{3\varrho}{s-\varrho + 1}}+ \frac{ut \cdot e \sqrt{s}}{s-\varrho+1}  \right) \sigma_{\varrho+1}\left(\qmat{A}\right) + t \cdot \frac{e \sqrt{s}}{s-\varrho+1}\tau_{\varrho+1}\left(\qmat{A}\right). 
\end{align*}
\end{small}
The first-line bound for $\normSpectral{\qmat{A}-P_{ \hat{\qmat{Y}}}\qmat{A}}$ essentially reflects the  error of $\|\qmat{A}-P_{{\qmat{Z}}}\qmat{A}\|$. Comparing with  $\bignorm{\qmat{A}- P_{\qmat{A}\bdOmega}\qmat{A}}$, this term is typically far smaller,   especially when $\tau_{k+1}(\qmat{A})$ is substantially smaller than $\tau_{\varrho+1}(\qmat{A})$. This improvement is accompanied by an additional term       $\|\qmat{Z}- P_{(\qmat{Z}\qmat{Z}^{T})^q\qmat{A}\bdOmega}\qmat{Z}\|$ that corresponds to the second line.  When  $l$ is considerably large, the  coefficient of $\sigma_{\varrho+1}(\qmat{A})$ is approximately a constant and moreover,  $\gamma_3$ is driven down to $1$ as $q$ increases. As a result, the sketch size $l$ and the number of SPI iterations $q$ are the primary mechanisms that help improve  the bound.






\color{black}

\section{Numerical Experiments} \label{sec:numer_exp}

This section presents numerical experiments to validate the effectiveness of  SPI. The following are observed: 1) TYUC17-SPI   in mixed-precision can improve the   accuracy of TYUC17 with the same storage budget\footnote{\color{black}TYUC17-SPI with the other two strategies in Sect. \ref{sec:same_storage_as_PSA} can clearly improve the accuracy, though it is hardware-dependence. To better illustrate the broad applicability of SPI, we  thus employ the mixed-precision strategy here that runs on any computer.}; 2) The sketch sizes guidance  in Sect. \ref{sec:sketch_size} is reliable; 3) Only a few   iterations  of SPI are generally sufficient in our tests.  
 All the experiments were conducted on a computer with an Intel Core i9-12700 CPU and 64GB of RAM. The Matlab version is 2023b.  

\subsection{Experimental setup}\label{sec:exp_setup}
\subsubsection{Synthetic data}  \label{sec:synthetic_data_def}   Data matrices similar to those in \cite{Practical_Sketching_Algorithms_Tropp,troppStreamingLowRankMatrix2019} are used.
\begin{enumerate}
    \item   Low-rank $+$ noise: The  matrices are generated by $\qmat{A}=\qmat{U}\boldsymbol{\Sigma}\qmat{V}^{T} +  {\frac{\gamma R}{n^2}}\qmat{E}$, where $\qmat{U},\qmat{V}$ are randomly orthonormal, \begin{small}$\boldsymbol{\Sigma}=\begin{bmatrix}
    I_R & 0\\
    0 & 0
    \end{bmatrix}
    $\end{small} is   diagonal, $\gamma$ is signal-to-noise ratio (SNR), and $\qmat{E}$ is a Gaussian noise matrix. We consider the case where $R=10$ and $m=n=1000$ and

        (a) LowNoise: $\gamma=10^{-4}$;  
        (b) MediumNoise: $\gamma=0.01$; 
        (c) HighNoise: $\gamma=0.1$.\\
    \item Polynomially (  power law\color{black}) decaying  spectrum: The  matrices are generated by $\qmat{A}=\qmat{U}\boldsymbol{\Sigma}\qmat{V}^{T}$, where $\boldsymbol{\Sigma}=\diag \left(1,\ldots,1,2^{-\alpha},\ldots,n^{-\alpha}\right)$ having $R$ ones, and $\qmat{U},\qmat{V}$ are randomly orthonormal.  $m=n=1000$, and the decay rate satisfies
     
         (a) SlowDecay: $\alpha=0.5$;
         (b) MediumDecay: $\alpha=1$;
         (c) FastDecay: $\alpha=2$.\\
    \item Exponentially decaying spectrum: The  matrices are generated by $\qmat{A}=\qmat{U}\boldsymbol{\Sigma}\qmat{V}^{T}$, where $\qmat{U},\qmat{V}$ are randomly orthonormal, $\boldsymbol{\Sigma}=\diag \left(1,\ldots,1,e^{-\alpha},\ldots,e^{-n\alpha}\right)$ which has $R$ ones, $m=n=1000$, and the decay rate satisfies
   
         (a) SlowDecay: $\alpha=0.01$; 
         (b) MediumDecay: $\alpha=0.1$; 
        (c) FastDecay: $\alpha=0.5$. 
\end{enumerate} 
\subsubsection{Real data}\label{sec:real_data}
Our real data contains the following two datasets: \color{black}
\begin{enumerate}
    \item NIST Special Database 19: 
    the dataset is a key resource for handprinted document and character recognition research. It contains handprinted sample forms from 3600 writers, over 810000 segmented handprinted character images covering digits, uppercase and lowercase letters, and ground truth classifications for each image. This large-scale, diverse dataset is invaluable for training robust recognition models and advancing related AI applications. 
    
    We have obtained a $16384\times 20000$ matrix from a subset  containing all $20000$ images of the digit $0$ from different writers, each being a $128\times 128$ image. Each image is reshaped into a column vector and   stacked to form the matrix. 

    \item Climate dataset: The climate datasets are obtained from the NCEP/NCAR Reanalysis 1 project. The NCEP-NCAR Reanalysis 1 (R1) is a comprehensive atmospheric dataset 
    providing historical atmospheric data from 1948 to the present. It is generated using a global weather model that combines observational data (from weather stations, satellites, etc.) with numerical weather prediction models to produce a consistent and high-resolution record of atmospheric variables (temperature, pressure, winds, etc.) at multiple pressure levels. The dataset has a spatial resolution of 2.5° latitude by 2.5° longitude and a temporal resolution of 6 hours, making it invaluable for studying both short-term weather patterns and long-term climate trends. 
    
    We have obtained a $10512\times 28144$ matrix from the dataset, each column being a vector reshaped by the pressure data of the $73\times 144$ meshgrid on the  surface. All daily data from $1948$ to $2025$ forms the matrix. 


\end{enumerate}
\subsubsection{Relative Error}
The relative error metric follows that in \cite{Practical_Sketching_Algorithms_Tropp,troppStreamingLowRankMatrix2019}:  
\begin{small}
\begin{align}\label{eq:relative_error_def}
     \text{Relative error:} ~S_F=  \frac{ \|\qmat{A}-\hat{\qmat{A}} \|_F}{\normF{\qmat{A}-\rankrTrun{A}}}-1    ~{\rm and}~ S_\infty = \frac{ \|\qmat{A}-\hat{\qmat{A}} \|}{ \normSpectralnoleftright{\qmat{A}-\rankrTrun{A}}}-1,
\end{align} 
\end{small}
where   $\qmat{A}$ is the original data and $\hat{\qmat{A}}$ is a rank-$r$ approximation computed by algorithms.
\subsubsection{Oracle error} \color{black}
The sketch sizes     affect  the results; thus it is possible to report the best possible performance for a fair comparison. This asks to enumerate over all possible parameters and evaluate them to select the most appropriate  ones. The associated error is called the oracle error \cite{Practical_Sketching_Algorithms_Tropp,troppStreamingLowRankMatrix2019}. Due to computer limit, it is only reachable for small test data in our experiments, e.g.,   the synthetic data. \color{black}

\subsubsection{Test matrix} Although the parameter guidance is based on the  Gaussian  test matrix assumption, to save storage in practice,  we use   the sparse {Rademacher} test matrix generated as follows: first draw a Rademacher matrix and then randomly set $99\%$ of the entries to be zero. It is unclear whether this distribution enjoys the Johnson-Lindenstrauss (JL) property in theory; our use here is purely empirical, motivated by its implementation  and storage advantage. In practice,  we observe  that it empirically exhibits   JL-like behavior, and it yields   results similar to those obtained with standard test matrices; see Sect. \ref{SMsec:insensitivity_test_matrix} and \ref{SMsec:JL_property_SparseRademacher} in the supplementary material for numerical evidence.   
\color{black}

\subsection{Numerical results for synthetic data}\label{sec:performance_synethetic_data_and_real_data}
\subsubsection{Reconstruction error}\label{sec:performance_synethetic_data}
\begin{figure}[htp]
    \centering
    \captionsetup{font=footnotesize}
    \begin{subfigure}{0.3\textwidth}
        \centering
        \includegraphics[width=\linewidth]{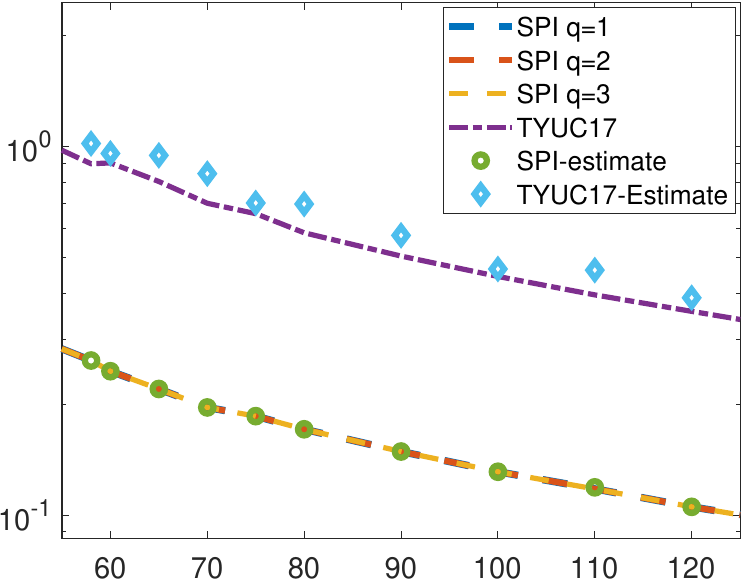}
        \caption{Low rank with low noise}
    \end{subfigure}%
    \begin{subfigure}{0.3\textwidth}
        \centering
        \includegraphics[width=\linewidth]{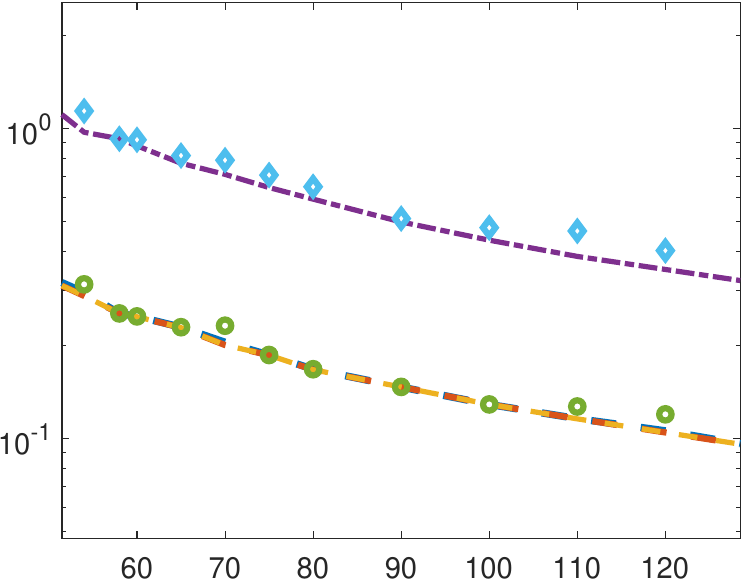}
        \caption{Low rank with medium noise}
    \end{subfigure}%
    \begin{subfigure}{0.3\textwidth}
        \centering
        \includegraphics[width=\linewidth]{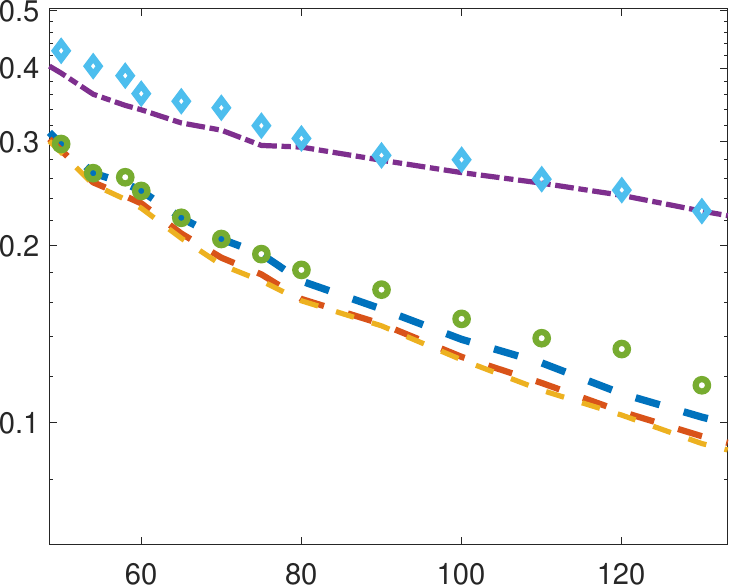}
        \caption{Low rank with high noise}
    \end{subfigure}

    \vspace{0.1cm} 

    \begin{subfigure}{0.3\textwidth}
        \centering
        \includegraphics[width=\linewidth]{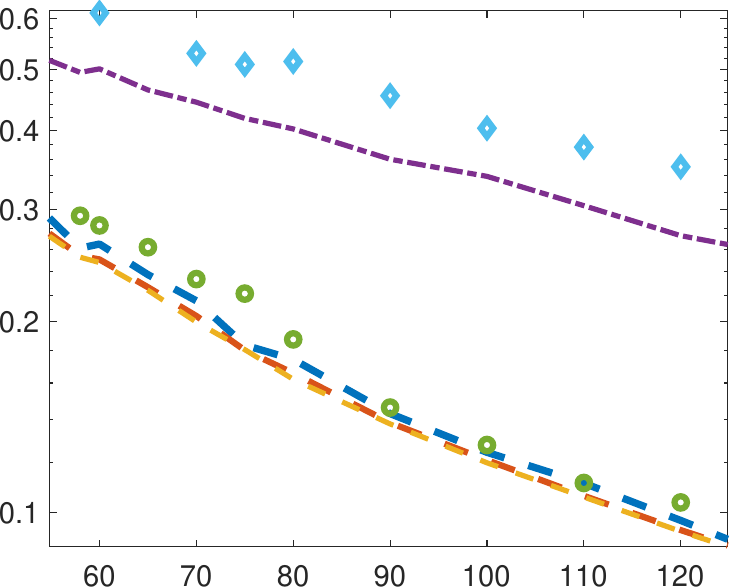}
        \caption{Slow polynomial decay}
    \end{subfigure}%
    \begin{subfigure}{0.3\textwidth}
        \centering
        \includegraphics[width=\linewidth]{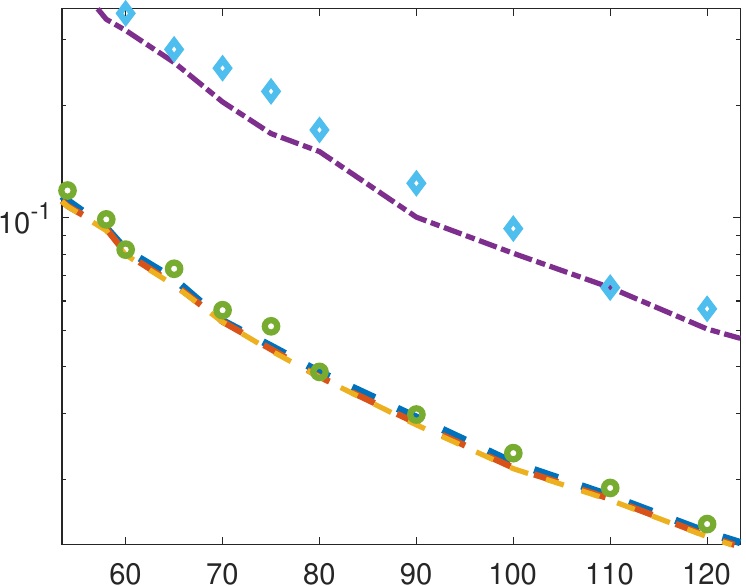}
        \caption{Medium polynomial decay}
    \end{subfigure}%
    \begin{subfigure}{0.3\textwidth}
        \centering
        \includegraphics[width=\linewidth]{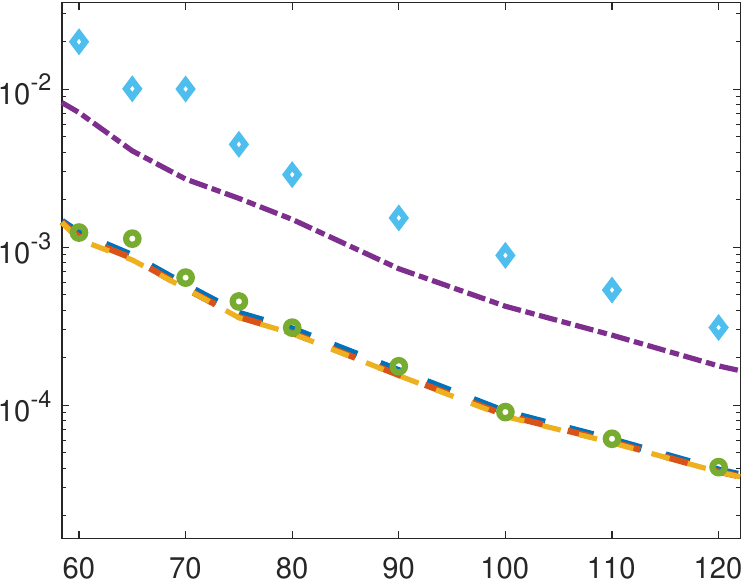}
        \caption{Fast polynomial decay}
    \end{subfigure}

    \vspace{0.1cm} 

    \begin{subfigure}{0.3\textwidth}
        \centering
        \includegraphics[width=\linewidth]{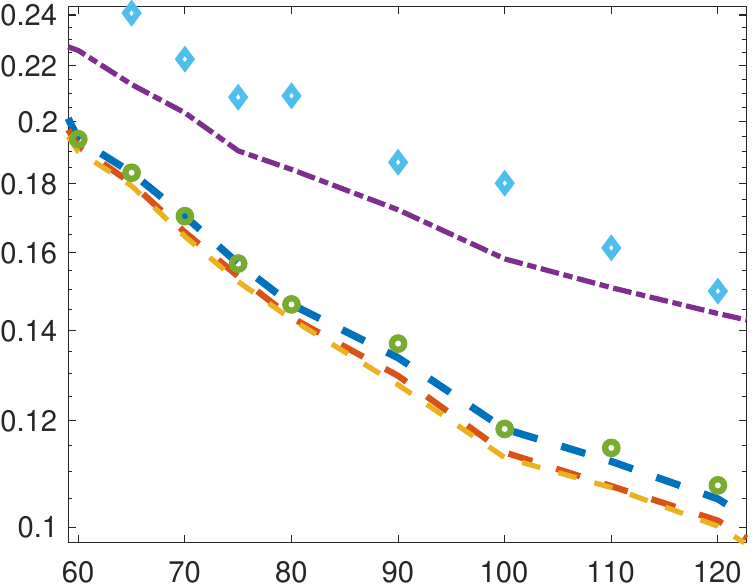}
        \caption{Slow exponential decay}
    \end{subfigure}%
    \begin{subfigure}{0.3\textwidth}
        \centering
        \includegraphics[width=\linewidth]{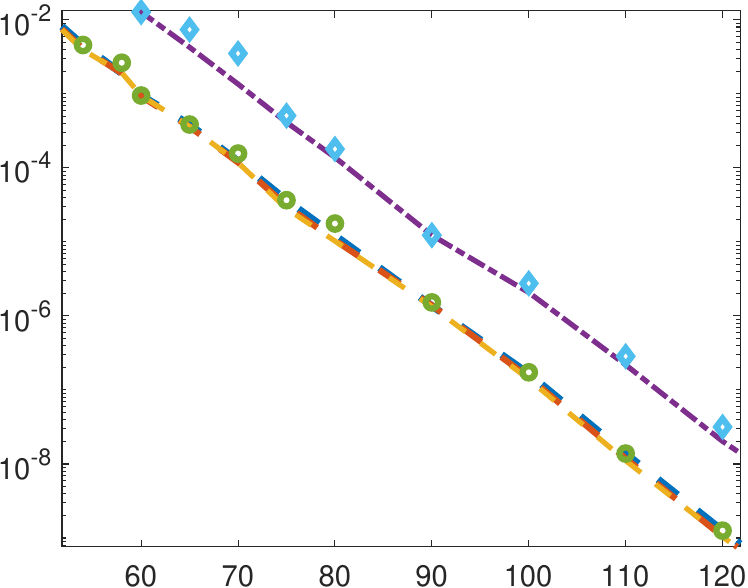}
        \caption{Medium exponential decay}
    \end{subfigure}%
    \begin{subfigure}{0.3\textwidth}
        \centering
        \includegraphics[width=\linewidth]{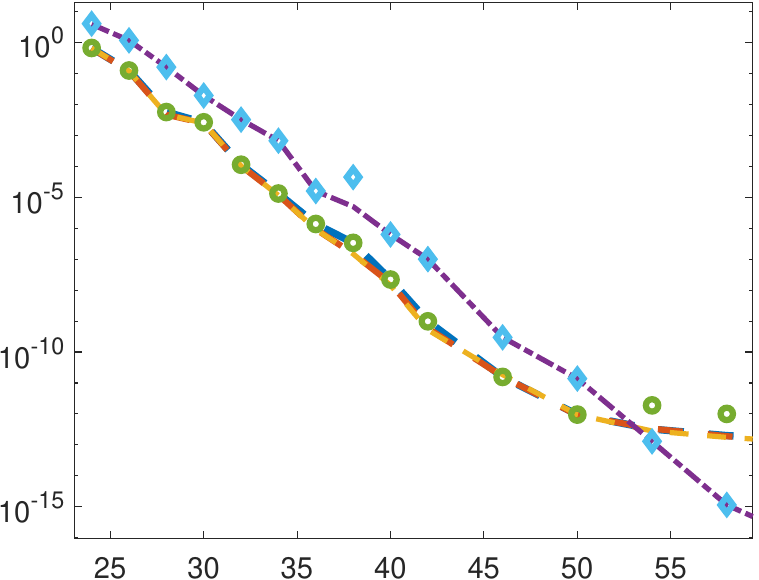}
        \caption{Fast exponential decay}
    \end{subfigure}

    \caption{Figures of synthetic data. $x$-axis means the    storage budget $\hat T n$ (parameterized by $\hat T$). $y$-axis means the relative Frobenius error $S_F$.  All \emph{{dashed} lines} represent oracle errors, while the \emph{markers} indicate errors calculated with parameter guidance. All results are averaged by 20 independent repeated experiments. }\label{fig:Synectic_data}
\end{figure}

\begin{figure}[htp]
    \centering
   \captionsetup{font=footnotesize}
       \begin{subfigure}{0.3\textwidth}
        \centering
        \includegraphics[width=\linewidth]{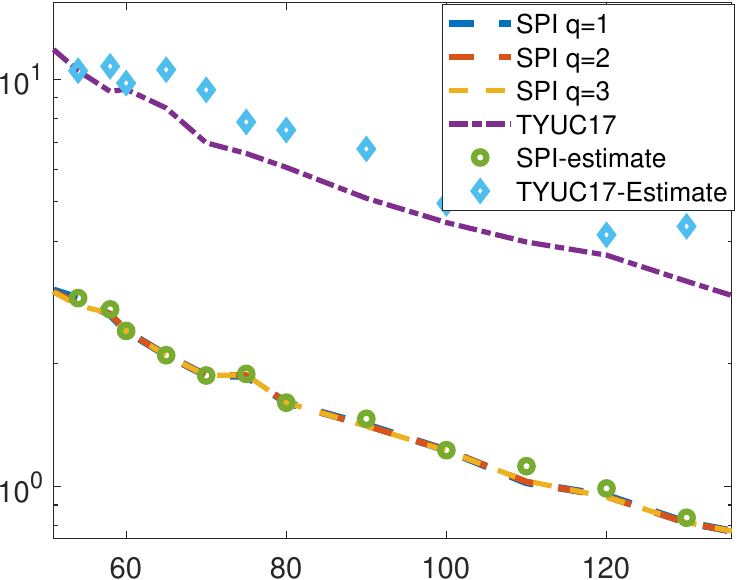}
        \caption{Low rank with low noise}
    \end{subfigure}%
    \begin{subfigure}{0.3\textwidth}
        \centering
        \includegraphics[width=\linewidth]{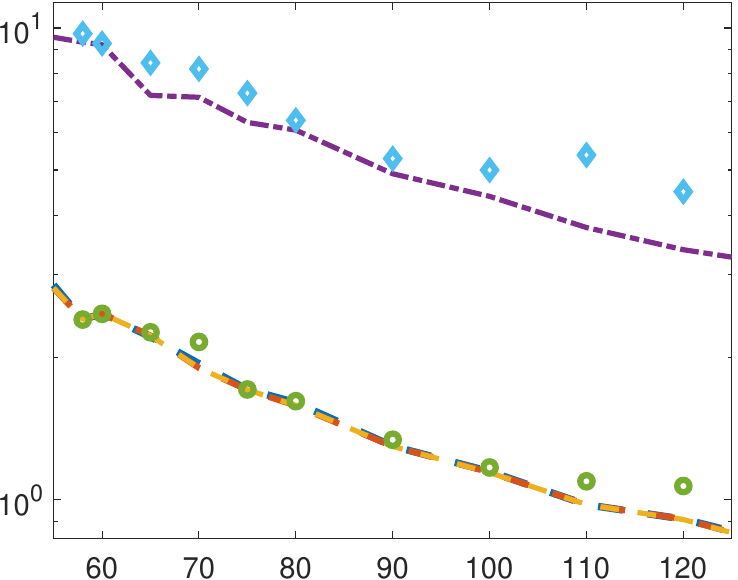}
        \caption{Low rank with medium noise}
    \end{subfigure}%
    \begin{subfigure}{0.3\textwidth}
        \centering
        \includegraphics[width=\linewidth]{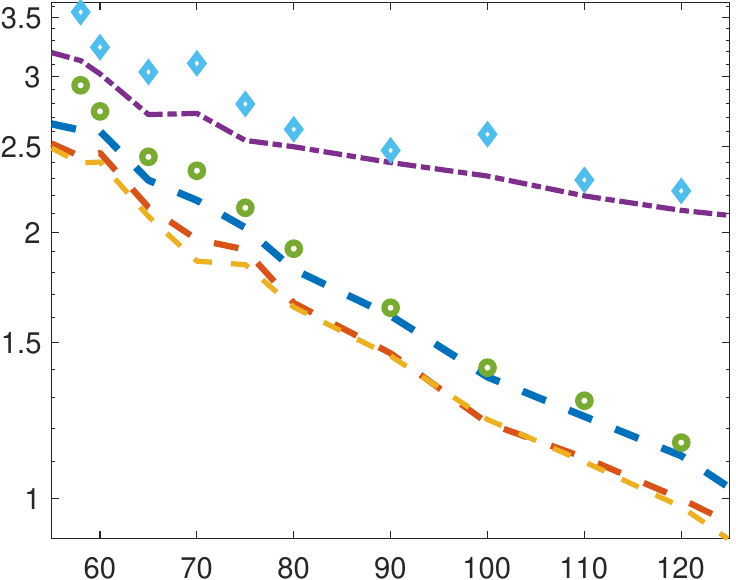}
        \caption{Low rank with high noise}
    \end{subfigure}

    \vspace{0.1cm} 

    \begin{subfigure}{0.3\textwidth}
        \centering
        \includegraphics[width=\linewidth]{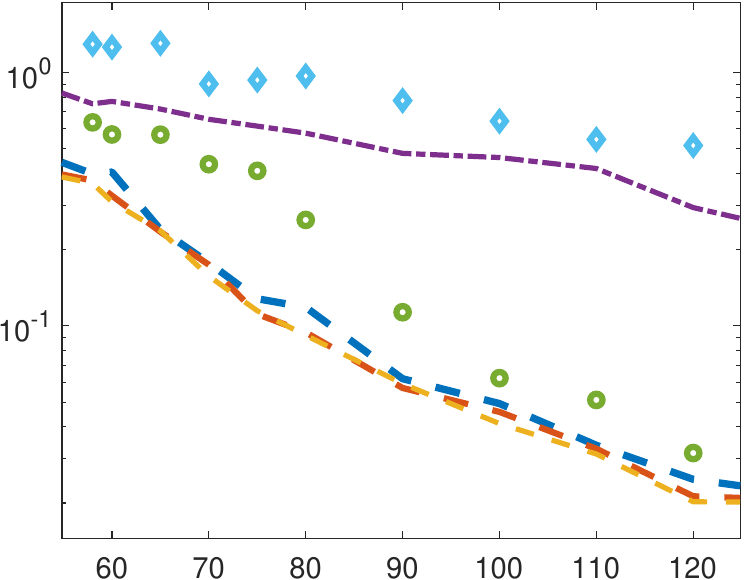}
        \caption{Slow polynomial decay}
    \end{subfigure}%
    \begin{subfigure}{0.3\textwidth}
        \centering
        \includegraphics[width=\linewidth]{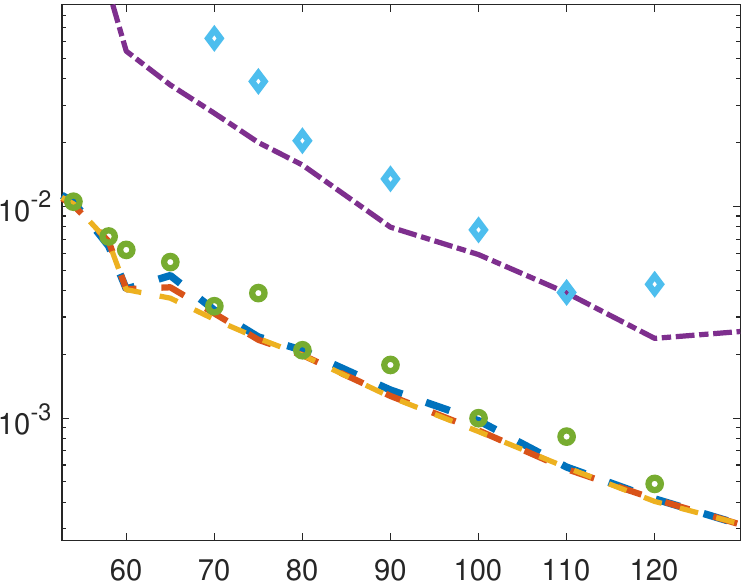}
        \caption{Medium polynomial decay}
    \end{subfigure}%
    \begin{subfigure}{0.3\textwidth}
        \centering
        \includegraphics[width=\linewidth]{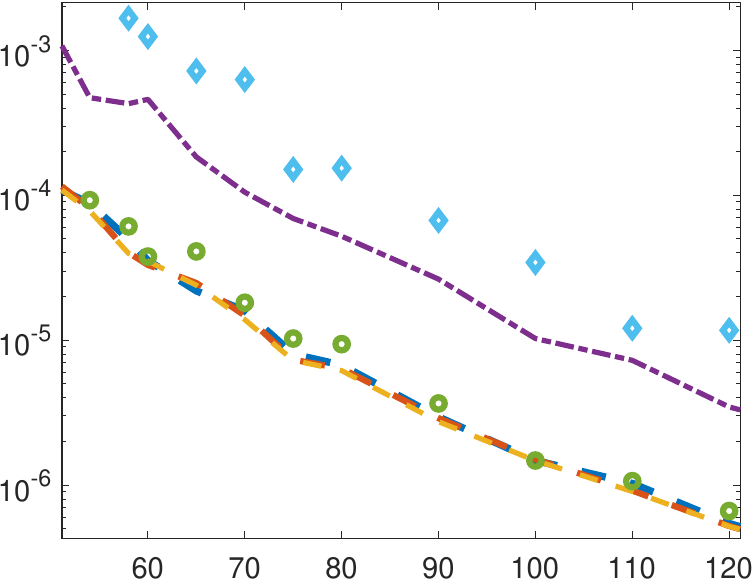}
        \caption{Fast polynomial decay}
    \end{subfigure}

    \vspace{0.1cm} 

    \begin{subfigure}{0.3\textwidth}
        \centering
        \includegraphics[width=\linewidth]{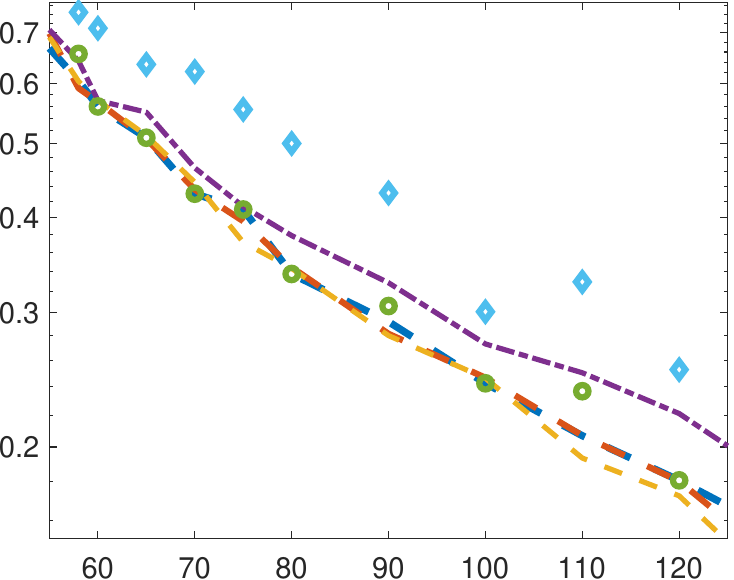}
        \caption{Slow exponential decay}
    \end{subfigure}%
    \begin{subfigure}{0.3\textwidth}
        \centering
        \includegraphics[width=\linewidth]{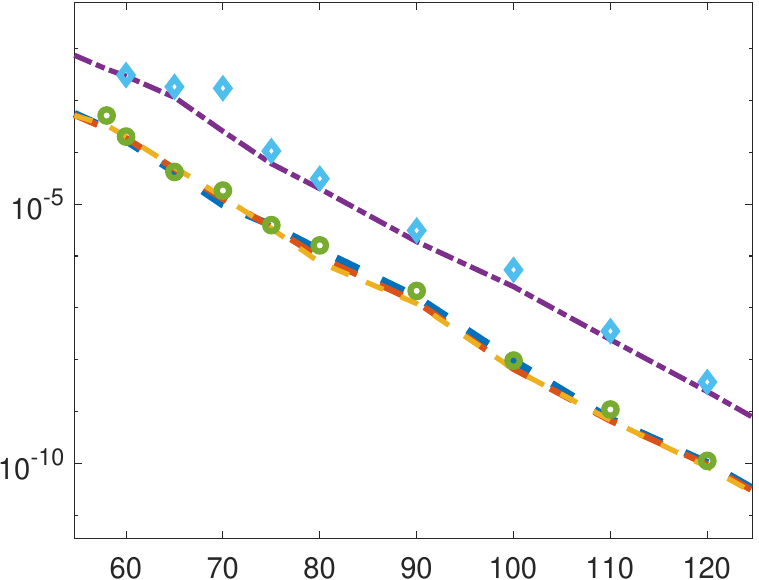}
        \caption{Medium exponential decay}
    \end{subfigure}%
    \begin{subfigure}{0.3\textwidth}
        \centering
        \includegraphics[width=\linewidth]{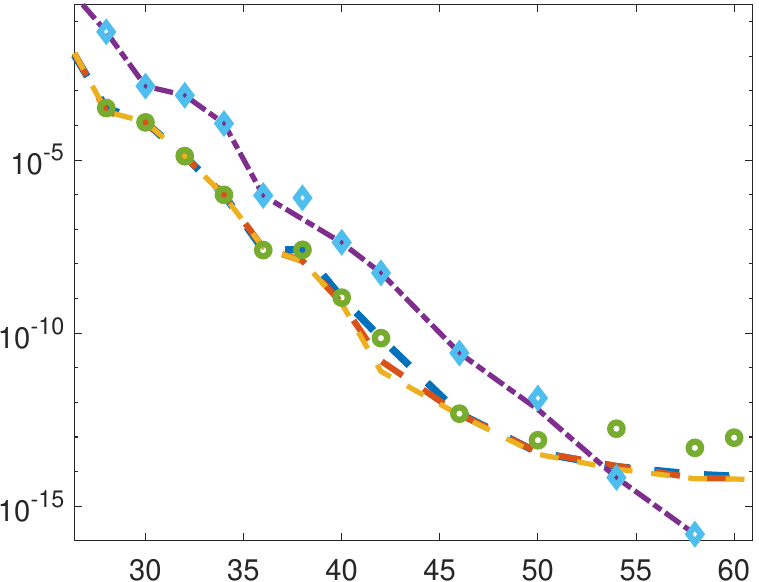}
        \caption{Fast exponential decay}
    \end{subfigure}

    \caption{Figures of synectic data. $y$-axis means the relative Spectral error $S_{\infty}$. }\label{fig:Synectic_data_spectral}
\end{figure}
\begin{example}\label{eg:synthetic_data}
    This example compares the performance between TYUC17-SPI (Algorithm \ref{alg:psa-sps}) and its   TYUC17 (\cite[Alg. 7]{Practical_Sketching_Algorithms_Tropp}) on synectic data with the same storage budget. All oracle errors   {both in  Frobenius norm and spectral norm} were computed by averaging $20$ independently repeated experiments (the errors are represented by lines)   in Fig. \ref{fig:Synectic_data} and \ref{fig:Synectic_data_spectral}, respectively. We also plot the relative errors of the algorithms, whose parameters were computed by theoretical guidance (see Sect. \ref{sec:sketch_size_guidance} for TYUC17-SPI and \cite[Sect. 4.5]{Practical_Sketching_Algorithms_Tropp} for TYUC17) to show that the guided parameters  are reliable (the errors are represented by markers). In the figure, $x$-axis means the    storage budget $\hat T n$ (parameterized by $\hat T$) and $y$-axis means the relative   errors.    \emph{{Dashed} lines} represent oracle errors, while the \emph{markers} indicate errors calculated with parameter guidance.
\end{example}


Figures \ref{fig:Synectic_data} and \ref{fig:Synectic_data_spectral} show three points:
\begin{itemize}
\item In most cases, the   errors   of TYUC17-SPI  are   smaller than those of TYUC17,   demonstrating the effectiveness of SPI with the same storage budget. The improvement is more significant when the  data matrix has a slow to moderate spectrum decay, such as the low rank matrix plus   noise and   the polynomial decay spectrum cases. {\color{black} The only exception is that  when the spectrum exhibits a fast exponential decay and a larger storage budget is available (in the bottom-right subfigure), TYUC17-SPI is outperformed by TYUC17. This performance drop is due to TYUC17-SPI reaching the precision limit of its mixed-precision strategy, with round-off errors becoming dominant. Nevertheless, its accuracy remains acceptable for most applications,  and this exception does not affect the overall trend, which still favors TYUC17-SPI.} 
\color{black}

\item The Frobenius errors (indicated by   markers) in Fig. \ref{fig:Synectic_data} computed by TYUC17-SPI using   theoretically guided sketch sizes are nearly identical to the best performance achieved (oracle errors, shown by   lines).  {For spectral errors in Fig. \ref{fig:Synectic_data_spectral}, most of the errors are also well aligned with the oracle errors, except for the   slow polynomial decay case. 
Nevertheless,  the estimation becomes more accurate when the storage budget is increased.} This suggests that our parameter guidance is   reliable. Meanwhile, the errors    of TYUC17 using parameter  guidance are also close to its oracle errors. This demonstrates that we can set sketch sizes with theoretical guidance a priori to improve the performance of the algorithm without resorting to an exhaustive search.  

\item We also observe that in several cases, the   errors of TYUC17-SPI with $q=1$ are close to those with  more iterations   ($q=2,3$).  {However, when the data matrix has a slower decay,   (e.g, highly   noisy low-rank, slow polynomial or slow exponential; see the subfigures (c), (d), (g) in both Fig. \ref{fig:Synectic_data} and \ref{fig:Synectic_data_spectral}), larger $q$   improves the accuracy, especially in spectral norm. 
This indicates that a single iteration usually suffices when the spectrum decays rapidly, whereas more iterations are beneficial for slow decay.
 }

A possible   reason why only  a few sketch-power iterations suffice  is that the sketch $\qmat{Z}=\qmat{A}\bdPhi$ has a relatively faster decay spectrum, and especially faster than that of  the original data matrix, as visualized in Fig. \ref{fig:SingularValueSlow}. 
 This implies that the energy of the sketched matrix can be effectively captured by a few sketch-power iterations.
\end{itemize}

\begin{figure}[htp] 
    \centering
    \captionsetup{font=footnotesize}
    \begin{subfigure}{0.3\textwidth}
        \centering
        \includegraphics[width=\linewidth]{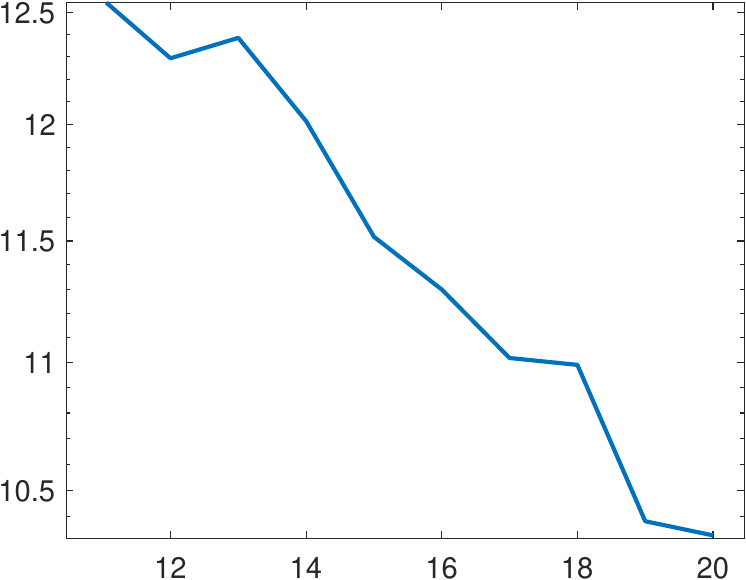}
        \caption{Low rank with high noise}
    \end{subfigure}%
    \begin{subfigure}{0.3\textwidth}
        \centering
        \includegraphics[width=\linewidth]{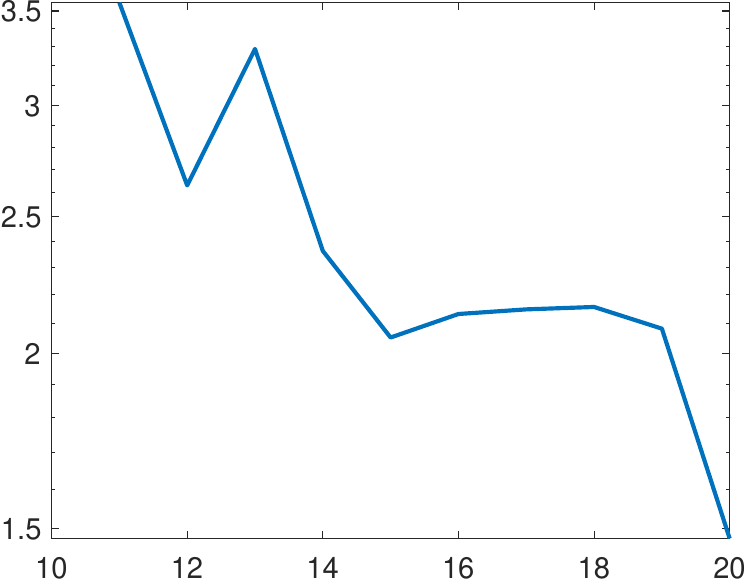}
        \caption{Fast polynomial decay}
    \end{subfigure}%
    \begin{subfigure}{0.3\textwidth}
        \centering
        \includegraphics[width=\linewidth]{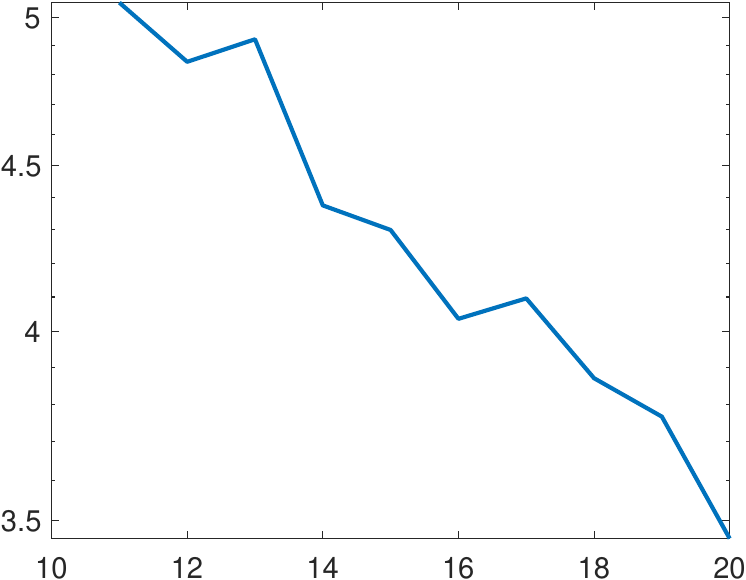}
        \caption{Slow exponential decay}
    \end{subfigure}
    \caption{Distortion ratio        $ \sigma_i(\qmat{A}\bdPhi)/\sigma_i(\qmat{A})$ for $r=10$ and $s=20$. The $x$-axis is the singular value index.}\label{fig:SingularValueSlow}
\end{figure}

\subsubsection{Effect of iterations in SPI} { 
For slow decay spectrum,   we observed that TYUC17-SPI with more iterations ($q=2,3$) has improved accuracy compared to $q=1$. In this section  we further investigate the long-term effect of increasing $q$, so as to understand the general trend and the potential benefits of additional iterations.}
\begin{example}
     { 
    We test  TYUC17-SPI with different $q$ on   slow polynomial data in Example \ref{eg:synthetic_data}, using $50$ independent runs for each $q$ with   $l=60$, $s=20$, $d=40$. Results are shown in Fig. \ref{fig:qFigure}, where the $x$-axis is $q$ and the $y$-axis means  the error $S_F$ (left) and  $S_{\infty}$ (right). }
\end{example}

\begin{figure}[htp] 
    \centering
    \captionsetup{font=footnotesize}
    \begin{subfigure}{0.45\textwidth}
        \centering
        \includegraphics[width=\linewidth]{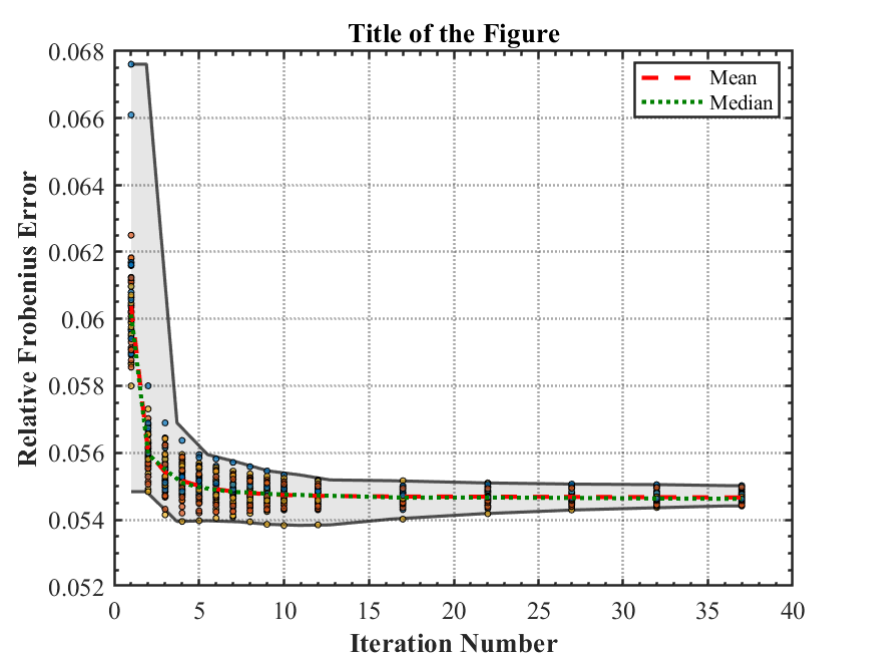}
        \caption{Frobenius error with different   $q$.}
    \end{subfigure}%
    \begin{subfigure}{0.45\textwidth}
        \centering
        \includegraphics[width=\linewidth]{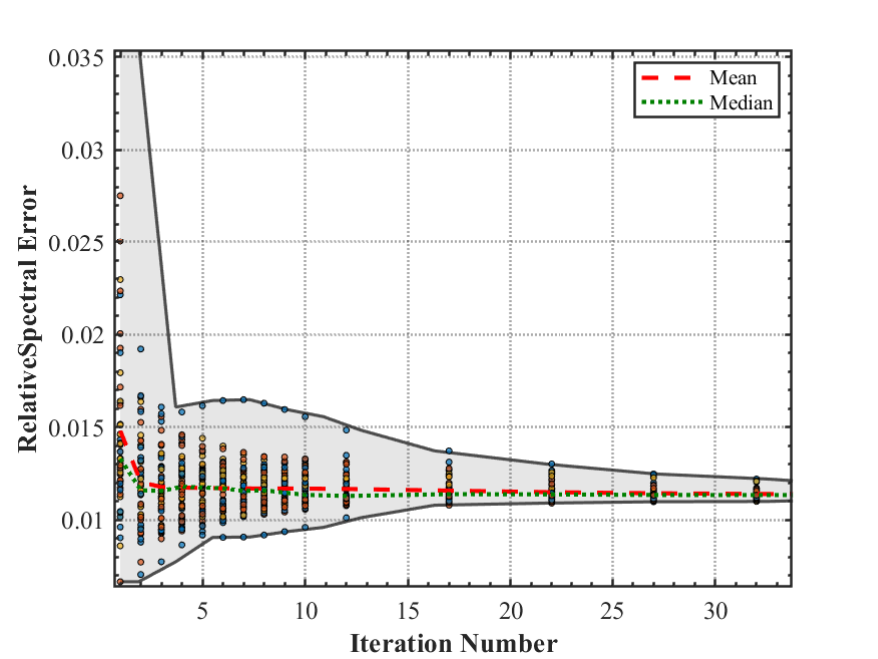}
        \caption{Spectral error with different   $q$.}
    \end{subfigure}%
    \caption{Error comparison of different  $q$. The scatter points represent the results of individual experiments and the shaded areas represent the regions between the minimum and maximum errors. The dashed lines represent the average result across all $50$ experiments. }\label{fig:qFigure}
\end{figure}

  Fig. \ref{fig:qFigure} shows the following   observations:
\begin{itemize}
    \item In the average sense, more iterations $q$ can bring better accuracy. Particularly, the spectral error $S_{\infty}$ has more significant improvement than the   Frobenius error $S_F$ as $q$ increases. 
    \item The stability of the algorithm is significantly improved with more iterations $q$, as the error range (the shaded area) becomes smaller. 
    \item From the density of the scatter points, we can see that most of the results are concentrated in a small range around the mean (the dashed line), especially for larger $q$ and the outliers are rare.
    \item By computing the coefficient of variation (CV, the ratio of the standard deviation to the mean) of the errors, we have two trends: 1) the CV decreases as $q$ increases, indicating that the results become more stable; 2) the CV of the spectral error is larger than that of the Frobenius error, indicating that the spectral error is more sensitive to randomness. 
\end{itemize}


\subsubsection{Wall-clock time comparisons}\label{sec:wall_clock_time_comparison} Comparing with  TYUC17, the additional cost of TYUC17-SPI arises primarily from the  SPI stage, which involves generating   $\bdPhi$, constructing   $\qmat{Z}=\qmat{A}\bdPhi$, and performing power iterations. When sparse random matrices are employed, the cost  of forming $\bdPhi$ is negligible and computing $\qmat{Z}$      scales with the number of nonzero entries in $\bdPhi$.

The   iteration phase consists of two main components. The first involves matrix multiplications of the form $\qmat{Z}^T\qmat{Y}$ and $\qmat{Z}\qmat{X}$, requiring $O(mls)$ flops. The second involves orthogonalization of $\qmat{Z}^T\qmat{Y}\in\bbR^{l\times s}$, which can be efficiently accomplished using a {QR} factorization in $O(ls^2)$ operations. Since $l,s\ll \min\{m,n\}$, this cost is negligible. Therefore, the dominant operations within the SPI procedure are the BLAS-3 matrix multiplications, and the incremental cost introduced by TYUC17-SPI remains modest. 


\begin{example}
    In this experiment, we compare the wall-clock time of TYUC17-SPI with   $q=1,2,3$ against   TYUC17, using the synthetic data  with medium polynomial decay as defined in Section~\ref{sec:synthetic_data_def}. The    dimensions are set to $m=20000$ and $n=30000$, with the parameters fixed at $l=1000$, $s=500$, and $d=1000$. The total  wall-clock time, along with that of each step, is reported in Table~\ref{SMtab:time_comparison}.
\end{example}
\begin{table}[htbp] 
\centering
\caption{Test additional time cost of TYUC17-SPI}\label{SMtab:time_comparison}
\resizebox{\columnwidth}{!}{
\begin{tabular}{cccccc}
\hline
\multirow{2}{*}{\textbf{Step}} & \multirow{2}{*}{\textbf{Time Complexity}} & \multirow{2}{*}{\textbf{TYUC17}} & \multicolumn{3}{c}{\textbf{TYUC17-SPI}} \\
 &  &  & \textbf{$q=1$} & \textbf{$q=2$} & \textbf{$q=3$} \\
\hline
$\bdPhi$ ($n\times l$) generation & $O(nnz)$ & -      & 0.0646 & 0.0646 & 0.0646 \\
$\bdOmega$ ($n\times s$) generation & $O(nnz)$ & 0.0645 & 0.0640 & 0.0640 & 0.0640 \\
$\qmat{Z}=\qmat{A}\bdPhi$ & $O(m\times nnz(\bdPhi))$ & -      & 0.1591 & 0.1591 & 0.1591 \\
$\qmat{Y}=\qmat{A}\bdOmega$ & $O(m\times nnz(\bdOmega))$ & 0.1373 & 0.1370 & 0.1370 & 0.1370 \\
\hdashline
$\qmat{Y}_2=\qmat{Z}^T\qmat{Y}\,(or~\qmat{Y}_3)$ & $O(mls)$ & -      & 0.0497 & 0.1116 & 0.1528 \\
$[\qmat{X},\sim]=\texttt{QR}(\qmat{Y}_2,0)$ & $O(ls^2)$ & -      & 0.0042 & 0.0112 & 0.0150 \\
$\qmat{Y}_3=\qmat{Z}\qmat{X}$ & $O(mls)$ & -      & 0.0652 & 0.1332 & 0.1998 \\
\hdashline
QR on $\qmat{Y}$ or $\qmat{Y}_3$ & $O(ms^2)$ & 0.1146 & 0.1190 & 0.1190 & 0.1190 \\
$\bdPsi$ ($d\times m$) generation & $O(nnz)$ & 0.0431 & 0.0429 & 0.0429 & 0.0429 \\
$\qmat{W}=\bdPsi  \qmat{A}$ & $O(n\times nnz(\bdPsi))$ & 0.1626 & 0.1613 & 0.1613 & 0.1613 \\
Solve $\qmat{B}=(\bdPsi\qmat{Q}) \backslash \qmat{W}$ & $O(nls)$ & 0.2419 & 0.2419 & 0.2419 & 0.2419 \\
$[\qmat{U},\qmat{\Sigma},\qmat{V}] = \texttt{trun\_svd}(\qmat{B},r)$ & $O(ns^2)$ & 0.2672 & 0.2705 & 0.2705 & 0.2705 \\
$\qmat{U}\leftarrow \qmat{Q}\qmat{U}$ & $O(msr)$ & 0.0260 & 0.0256 & 0.0256 & 0.0256 \\
\hline 
 {TOTAL} & - &  {1.0571} &  {1.4049} &  {1.5415} &  {1.6625} \\
\hline
\end{tabular}
}
\end{table}

Table \ref{SMtab:time_comparison} shows the wall-clock time of TYUC17 and TYUC17-SPI with   $q=1,2,3$. We can see that the additional cost   is from $32\%$ when $q=1$ to $56\%$ when $q=3$ compared to TYUC17, which mainly comes from   constructing   $\qmat{Z}$. The cost of one iteration  is smaller compared to the cost of constructing   $\qmat{Z}$. Particularly, orthogonalizing $\qmat{Z}^T\qmat{Y}$ is negligible. In summary, the computational overhead of TYUC17-SPI is small, which is consistent with our previous analysis.

\subsection{Numerical results for real data}\color{black} 
 {This section evaluates  different algorithms on   NIST   and climate datasets, which are commonly used for machine leraning and scientific computing,  respectively. The former exhibits a relatively smooth polynomial spectral decay, whereas for the latter, the leading five singular values drop by two orders of magnitude, with the rest decaying at a slow exponential rate.   The visualization of these datasets and their singular values is shown in Fig. \ref{fig:NISTImage} and   \ref{fig:worldmap}.}
\begin{figure}
  \centering
     \captionsetup{font=footnotesize}
  \begin{minipage}[t]{0.43\textwidth}\vspace{0pt}
    \centering

    \begin{subfigure}[t]{0.48\linewidth}
      \centering
      \includegraphics[width=0.95\linewidth,frame]{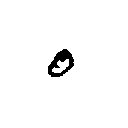} 

    \end{subfigure}\hfill
    \begin{subfigure}[t]{0.48\linewidth}
      \centering
      \includegraphics[width=0.95\linewidth,frame]{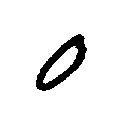}

    \end{subfigure}

    \vspace{0.8em}

    \begin{subfigure}[t]{0.48\linewidth}
      \centering
      \includegraphics[width=0.95\linewidth,frame]{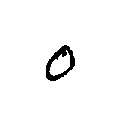}

    \end{subfigure}\hfill
    \begin{subfigure}[t]{0.48\linewidth}
      \centering
      \includegraphics[width=0.95\linewidth,frame]{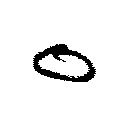}

    \end{subfigure}
  \end{minipage}
  \hfill
  \begin{minipage}[t]{0.50\textwidth}\vspace{0pt}
    \centering
    \begin{subfigure}[t]{\linewidth}
      \centering
      \includegraphics[width=\linewidth]{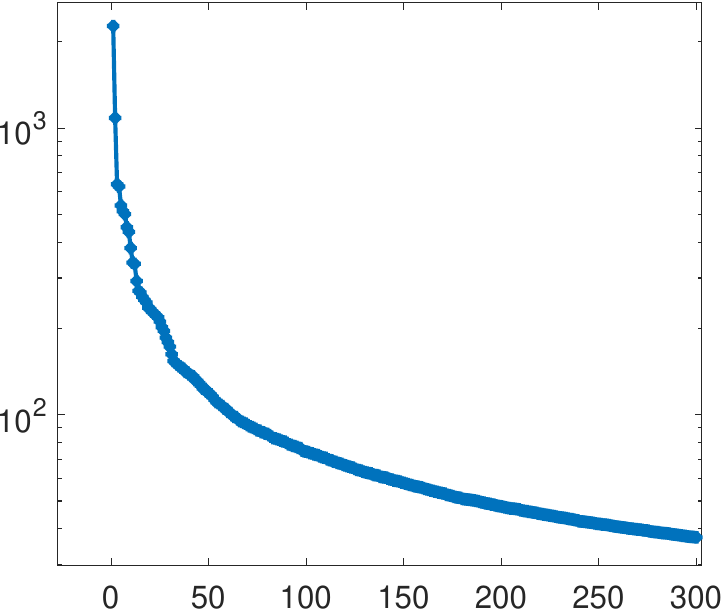}

    \end{subfigure}
  \end{minipage}

  \caption{Left images are four samples of the digit ``0'' from different writers in NIST dataset. Right figure shows the singular values of the  matrix generated in Sect. \ref{sec:real_data}. }
  \label{fig:NISTImage}
\end{figure}

\begin{figure}[t]
  \centering
     \captionsetup{font=footnotesize}
  \begin{subfigure}[c]{0.49\textwidth}
    \centering
    \adjincludegraphics[width=\linewidth,valign=m]{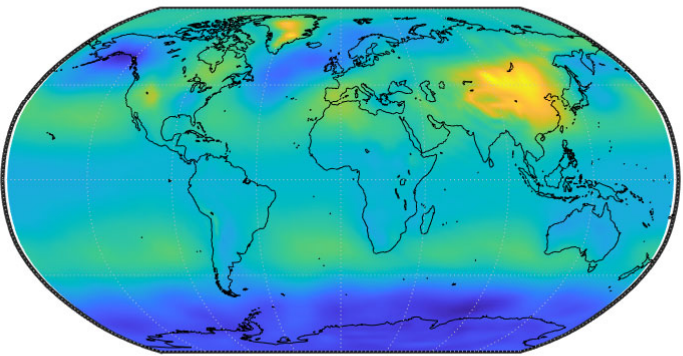} 
  \end{subfigure}\hfill
  \begin{subfigure}[c]{0.49\textwidth}
    \centering
    \adjincludegraphics[height=0.7\linewidth,valign=m]{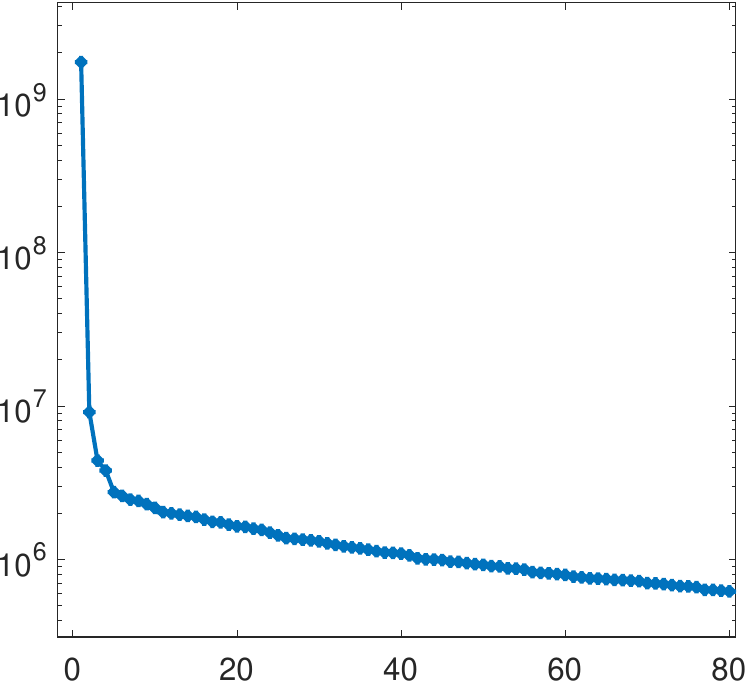}
  \end{subfigure}
  \caption{Left figure visualizes the original data for a single day. Right figure shows the singular values.}\label{fig:worldmap}
\end{figure}

\subsubsection{Metrics}\label{sect:metrics}
As discussed at the end of Sect. \ref{sec:error_bound_q=1_stated} and \ref{sec:error_bound_q>=1_stated}, the key difference between TYUC17 and TYUC17-SPI lies in the orthogonal projection. Therefore, in addition to the relative error \eqref{eq:relative_error_def},  we further consider the following  two errors that will be used in the remainder of this subsection: 

1) \textbf{RangeError}: $\|\qmat{A}-\qmat{U}\qmat{U}^{T} \qmat{A}\|_p/\| \qmat{A}-\lfloor \qmat{A}\rfloor_r\|_p -1$; 

2) \textbf{ExtraError}: $\|\qmat{U}\tilde{\qmat{U}}^T(\qmat{Q}^T\qmat{A}-(\bdPsi\qmat{Q})^{\dagger}\bdPsi\qmat{A})\|_p/\|\qmat{A}-\lfloor \qmat{A}\rfloor_r\|_p$;  


\noindent ``RangeError'' refers to orthogonal projection error (with  $\qmat{U}$ in place of $\qmat{Q}$, consistent with $\|\qmat{A}-\qmat{U}\boldsymbol{\Sigma}\qmat{V}^T\|_p$).     In ``ExtraError'', the subtraction of one is omitted as it may be less than one. The term in its numerator comes because the   error of   TYUC17 and TYUC17-SPI have two error sources:
\begin{align*}
  \|\qmat{A}-\hat{\qmat{A}}\|_p =   \normP{\qmat{A}-\qmat{U}\boldsymbol{\Sigma}\qmat{V}^{T}}=\|(\qmat{A}-\qmat{U}\qmat{U}^{T}\qmat{A}) + (\qmat{U}\tilde{\qmat{U}}^T(\qmat{Q}^T\qmat{A}-(\bdPsi\qmat{Q})^{\dagger}\bdPsi\qmat{A}))\|_p,
\end{align*}
where the second term is    caused by the sketch-and-solve  stage (the decomposition is in the same spirit  to that in \cite[Lemma A.3]{Practical_Sketching_Algorithms_Tropp}; see Sect. \ref{SMsec:proof_oblique_projection_error_decomposition} in the supplementary material for a proof). Note that it is perpendicular to the first term. \color{black}

\subsubsection{A RSVD-Onepass implementation under row-wise format}\label{sec:rsvd_one_pass}
{Before the experiment,} we include another algorithm:  a one-pass implementation  of the original RSVD   \cite{yu2018efficient,bjarkason2019PassEfficientRandomizedAlgorithms} (RSVD-Onepass for short).    
First, it computes two sketches:
\begin{align*}
    \qmat{Y} = \qmat{A}\bdOmega, \qmat{W} = \qmat{A}^{T}\qmat{A}\bdOmega;
\end{align*}
here   $\qmat{A}$ must be in row-wise format such that $\qmat{W}$ can be constructed in a single pass of $\qmat{A}$. Then, it computes the thin QR   of $\qmat{Y}$ and computes $\qmat{B}$ as follows:
\begin{align*}
    [\qmat{Q},\qmat{R}] = \texttt{qr}(\qmat{Y},0), \qmat{B} = \qmat{R}^{-T}\qmat{W}^T.
\end{align*}
Finally, the algorithm computes  the truncated SVD of $\qmat{B}\approx \tilde{\qmat{U}}\boldsymbol{\Sigma}\qmat{V}^T$ and the left approximate singular vectors $\qmat{U}=\qmat{Q}\tilde{\qmat{U}}$, and the approximation is $\hat{\qmat{A}}=\qmat{U}\boldsymbol{\Sigma}\qmat{V}^T$. 
{\color{black} It is not hard to see that RSVD-Onepass is algebraically equivalent to RSVD} of \cite{FindingStructureHalko}. The algorithm requires two sketches $\qmat{Y},\qmat{W}$,   taking storage $ms+ns$.

 {A key difference is that RSVD-Onepass is a orthogonal projection method, while TYUC17 and TYUC17-SPI are oblique projection methods. Thus  RSVD-Onepass has only one error source: 
}
\begin{align*}
  \|\qmat{A}-\hat{\qmat{A}}\|_p =  \|\qmat{A}-\qmat{U}\boldsymbol{\Sigma}\qmat{V}^{T}\|_p=\|\qmat{A}-\qmat{U}\qmat{U}^{T}\qmat{A}\|_p, 
\end{align*}
because $\qmat{U}^T\qmat{A} =\tilde{\qmat{U}}^T\qmat{Q}^T\qmat{A} =\tilde{\qmat{U}}^T\qmat{B} = \boldsymbol{\Sigma}\qmat{V}^T$ due to  $\qmat{B}=\qmat{R}^{-T}\qmat{W}^T=\qmat{Q}^T\qmat{A}$. 

 {As   RSVD-Onepass only supports row-wise streaming data, it is not directly comparable to TYUC17 or TYUC17-SPI that allow general linear updates of the data matrix. Nevertheless, we still report its results for two reasons. First, as discussed,  RSVD only involves the orthogonal projection error, making it a strong baseline under storage-limited settings. Second, the comparison helps illustrate how SPI affects the two error sources. On the other hand, it is also conceivable that RSVD-Onepass itself could be enhanced by SPI, which we leave for future investigation.}  
\subsubsection{NIST dataset performance}


The main task in the NIST dataset is to extract the principal components of the data matrix, i.e., its left singular vectors in the context. Hence, we mainly focus on approximating the dominant $10$ singular values and visualizing the approximated left singular vectors. Reconstruction errors are also reported for completeness. \color{black}
In the sequel,  TYUC17-SPI and TYUC17 employ their theoretically guided sketch sizes, while RSVD-Onepass uses an exhaustive search to find the best parameter $s$.
\color{black}

\begin{example}
This example shows the  feature extraction quality of the tested algorithms in NIST dataset.  We used the $16384 \times 20000$ matrix generated in Sect. \ref{sec:real_data}. To test the  performance, we mainly focusing on the dominant singular values and   left singular vectors. All results are averaged over 10 independent repeated experiments.
We set the same storage budget $100n$ and $160n$ for all the algorithms. 
\end{example}



\begin{figure}[htbp]
    \captionsetup{font=footnotesize}
    \centering
    \begin{subfigure}[b]{0.45\textwidth}
        \includegraphics[width=0.95\textwidth]{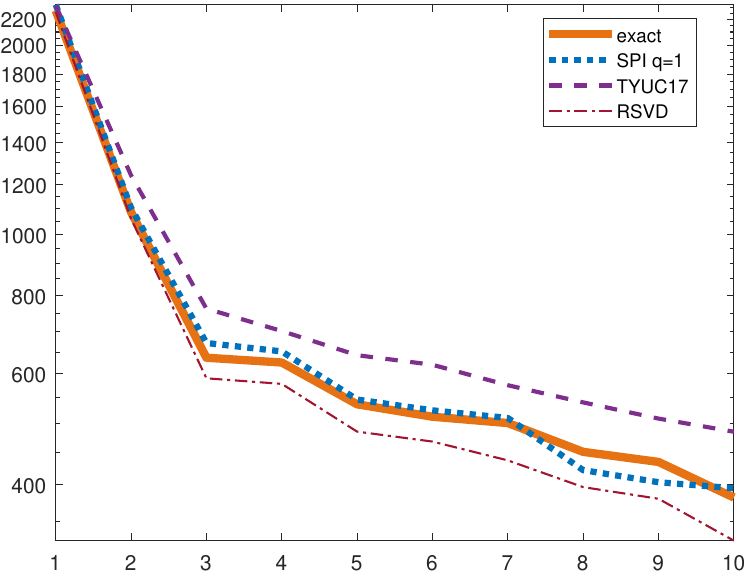} 
        \caption{\small Singular value (storage: $100n$)}\label{fig:NIST_SingularValue_T50}
    \end{subfigure}
    \begin{subfigure}[b]{0.45\textwidth}  
        \includegraphics[width=0.95\textwidth]{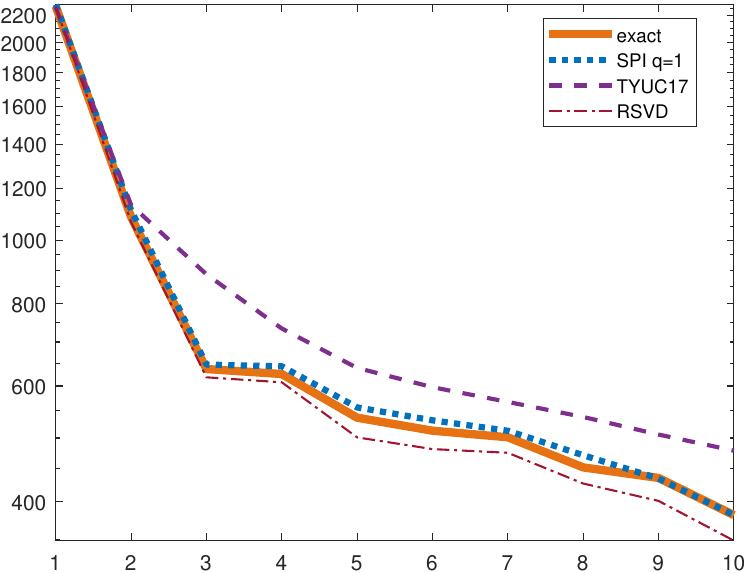} 
        \caption{\small Singular value (storage: $ 160n$)}\label{fig:NIST_SingularValue_T80}
    \end{subfigure}
    \caption{Singular value estimation. The $x$-axis means the     singular value index and the $y$-axis shows the magnitude of the singular values. 
    TYUC17-SPI  uses the  parameter guided for polynomial decay and   TYUC17 uses the parameters suggested in  \cite[Sect. 4.5.2]{Practical_Sketching_Algorithms_Tropp}. The exact singular values are computed by Matlab's \texttt{svds} function.}\label{fig:SingularValueNIST}
\end{figure}

\color{black}

Fig. \ref{fig:SingularValueNIST} illustrates the accuracy of   approximating the first ten singular values under two storage budgets: \( 100n\)   and \( 160n\). The \(x\)-axis corresponds to the singular value index, while the \(y\)-axis shows the singular values. The red curve  represents the true singular values. We observe that TYUC17-SPI (\(q=1\)) closely follows the exact singular values, indicating its high accuracy. By contrast,   RSVD-Onepass tends to underestimate the singular values, and TYUC17   overestimates them. When the storage budget increases from \(100n\) to \(160n\), all methods improve their accuracy, but TYUC17-SPI remains the closest to the exact values. 

\begin{figure}[htbp]
    \captionsetup{font=footnotesize}
    \centering
    \begin{subfigure}[b]{0.45\textwidth}
        \includegraphics[width=0.95\textwidth]{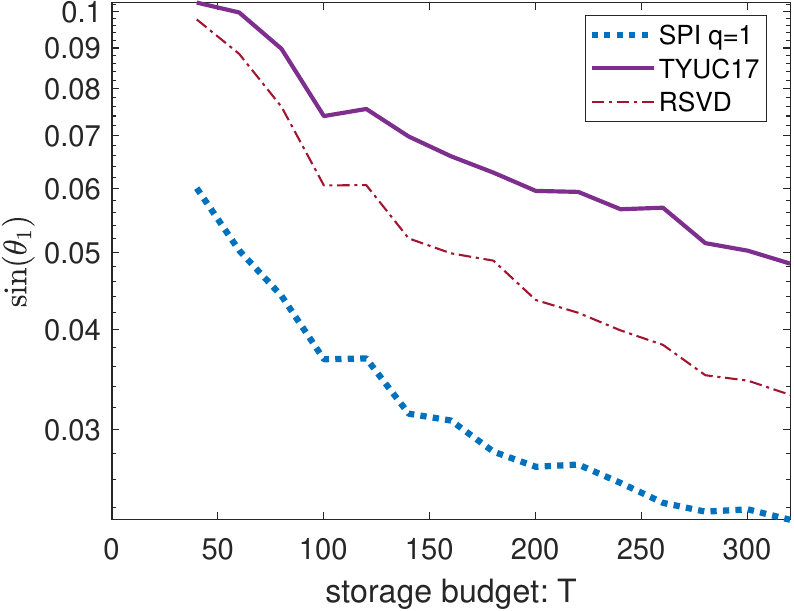} 
        \caption{\small The 1st canonical angle}\label{fig:NIST_SingularVectorError1}
    \end{subfigure}
    \begin{subfigure}[b]{0.45\textwidth}
        \includegraphics[width=0.93\textwidth]{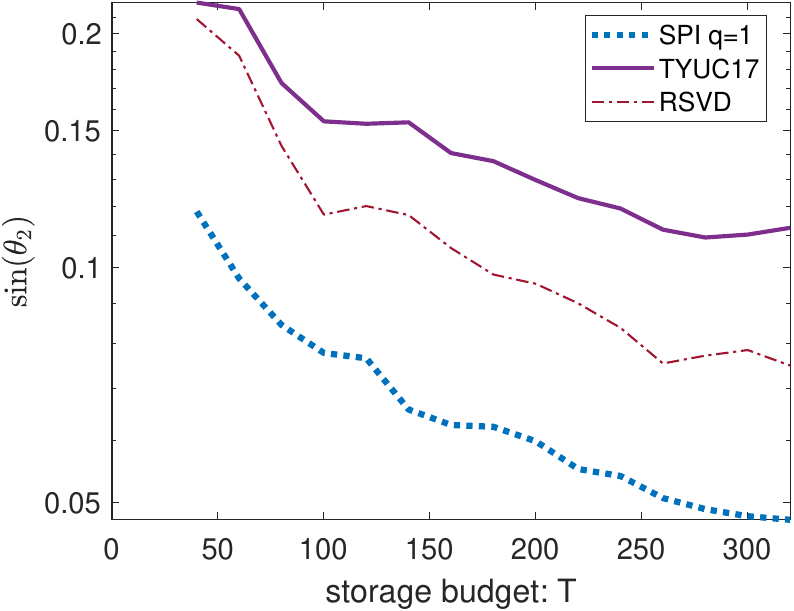} 
        \caption{\small The 2nd canonical angle}\label{fig:NIST_SingularVectorError2}
    \end{subfigure}
    \caption{Sines of the first two canonical angles between the true and computed left dominant singular subspaces, with increasing  storage budget $\hat Tn$ ($x$-axis: $\hat T$).} \label{fig:NIST_SingularVectorError}
\end{figure}

    \begin{figure}
        \captionsetup{font=small}
        \begin{flushleft}
        \hspace{0.07\textwidth} \textbf{\small RSVD-OP} \hspace{0.05\textwidth} 
        \textbf{\small TYUC17-SPI} \hspace{0.05\textwidth}
        \textbf{\small TYUC17} \hspace{0.10\textwidth}
        \textbf{\small Exact}
        \end{flushleft}
        \centering
        \begin{subfigure}[b]{\textwidth}
            \centering
            \includegraphics[width=0.9\textwidth]{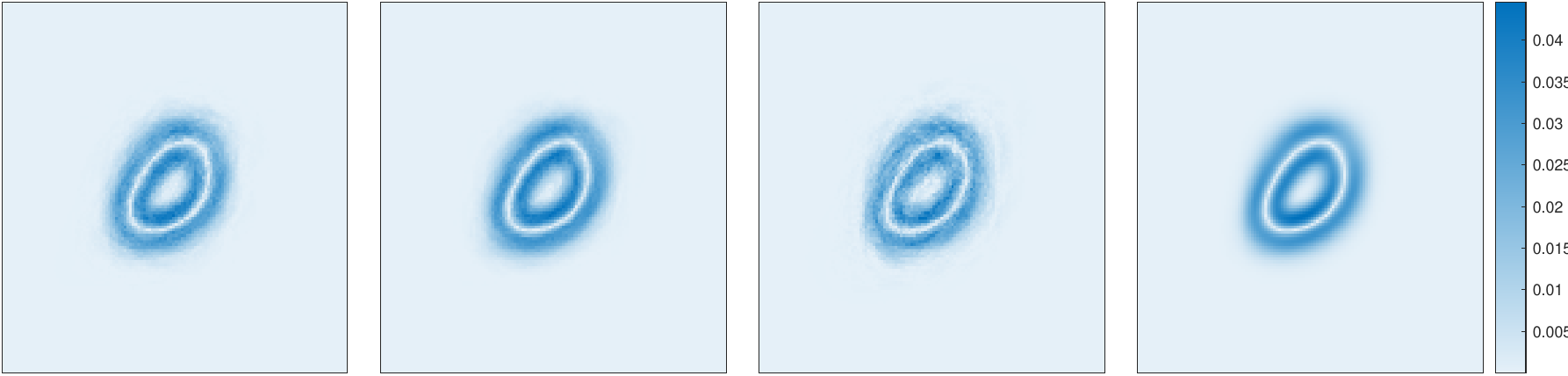} 
            \caption{The 2nd singular vector}
        \end{subfigure}
        \vspace{0.2cm} 
        \begin{subfigure}[b]{\textwidth}
            \centering
            \includegraphics[width=0.9\textwidth]{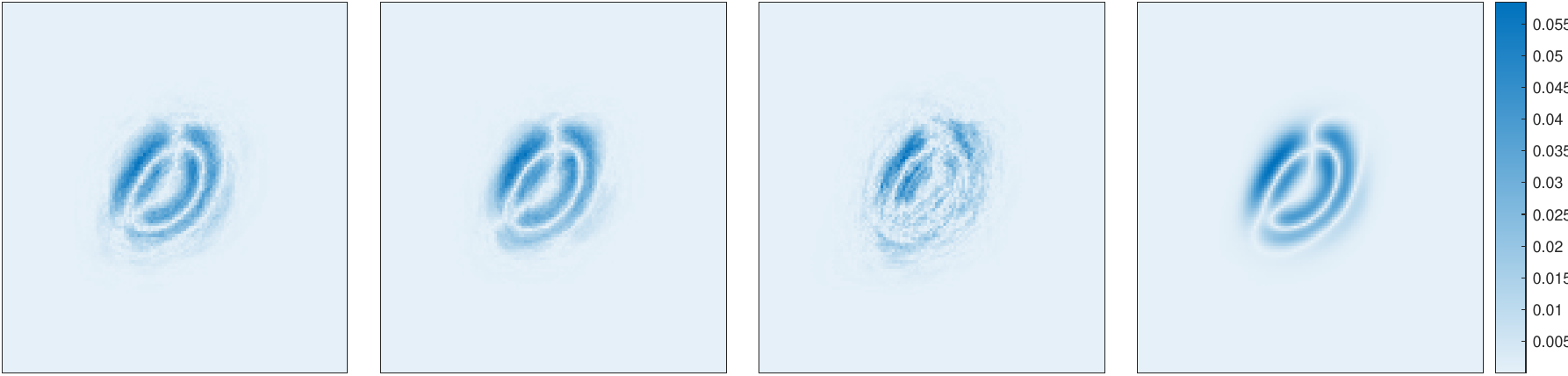}
            \caption{The 3rd singular vector}
        \end{subfigure}
        \vspace{0.2cm}
        \begin{subfigure}[b]{\textwidth}
            \centering
            \includegraphics[width=0.9\textwidth]{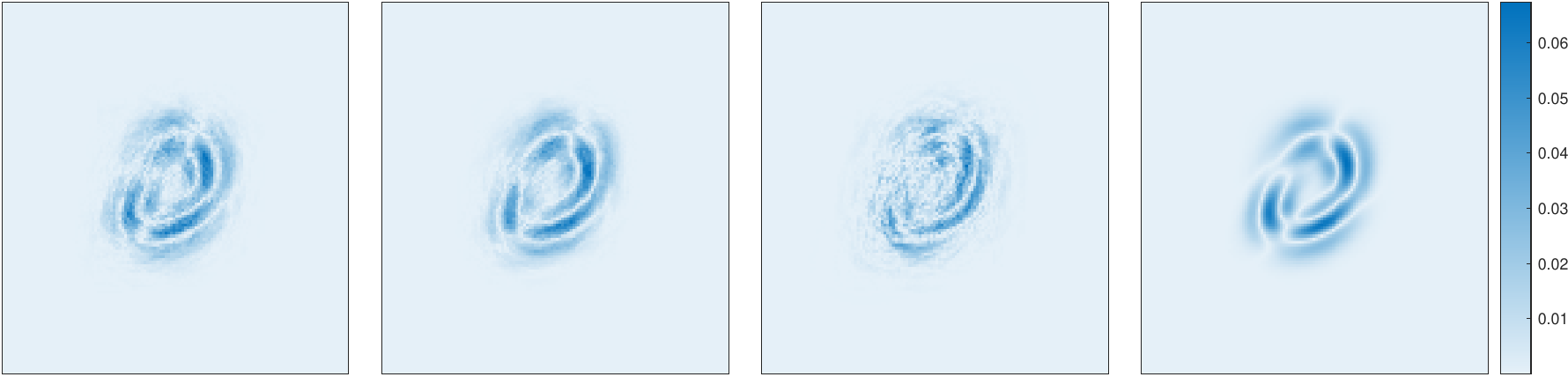}
            \caption{The 8th singular vector}
        \end{subfigure}
        \caption{ Three singular vectors computed by the algorithms. RSVD-OP means RSVD-Onepass. 
        }\label{fig:NIST_SingularVector}
    \end{figure}

    Fig. \ref{fig:NIST_SingularVectorError}  {displays   the first two canonical angles between true and computed left dominant singular subspaces  by   TYUC17-SPI, TYUC17 and RSVD-Onepass algorithms, respectively, with increasing storage budget $\hat{T}n$.   $x$-axis represents \(\hat T\), while the $y$-axis shows the sines of the canonical angles. We observe that TYUC17-SPI consistently achieves the highest accuracy among the three. RSVD-Onepass perform better than TYUC17, because its storage budget $\hat Tn = (m+n)s$ leads to a larger effective sketch size than TYUC17's $ms+dn$ (since $d>s$, $\qmat{Y}$ of the former can be larger). TYUC-SPI outperforms both by leveraging   storage for $\qmat{Z}$ via the mixed-precision strategy, resulting in the strongest $\qmat{Y}$ and overall best performance.}

  Fig. \ref{fig:NIST_SingularVector} visualizes the 2nd, 3rd, and 8th singular vectors produced by the algorithms {with storage $320n$}. Each vector is reshaped into a $128\times 128$ matrix. We observe that TYUC17-SPI (second column) yields the most accurate approximations, closely matching the exact vectors (fourth column). In contrast, TYUC17 (third column) shows noticeable deviations starting from the 3rd singular vector. This      illustrates   SPI's effectiveness in improving the accuracy of singular vector approximations. While RSVD-Onepass (first  column) also produces relatively high-quality results, the TYUC17-SPI solutions are sharper and exhibit less noise, demonstrating  the accuracy using SPI.\color{black} 

\begin{table}[htbp] 
  \centering
  \caption{Reconstruction errors of different algorithms on NIST. ``Error'' refers to the relative error \eqref{eq:relative_error_def}. RSVD-Onepass's Error$=$RangeError and it has no ExtraError.}
  \label{tab:NIST_errors}
  \resizebox{\columnwidth}{!}{%
    \begin{tabular}{llcccccc} 
      \toprule
      \multirow{2}{*}{Method} & \multirow{2}{*}{Type} & \multicolumn{3}{c}{Frobenius Error} & \multicolumn{3}{c}{Spectral Error} \\
      \cmidrule(lr){3-5} \cmidrule(lr){6-8} 
       & & 200 & 260 & 320 & 200 & 260 & 320 \\
      \midrule
      \multirow{3}{*}{\makecell{TYUC17 \\ -SPI}} & Error & 0.0423 & 0.0335 & 0.0221 & 7.648e-3 & 4.177e-3 & 1.742e-3 \\
      & RangeError & 0.0178 & 0.0123 & 0.0098 & 5.099e-3 & 3.062e-3 & 1.309e-3 \\
      & ExtraError & 0.2244 & 0.2082 & 0.1586 & 0.5436 & 0.4268 & 0.3703 \\
      \midrule
      \multirow{3}{*}{TYUC17} & Error & 0.1386 & 0.1846 & 0.0773 & 1.562e-1 & 5.195e-2 & 3.155e-2 \\
      & RangeError & 0.0958 & 0.0748 & 0.0545 & 1.277e-1 & 2.878e-2 & 2.379e-2 \\
      & ExtraError & 0.3093 & 0.2450 & 0.2205 & 0.7085 & 0.6708 & 0.4612 \\
            \midrule
      RSVD-OP &  Error & 0.0326 & 0.0222 & 0.0165 & 5.133e-3 & 3.447e-3 & 1.371e-3 \\
      \bottomrule
    \end{tabular}%
  }
\end{table}




 {
Table \ref{tab:NIST_errors} summarizes the errors of the   algorithms under three storage budgets: \(200n\), \(260n\), and \(320n\). The table includes both Frobenius and spectral norm errors. We observe that TYUC17-SPI consistently outperforms TYUC17 across all metrics, demonstrating the effectiveness of the SPI strategy. Specifically, for the RangeError,   TYUC17-SPI shows over $5\times$ improvement for the Frobenius norm and $10\times$ improvement for the spectral norm compared to TYUC17. Furthermore, for the ExtraError caused by the  sketch-and-solve stage,   SPI also has considerable improvements.

We   also see  that the RangeError of TYUC17-SPI is lower than that of RSVD (whose RangError$=$Error according to Sect. \ref{sec:rsvd_one_pass}), particularly under the Frobenius norm. 
 However, RSVD performs best on the Error metric, as its orthogonal projection avoids the corange $\qmat{V}$ bias inherent in oblique methods. This does not contradict Fig. \ref{fig:SingularValueNIST}, \ref{fig:NIST_SingularVectorError}, and \ref{fig:NIST_SingularVector}: while these figures only reflect  singular value and range $\qmat{U}$ estimation, the Error metric here is additionally influenced by the corange $\qmat{V}$.     This   suggests that SPI may   be   applied to also enhance the corange estimation in the future study.
}



 \subsubsection{Climate dataset performance}

\begin{example}
{We tested the algorithms in scientific computation, using the $10512\times 28144$ climate data introduced in Sect. \ref{sec:real_data}. Feature vectors derived from climate data are   paramountly important in both scientific research and practical applications. These vectors encapsulate critical information extracted from the datasets.   For the pressure data of the $73\times 144$ meshgrid on the surface, the left singular vectors   reveal the spatial structure of the data, for example, the distribution of pressure belts. We test rank-$10$ approximation and all results are averaged over $10$ independent trials.
}\end{example}


\begin{figure}[htbp]
    \captionsetup{font=small}
    \centering
    \begin{subfigure}[b]{0.45\textwidth}
        \includegraphics[width=0.95\textwidth]{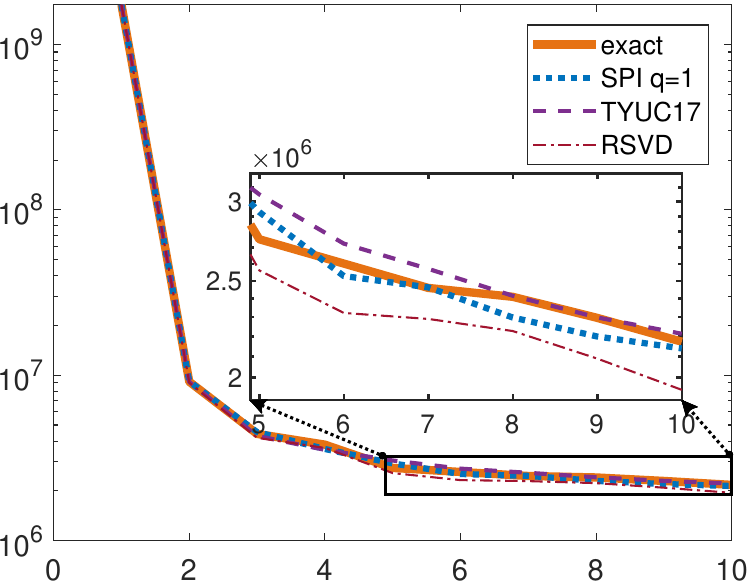} 
        \caption{\small Singular value (storage: $110n$)}\label{fig:Climate_SingularValue_T110}
    \end{subfigure}
    \begin{subfigure}[b]{0.45\textwidth}
        \includegraphics[width=0.95\textwidth]{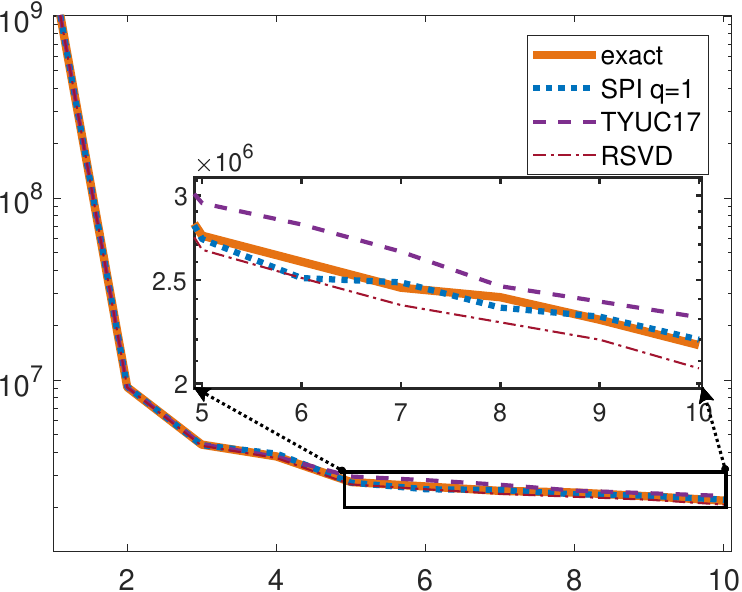} 
        \caption{\small Singular value (storage: $ 170n$)}\label{fig:Climate_SingularValue_T170}
    \end{subfigure}
    \caption{Singular value estimation in climate dataset. 
    }\label{fig:SingularValueClimate}
\end{figure}

\begin{figure}[htbp]
    \captionsetup{font=small}
    \centering
    \begin{subfigure}[b]{0.45\textwidth}
        \includegraphics[width=0.95\textwidth]{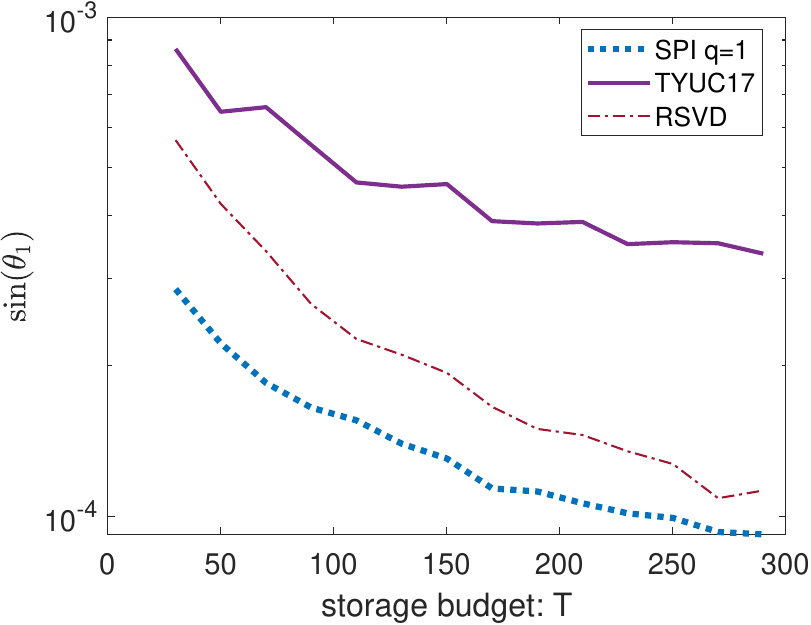} 
         \caption{\small The 1st canonical angle}\label{fig:Climate_SingularVectorError1}
    \end{subfigure}
    \begin{subfigure}[b]{0.45\textwidth}
        \includegraphics[width=0.95\textwidth]{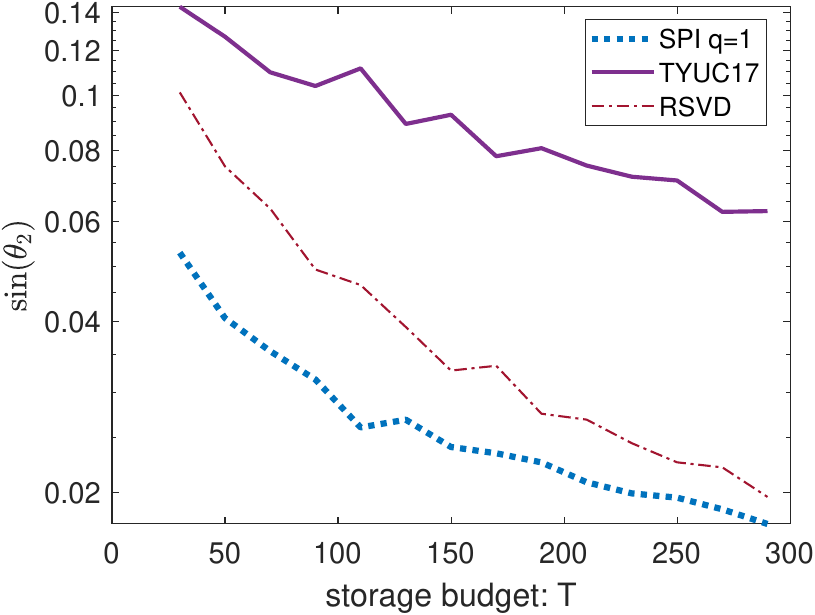} 
         \caption{\small The 2nd canonical angle}\label{fig:Climate_SingularVectorError2}
    \end{subfigure}
    \caption{Sines of the first two canonical angles between the true and computed  left dominant singular subspace, with increasing  storage budget $\hat Tn$ ($x$-axis: $\hat T$).} \label{fig:Climate_SingularVectorError}
\end{figure}

 {The singular value estimation and canonical angles  are respectively plotted in Fig. \ref{fig:SingularValueClimate} and \ref{fig:Climate_SingularVectorError}. The behavior on the climate dataset is similar to that on NIST, but all algorithms perform better owing to the faster spectral decay of the climate matrix (c.f. Fig. \ref{fig:worldmap}).   In particular, Fig. \ref{fig:Climate_SingularVectorError} shows that the curves of RSVD-Onepass are closer to those of TYUC17-SPI, as the rapid spectral decay   enables RSVD to perform well. 
   }
  

 \color{black}

\begin{figure}
    \captionsetup{font=small}
    \centering
    \begin{subfigure}[b]{0.49\textwidth}
        \includegraphics[width=1\textwidth]{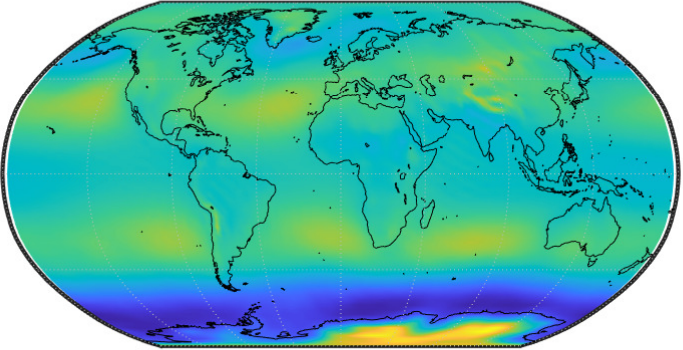} 
        \caption{\small RSVD-Onepass}\label{fig:Climate_Geo_rsvd_1}
    \end{subfigure}
    \begin{subfigure}[b]{0.49\textwidth}
        \includegraphics[width=1\textwidth]{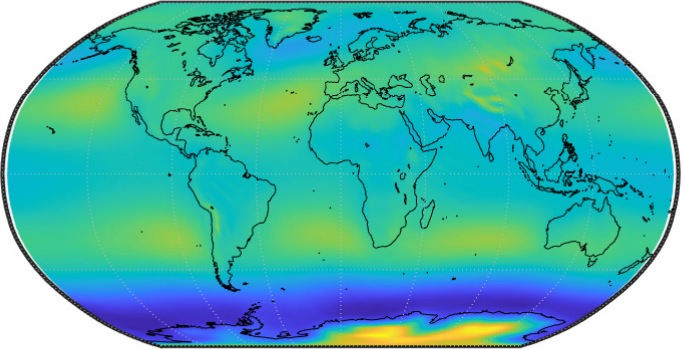} 
        \caption{\small TYUC17-SPI}\label{fig:Climate_Geo_spi_1}
    \end{subfigure}
    \begin{subfigure}[b]{0.49\textwidth}
        \includegraphics[width=1\textwidth]{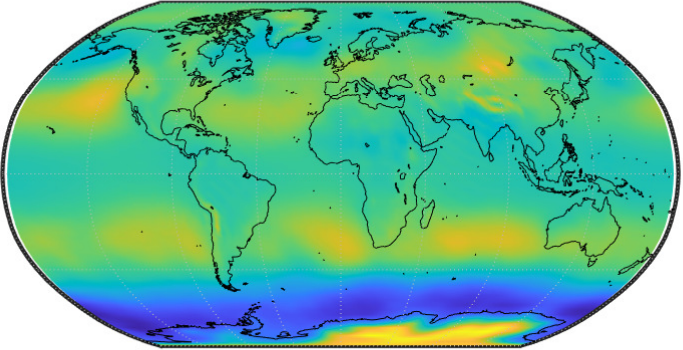} 
        \caption{\small TYUC17}\label{fig:Climate_Geo_TYUC17_1} 
    \end{subfigure}
    \begin{subfigure}[b]{0.49\textwidth}
        \includegraphics[width=1\textwidth]{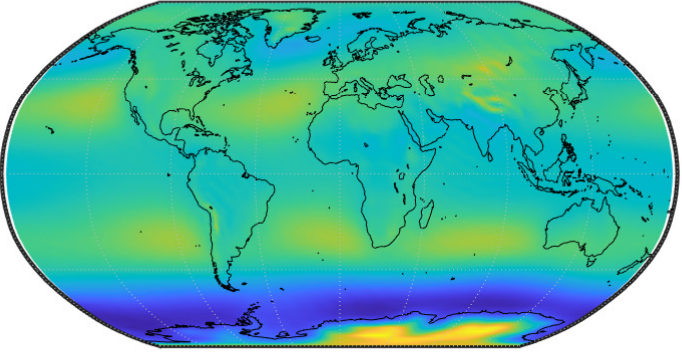} 
        \caption{\small Exact}\label{fig:Climate_Geo_exact_1}
    \end{subfigure}
    \caption{The 1st singular vector of climate data. 
    All heat maps use the same color bar (omitted here for brevity). }\label{fig:Climate_Geo_1}
\end{figure}

\color{black}
 {Fig. \ref{fig:Climate_Geo_1} and \ref{fig:Climate_Geo_2} visualize the features extracted from the climate dataset. They reveal some long-term distinct climatological patterns and can be explained respectively. We set storage $60n$. The leading singular vector   in Fig. \ref{fig:Climate_Geo_1} captures persistent pressure belt structures with remarkable fidelity. While both SPI and RSVD-Onepass generate physically consistent approximations, the TYUC17 solution exhibits notable deviations.}
 

Fig. \ref{fig:Climate_Geo_2} further illustrates the second singular vector's sensitivity to Eurasian landmass dynamics and Southern Hemisphere westerlies. The SPI reconstruction closely matches the reference solution, particularly in resolving the pressure gradient over prevailing westerlies. By contrast, TYUC17 and RSVD-Onepass show systematic errors south of 45°S latitude. This differential performance suggests that    TYUC17 enhanced by SPI   effectively preserves critical synoptic-scale features   that are attenuated by prior methods.


\begin{figure}
    \captionsetup{font=small}
    \centering
    \begin{subfigure}[b]{0.49\textwidth}
        \includegraphics[width=1\textwidth]{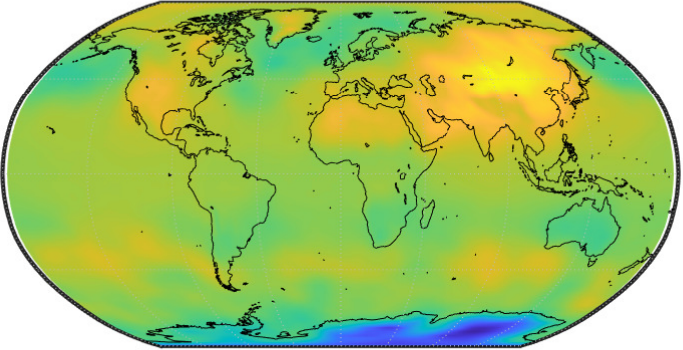} 
        \caption{\small RSVD-Onepass}\label{fig:Climate_Geo_rsvd_2}
    \end{subfigure}
    \begin{subfigure}[b]{0.49\textwidth}
        \includegraphics[width=1\textwidth]{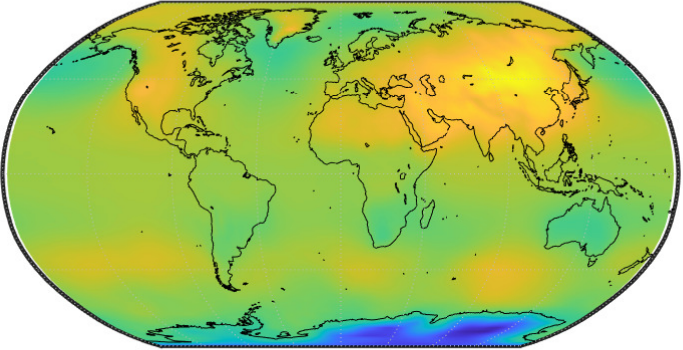} 
        \caption{\small TYUC17-SPI}\label{fig:Climate_Geo_spi_2}
    \end{subfigure}
    \begin{subfigure}[b]{0.49\textwidth}
        \includegraphics[width=1\textwidth]{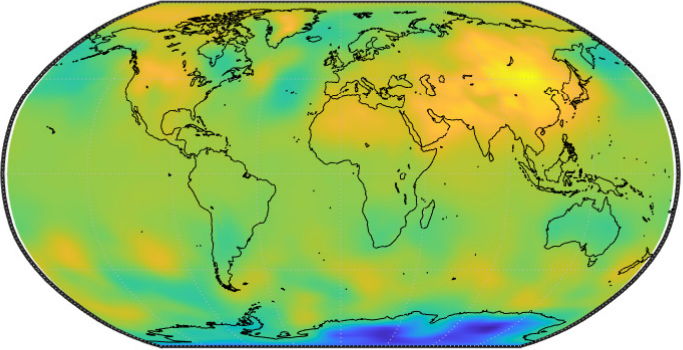} 
        \caption{\small TYUC17}\label{fig:Climate_Geo_TYUC17_2}
    \end{subfigure}
    \begin{subfigure}[b]{0.49\textwidth}
        \includegraphics[width=1\textwidth]{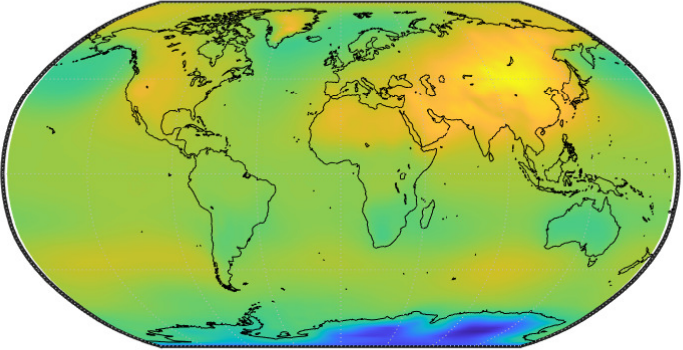} 
        \caption{\small Exact}\label{fig:Climate_Geo_exact_2}
    \end{subfigure}
    \caption{The 2nd singular vector of climate data. 
   }\label{fig:Climate_Geo_2}
\end{figure}

\begin{table}[htbp] 
  \centering
  \small 
  \caption{Reconstruction errors of different algorithms on climate dataset. ``Error'' refers to the relative error \eqref{eq:relative_error_def}. RSVD-Onepass's Error$=$RangeError.}
  \label{tab:errors_climate}
  \begin{tabular}{llcccccc} 
    \toprule
    \multirow{2}{*}{Method} & \multirow{2}{*}{Type} & \multicolumn{3}{c}{Frobenius Error} & \multicolumn{3}{c}{Spectral Error} \\
    \cmidrule(lr){3-5} \cmidrule(lr){6-8} 
     & & 110 & 140 & 170 & 110 & 140 & 170 \\
    \midrule
    \multirow{3}{*}{\makecell{TYUC17 \\ -SPI}} & Error & 0.0281 & 0.0175 & 0.0099 & 2.316e-3 & 1.282e-3 & 6.402e-4 \\
    & RangeError & 0.0079 & 0.0043 & 0.0031 & 1.445e-3 & 9.454e-4 & 3.219e-4 \\
    & ExtraError & 0.2023 & 0.1635 & 0.1170 & 0.5928 & 0.4072 & 0.2961 \\
    \midrule
    \multirow{3}{*}{TYUC17} & Error & 0.0933 & 0.0659 & 0.0452 & 1.825e-1 & 4.766e-2 & 1.333e-2 \\
    & RangeError & 0.0590 & 0.0440 & 0.0298 & 1.124e-1 & 2.887e-2 & 8.560e-3 \\
    & ExtraError  & 0.2715 & 0.2148 & 0.1789 & 0.7737 & 0.6137 & 0.5224 \\
     \midrule
    RSVD-OP & Error & 0.0126 & 0.0078 & 0.0050 & 2.511e-3 & 1.190e-3 & 4.467e-4 \\   
    \bottomrule
  \end{tabular}
\end{table}

Table \ref{tab:errors_climate} reports the errors on the climate dataset. Overall errors are smaller than those on NIST, owing to the faster spectral decay of the climate data, such that algorithms can obtain lower errors with fewer storage budgets (compared with NIST). Meanwhile, the   performance is consistent with that in Table \ref{tab:NIST_errors}: TYUC17-SPI again achieves the smallest RangeError, while RSVD performs best on the Error metric.

 \subsubsection{Performance with different $q$}
We conduct quantitative experiments to examine multiple SPI iterations, as follows.
 \begin{example}
        We test the performance of TYUC17-SPI with $q=1,2,3$ on both datasets and set the storage   $190n$. All results are averaged over $10$ independent trials.
 \end{example}

\begin{table}[h!] 
\centering
\caption{Performance of TYUC-SPI with $q=1,2,3$ on both real datasets.}
\resizebox{\columnwidth}{!}{
\begin{tabular}{cccccccc} 
\toprule
\multirow{2}{*}{Data} & \multirow{2}{*}{Type} & \multicolumn{3}{c}{Frobenius Error} & \multicolumn{3}{c}{Spectral Error} \\
\cmidrule(lr){3-5} \cmidrule(lr){6-8}
& & 1 & 2 & 3 & 1 & 2 & 3 \\
\midrule
\multirow{3}{*}{NIST} & Error & 0.0496 & 0.0489 & 0.0487 & 0.0159 & 0.0138 & 0.0124 \\
& RangeError & 0.0222 & 0.0217 & 0.0215 & 0.0093 & 0.0080 & 0.0077 \\
& ExtraError & 0.2383 & 0.2373 & 0.2373 & 0.5771 & 0.5691 & 0.5652 \\
\midrule
\multirow{3}{*}{Climate} & Error &0.0087 & 0.0084 & 0.0083 & 2.124e-4 & 2.006e-4 & 1.965e-4 \\
& RangeError &2.661e-3 & 2.626e-3 & 2.615e-3 & 1.550e-4 & 1.463e-4 & 1.379e-4 \\
& ExtraError &0.1102&0.1078&0.1069&0.4315&0.4312&0.4310 \\
\bottomrule
\end{tabular}
}
\label{tab: Qreal_performance}
\end{table}


Table \ref{tab: Qreal_performance} reports the performance of TYUC17-SPI with    $q=1,2,3$ on the NIST and climate datasets. Errors generally decrease as $q$ increases, especially for the spectral norm, similar to Fig. \ref{fig:qFigure}, confirming the effectiveness of multiple SPI iterations.



\color{black}

\section{Analysis for the \texorpdfstring{$q=1$}{q=1} Bound}\label{sec:analysis}


This section is focused on proving Theorem \ref{thm:oblique_proj_error_q=1}.  The key is to establish the projection error $\|(\qmat{I}-P_{\hat{\qmat{Y}}})\qmat{A}\|_F$.  \color{black}
The difficulty  lies in the interaction between $\qmat{A}$ and $\qmat{Z}$. Roughly speaking,
 in the standard power iteration $\qmat{Y}=(\qmat{A}\qmat{A}^{T})^q\qmat{A}\bdOmega$, one can analyze $\boldsymbol{\Sigma}^{(2q+1)}\qmat{V}^{T}\bdOmega$, with $\qmat{A}=\qmat{U}\boldsymbol{\Sigma}\qmat{V}^{T}$ its SVD, so that the analysis of RSVD applies similarly. 
However, SPI $\hat{\qmat{Y}}=\qmat{Z}\qmat{Z}^{T}\qmat{A}\bdOmega$ does not have such a   representation. We address this   using similar ideas as \cite{FindingStructureHalko}, but   with some more involved decompositions. Fortunately, when $q=1$ these decompositions still work. In the following, Sect. \ref{sec:tech_lemma_qeq1} introduces necessary lemmas, Sect. \ref{sec:deter_qb_err} provides deterministic projection error, while Sect. \ref{sec:prob_qb_err} and \ref{sec:onepass_err_1}    establish the averaged error bounds.  
\color{black}


\subsection{Technical lemmas} \label{sec:tech_lemma_qeq1} Some useful lemmas are recalled first.
\begin{lemma}[Parallel sums \cite{tropp2023RandomizedAlgorithms}]\label{lem:parallel_sum}
    For positive semidefinite (PSD) matrices $\qmat{A},\qmat{B}$ with the same dimension, define the parallel sum $\qmat{A} : \qmat{B}$ to be the PSD matrix
    \begin{align*}
    \qmat{A} : \qmat{B} = \qmat{A} - \qmat{A} (\qmat{A} + \qmat{B})^\dagger \qmat{A},
    \end{align*}
 where  $^\dag$ denotes the Moore–Penrose inverse.   The following   hold:
    \begin{enumerate}
        \item The parallel sum is symmetric; that is, $\qmat{A} : \qmat{B} = \qmat{B} : \qmat{A}$.
        \item If $\qmat{A}$ and $\qmat{B}$ are strictly positive definite, then $\qmat{A} : \qmat{B} = (\qmat{A}^{-1} + \qmat{B}^{-1})^{-1}$.
        \item The mapping $\qmat{B} \mapsto \qmat{A} : \qmat{B}$ is monotone and concave with respect to the PSD ordering, and it is bounded above as $\qmat{A} : \qmat{B} \preceq \qmat{A}$.
        \item $\| \qmat{A} : \qmat{B} \|_p \leq \| \qmat{A} \|_p : \| \qmat{B} \|_p$ for any matrix Schatten $p$-norm with $1 \leq p \leq \infty$.
    \end{enumerate}
    \end{lemma}

    Here $\normP[p]{\qmat{A}}=(\sum_{i=1}\sigma_i(\qmat{A})^p)^{1/p}$. In particular $\normP[2]{\qmat{A}}=\normF{\qmat{A}}$ and $\normP[\infty]{\qmat{A}}=\normSpectral{\qmat{A}}$. 


\begin{lemma}[Moment bounds \cite{tropp2023RandomizedAlgorithms}] \label{lem:SGT}
    Consider fixed matrices $\qmat{S} \in \mathbb{R}^{l \times n}$ and $\qmat{T} \in \mathbb{R}^{p \times q}$ and a standard Gaussian matrix $\qmat{G} \in \mathbb{R}^{n \times p}$. Then 
    \begin{align*}
        \mathbb E{\normF{\qmat{S}\qmat{G}\qmat{T}}^2}  = \normF{\qmat{S}}^2 \normF{\qmat{T}}^2,  ~
        \mathbb E{\normSpectral{\qmat{S}\qmat{G}\qmat{T}}^2}  \leq \left(\normF{\qmat{S}}\normSpectral{\qmat{T}}+\normSpectral{\qmat{S}}\normF{\qmat{T}}\right)^2.
    \end{align*}
\end{lemma}

\begin{lemma}[Inverse moment bounds, Frobenius norm  \cite{tropp2023RandomizedAlgorithms,FindingStructureHalko}]\label{lem:inverse_moment_F}
Let \( \qmat{G} \in \mathbb{R}^{r \times k} \) with \( r \leq k \) be standard Gaussian. Then  
    \begin{align*}
    \mathbb{E}(\qmat{G}\qmat{G}^{T})^{-1} &= \frac{1}{k - r - 1} \mathbf{I}_r, & r &\leq k - 2, \\
    \mathbb{E}\|(\qmat{G}\qmat{G}^{T})^{-1}\|_F^2 &= \frac{r(k - 1)}{(k - r)(k - r - 1)(k - r - 3)}, & r &\leq k - 4.
    \end{align*}
   For any     \( t > 1 \) and $r\leq k-4$, 
    \begin{equation*} 
    \mathbb{P} \left\{ \|\qmat{G}^\dagger\|_F > t\cdot\sqrt{\frac{3r}{k - r + 1}} \right\} 
    \leq   t^{-(k - r ) }.
    \end{equation*}
    \end{lemma}
    \begin{lemma}[Inverse moment bounds, spectral norm \cite{tropp2023RandomizedAlgorithms,FindingStructureHalko}]\label{lem:inverse_moment_spectral}
        For $0 < p \leq 18$ and $p \leq \frac{k-r}{2}$, the standard Gaussian matrix $\qmat{G} \in \mathbb{R}^{r \times k}$ satisfies
        \begin{align*}
        \left( \mathbb{E} \| (\qmat{G}\qmat{G}^{T})^{-1} \|^p \right)^{1/p} \leq \frac{e^2 (k + r)}{2(k - r)^2}, ~
            \probleftright{ \|\qmat{G}^\dagger\| \geq \frac{e \sqrt{k}}{ k-r+1  } \cdot t} &\leq t^{- {(k-r+1)} },~~(t>1).
        \end{align*}       
     \end{lemma}

     \begin{corollary}
        \label{col:SigmaPhi2Phi1_second_moment_bound} Let $\qmat{A}\in\bbR^{m\times n}$ be fixed, and $\qmat{G}_1\in\bbR^{s\times k}~(\color{black}k\geq s+2\color{black}),\qmat{G}_2\in\bbR^{n\times k}$ be independent standard Gaussian matrices. Then
        \begin{align*}
           \mathbb E\|\qmat{A}\qmat{G}_2\qmat{G}_1^\dagger \|_F^2 =\! \frac{s\normFSquare{\qmat{A}}}{k-s-1}
            ; ~ \mathbb E\|\qmat{A}\qmat{G}_2\qmat{G}_1^\dagger \|^2 \!\leq\! \left(\normSpectral{\qmat{A}}\sqrt{\frac{s}{k-s-1}}+\normF{\qmat{A}}\frac{e\sqrt{k+s}}{\sqrt{2}(k-s)}\right)^2.
        \end{align*}
    \end{corollary}

     \begin{lemma}[\cite{vershynin2018high}]
        \label{lem:LipschitzConcentrationIneq}
       Let \( h \) be a Lipschitz function on matrices:
        $  
             |h(\qmat{X}) - h(\qmat{Y})| \leq L \normF{\qmat{X} - \qmat{Y}}$, $\forall\qmat{X}, \qmat{Y}.
        $
         Draw a standard Gaussian matrix \( \qmat{G} \). Then for $u>0$, 
         \[
             \probleftright{ h(\qmat{G}) - \mathbb E h(\qmat{G})  > uL} \leq e^{-u^2/2}.
         \]
     \end{lemma}
\begin{corollary}\label{col:SGT_deviation}
    For standard Gaussian matrix $\qmat{G}$ and two   random matrices $\qmat{S}$ and $\qmat{T}$ independent from $\qmat{G}$, for $u>0$ we have:  
    \begin{align}
        \probleftright{\normSpectral{\qmat{S}\qmat{G}\qmat{T}}>e_1e_4+e_2e_3+e_2e_4u}<e^{-u^2/2}+\probleftright{E^C},
    \end{align}
    where $E\!=\!\{\qmat{S},\qmat{T}\!\mid\! \normF{\qmat{S}} \leq e_1,\normSpectral{\qmat{S}} \leq e_2,\normF{\qmat{T}} \leq e_3,\normSpectral{\qmat{T}}\leq e_4\}$ and $e_i>0,~i=1,\ldots,4$.
     \end{corollary}


The proofs of Corollaries \ref{col:SigmaPhi2Phi1_second_moment_bound} and \ref{col:SigmaPhi2Phi1_second_moment_bound} are given in Appendix \ref{sec:proofs_technical_lemmas}. 
Using the above results, the following   bounds can be derived. They have appeared in the proofs of \cite[Thm. 10.7 and Thm. 10.8]{FindingStructureHalko}. 
\begin{corollary}
    \label{col:AG2G1_deviation_refined_from_finding}
    Under the setting of Corollary \ref{col:SigmaPhi2Phi1_second_moment_bound}, for any $t>1$, $u>0$, $k-s\geq 4$, we have
    \begin{small}
    \begin{align*}
        &\probleftright{\normSpectralnoleftright{\qmat{A}\qmat{G}_2\qmat{G}_1^\dagger} > (\normF{\qmat{A}} + u\normSpectralnoleftright{\qmat{A}})\frac{e\sqrt{k}}{k-s+1}t  + \normSpectralnoleftright{\qmat{A}}\sqrt{\frac{3s}{k-s+1}}t  } < 2t^{-{(k-s)}} + e^{-u^2/2};\\
       & \probleftright{\normF{\qmat{A}\qmat{G}_2\qmat{G}_1^\dagger}  > \normF{\qmat{A}}\sqrt{\frac{3s}{k-s+1}}t + \normSpectralnoleftright{\qmat{A}}\frac{e\sqrt{k}}{k-s+1}tu   } < 2t^{-{(k-s)}} + e^{-u^2/2}.
    \end{align*}
\end{small}
\end{corollary}

\subsection{Deterministic projection error}\label{sec:deter_qb_err}

For     $\qmat{A} \in\bbR^{m\times n}$, without loss of generality   assume that  $m\geq n$. For any $\varrho\geq r$ with $r$ the target rank,   partition $\qmat{A}$   as 
\begin{align}\label{eq:A_SVD}
    \begin{array}{c@{}c@{}c@{}c@{}c}
        &   & \begin{array}{cc} \scriptstyle \varrho & \scriptstyle n-\varrho \end{array} &    \begin{array}{c}
               \scriptstyle n \end{array} & \\[-2pt]  
               \qmat{A}=\qmat{U}\boldsymbol{\Sigma}\qmat{V}^{T}= &      \qmat{U} &
           \left[
           \begin{array}{cc}
               \boldsymbol{\Sigma}_1 &  \\ 
                & \boldsymbol{\Sigma}_2
           \end{array}
           \right] &
               \left[
               \begin{array}{c}
                   \mathbf{V}_1^{T}  \\
                   \mathbf{V}_2^{T}
               \end{array}
               \right] &     
           \begin{array}{c}
               \scriptstyle \varrho \\[2pt]
               \scriptstyle n-\varrho
           \end{array}
       \end{array},
\end{align}  
     where $\qmat{U}\in\bbR^{m\times n}$, $\qmat{V}\in\bbR^{n\times n}$ are orthonormal, $\boldsymbol{\Sigma}\in\bbR^{n\times n}$,  $\boldsymbol{\Sigma}_1\in\bbR^{\varrho \times \varrho }$, $\boldsymbol{\Sigma}_2 \in\bbR^{(n-\varrho )\times (n-\varrho )}$ are diagonal matrices, and 
     $\qmat{V}_1\in \bbR^{n\times \varrho },\qmat{V}_2\in\bbR^{n\times (n-\varrho )}$. The singular values in $\boldsymbol{\Sigma}$ are arranged in a descending order.
    
    Let   $\bdOmega\in\bbR^{n\times s}, \bdPhi\in\bbR^{n\times l}$ be two test matrices with $l> s > \varrho $. We decompose $\bdOmega$ and $\bdPhi$ determined by the partition of $\qmat{V}$ as follows:
    \begin{equation}\label{eq:notations_Omega_Psi_D}
        \begin{split}
        &\bdOmega_1=\qmat{V}_1^{T}\bdOmega \in\bbR^{\varrho \times s}, ~ \bdOmega_2=\qmat{V}_2^{T}\bdOmega\in\bbR^{(n-\varrho )\times s },\\
        &  \bdPhi_1=\qmat{V}_1^{T}\bdPhi\in\bbR^{\varrho \times l},\, \bdPhi_2=\qmat{V}_2^{T}\bdPhi\in\bbR^{(n-\varrho )\times l}.
        \end{split}
    \end{equation}
    In addition, we introduce the auxiliary notation  
    \begin{align}\label{eq:notation_B}
        \quad \qmat{H}=\boldsymbol{\Sigma}_2^2\bdOmega_2\bdOmega_1^\dagger\boldsymbol{\Sigma}_1^{-2}.
    \end{align}
    Let $\qmat{Y}=\qmat{A}\bdOmega\in\bbR^{m\times s}$, $\qmat{Z}=\qmat{A}\bdPhi\in\bbR^{m\times l}$, and $\hat{\qmat{Y}}=\qmat{Z}\qmat{Z}^{T}\qmat{Y}\in \bbR^{m\times s}$. 
   The projection error with respect to   $\qmat{I}-P_{\hat{\qmat{Y}}}=\qmat{I}- \hat{\qmat{Y}}\hat{\qmat{Y}}^\dag=\qmat{I}- \hat{\qmat{Y}}(\hat{\qmat{Y}}^{T}\hat{\qmat{Y}})^\dag\hat{\qmat{Y}}^{T}$ is given as follows. 

\begin{theorem}\label{thm:projection_error_sketchPower}
    Let $\qmat{A}$ be partitioned as in \eqref{eq:A_SVD} and recall the notations   introduced in \eqref{eq:notations_Omega_Psi_D} and \eqref{eq:notation_B}. Assume that   $\bdOmega_1,\bdPhi_1$ have full row rank.   
    For $\hat{\qmat{Y}}=\qmat{Z}\qmat{Z}^{T} \qmat{Y}=\qmat{A}\bdPhi(\qmat{A}\bdPhi)^{T}\qmat{A}\bdOmega$, we have the following estimation:
    \begin{align}\label{eq:projection_error_sketchPower}
        \normP{\qmat{A}-\hat{\qmat{Y}}\hat{\qmat{Y}}^\dagger \qmat{A}}^2\leq& \normP{\boldsymbol{\Sigma}_2}^2+\normP{\boldsymbol{\Sigma}_2\bdPhi_2\bdPhi_1^\dagger+\boldsymbol{\Sigma}_2\bdPhi_2\bdPhi_2^{T}\qmat{H} (\bdPhi_1\bdPhi_1^{T} )^{-1}}^2\\
        &\cdot\normSpectral{\left(\qmat{I}+\bdPhi_1\bdPhi_2^{T}\qmat{H} (\bdPhi_1\bdPhi_1^{T} )^{-1}\right)^{-1}}^2 \label{eq:inverseTerm}
    \end{align}
by assuming the invertibility of $\qmat{I}+\bdPhi_1\bdPhi_2^{T}\qmat{H} (\bdPhi_1\bdPhi_1^{T} )^{-1}$. 
\end{theorem}

When $\bdOmega$ and $\bdPhi$ are standard Gaussian, $\bdOmega_1,\bdPhi_1$ are also standard Gaussian due to the rotational invariance, which together with their sizes imply that they have full row rank with probability $1$. Proposition \ref{prop:inv_F1}   shows that $\qmat{I}+\bdPhi_1\bdPhi_2^{T}\qmat{H} (\bdPhi_1\bdPhi_1^{T} )^{-1}$ is invertible with probability $1$ in the same situation.

\begin{proof} 
    The proof can be divided into four steps. The first step is to consider a trivial case. The second is to remove $\qmat{U}$ from the projection error. The third step is to find an acceptable replacement for $\hat{\qmat{Y}}$ to further simplify the analysis, and the last one is to decompose the terms.

    {\bf Step 1: Trivial case.}  First we consider the case that $\boldsymbol{\Sigma}_1$ is singular. Thus it must hold that $\boldsymbol{\Sigma}_2 = 0$, and $\boldsymbol{\Sigma}_1$ can be further written as $\boldsymbol{\Sigma}_1 = \begin{bmatrix} \boldsymbol{\Sigma}_0 & \\ & 0\end{bmatrix}$ with $\boldsymbol{\Sigma}_0\in\bbR^{r_0\times r_0}$   invertible and $r_0$ the rank of $\qmat{A}$. In this case,   
    the SVD of $\qmat{A}$   can be expressed as  $\qmat{A} = \qmat{U}_1\boldsymbol{\Sigma}_1\qmat{V}_1^{T} = \qmat{U}_0 \boldsymbol{\Sigma}_0 \qmat{V}_0^{T}$, where $\qmat{U}_0$ and $\qmat{V}_0$ are composed of the columns in $\qmat{U}_1,\qmat{V}_1$ corresponding to $\boldsymbol{\Sigma}_0$.  
    Therefore, $\qmat{V}_0^{T}\bdPhi$ and $\qmat{V}_0^{T}\bdOmega$, as   sub-matrices of $\bdPhi_1=\qmat{V}_1^{T}\bdPhi$ and $\bdOmega_1=\qmat{V}_1^{T}\bdOmega$,   have full row rank due to the assumption.  Thus 
    \begin{align*}
        {\rm range}( \hat{\qmat{Y}})&=\rangemat{\qmat{A}\bdPhi(\qmat{A}\bdPhi)^{T}\qmat{A}\bdOmega}=\rangemat{\qmat{U}_0\boldsymbol{\Sigma}_0 \qmat{V}_0^{T}\bdPhi\bdPhi^{T}\qmat{V}_0\boldsymbol{\Sigma}_0^2\qmat{V}_0^{T}\bdOmega}\\
        &=\rangemat{\qmat{U}_0}=\rangemat{\qmat{A}},
    \end{align*}
    where the third equality holds because $\boldsymbol{\Sigma}_0 \qmat{V}_0^{T}\bdPhi\bdPhi^{T}\qmat{V}_0\boldsymbol{\Sigma}_0^2\qmat{V}_0^{T}\bdOmega$ has full  row rank. Therefore,   $\normP{(\qmat{I}- P_{\hat{\qmat{Y}}})\qmat{A}}=0$ and the estimation naturally holds.  Thus in the remaining proof, we only consider   that  $\boldsymbol{\Sigma}_1$ is invertible. 

    {\bf Step 2: Reducing $\qmat{U}$.}  We then remove $\qmat{U}$   from the projection error. Note that 
    \begin{small}
    \begin{align*}
        \hat{\qmat{Y}}=\qmat{A}\bdPhi(\qmat{A}\bdPhi)^{T}\qmat{Y}=\qmat{A}\bdPhi(\qmat{A}\bdPhi)^{T}\qmat{A}\bdOmega=\qmat{U}\boldsymbol{\Sigma} \qmat{V}^{T}\bdPhi\bdPhi^{T}\qmat{V}\boldsymbol{\Sigma}^2\qmat{V}^{T}\bdOmega =\qmat{U}\qmat{D}\qmat{V}^{T}\bdOmega \in\bbR^{m\times s},
    \end{align*}
\end{small}
    where the randomized embedding properties are defined by the following notation: 
    \begin{align}\label{eq:notations_Omega_Psi_D_1}
       \qmat{D}:=\boldsymbol{\Sigma}\qmat{V}^{T}\bdPhi\bdPhi^{T}\qmat{V}\boldsymbol{\Sigma}^2 \in\bbR^{n\times n}.
    \end{align}
    Here   $\bdPhi$ generates a well approximation to the invariant subspace corresponding to $\boldsymbol{\Sigma}_1$, such that the dominant part of $\qmat{D}$ is close to $\boldsymbol{\Sigma}_1^3$ (note that in the standard power iteration, $\qmat{D}=\boldsymbol{\Sigma}\qmat{V}^{T}\qmat{V}\boldsymbol{\Sigma}^2 = \boldsymbol{\Sigma}^3$). 
As $P_{\hat{\qmat{Y}}}=\hat{\qmat{Y}}(\hat{\qmat{Y}}^{T}\hat{\qmat{Y}})^\dag\hat{\qmat{Y}}^{T}$, we  have 
    \begin{align}\label{eq:proof_q=1_deterministic_1}
        \normP{(\qmat{I}-P_{\hat{\qmat{Y}}})\qmat{A}}&=
      \normP{\qmat{U}\boldsymbol{\Sigma}\qmat{V}^{T}-\qmat{U}\qmat{D}\qmat{V}^{T}\bdOmega(\bdOmega^{T}\qmat{V}\qmat{D}^{T}\qmat{D}\qmat{V}^{T}\bdOmega)^{\dag}\bdOmega^{T}\qmat{V}\qmat{D}^{T}\qmat{U}^{T}\qmat{U}\boldsymbol{\Sigma}\qmat{V}^{T}}\nonumber\\
        &=\normP{\boldsymbol{\Sigma}-\qmat{D}\qmat{V}^{T}\bdOmega(\bdOmega^{T}\qmat{V}\qmat{D}^{T}\qmat{D}\qmat{V}^{T}\bdOmega)^{\dag}\bdOmega^{T}\qmat{V}\qmat{D}^{T}\boldsymbol{\Sigma}}\nonumber\\
        &=\normP{(\qmat{I}-P_{\qmat{D}\qmat{V}^{T}\bdOmega})\boldsymbol{\Sigma}},
    \end{align}
    where the third equality is due to $\qmat{U}^{T}\qmat{U}=\qmat{I}$ and unitary invariance of the Schattern-$p$ norm. Thus   it suffices to bound $\normP{(\qmat{I}-P_{\qmat{D}\qmat{V}^{T}\bdOmega})\boldsymbol{\Sigma}}$ instead.

 {\bf Step 3: Further simplification.} 
 Recalling the notations in \eqref{eq:notations_Omega_Psi_D} and \eqref{eq:notations_Omega_Psi_D_1}, we may express  $\qmat{D}\qmat{V}^{T}\bdOmega$ as
 \begin{align}\label{eq:DVOMega}\qmat{D}\qmat{V}^{T}\bdOmega &= \boldsymbol{\Sigma}(\qmat{V}^{T}\bdPhi)(\bdPhi^{T}\qmat{V})\boldsymbol{\Sigma}^2 (\qmat{V}^{T} \bdOmega) \nonumber \\
&=        \begin{bmatrix}
       \boldsymbol{\Sigma}_1\bdPhi_1\\
       \boldsymbol{\Sigma}_2\bdPhi_2
   \end{bmatrix}\begin{bmatrix}
       \bdPhi_1^{T}\boldsymbol{\Sigma}_1^2 & \bdPhi_2^{T}\boldsymbol{\Sigma}_2^2
   \end{bmatrix}\begin{bmatrix}
       \bdOmega_1\\
       \bdOmega_2
   \end{bmatrix} \in\bbR^{n\times s}.
\end{align}
 Expression   \eqref{eq:DVOMega} shows that directly dealing with $\qmat{D}\qmat{V}^{T}\bdOmega$ is too complicated due to the coupling of $\bdPhi_1,\bdPhi_2,\bdOmega_1,\bdOmega_2$. Thus we find an acceptable replacement of $\qmat{D}\qmat{V}^{T}\bdOmega$, preparing for decoupling. To this end, first denote   
      \begin{align}
        \begin{bmatrix}
            \qmat{F}_1\\
            \qmat{F}_2
        \end{bmatrix}&:=\qmat{D}\qmat{V}^{T}\bdOmega(\bdOmega_1^\dagger\boldsymbol{\Sigma}_1^{-2} (\bdPhi_1\bdPhi_1^{T})^{-1}\boldsymbol{\Sigma}_1^{-1}) \label{eq:relation_F1F2_DVOmega} \in\bbR^{n\times \varrho },
    \end{align}
    where the    invertibility of $\bdPhi_1\bdPhi_1^{T}$   follows from the assumption. Here $\qmat{F}_1\in\bbR^{\varrho \times \varrho }$ and $\qmat{F}_2\in\bbR^{(n-\varrho )\times \varrho }$. 
    It follows from basic linear algebra that  
    the range spanned by $\qmat{D}\qmat{V}^{T}\bdOmega$ and $\begin{bmatrix}
     \qmat{F}_1\\
     \qmat{F}_2
 \end{bmatrix}$ satisfy:
    \begin{align}\label{eq:range_F1F2_DVOmega}
         \rangemat{\begin{bmatrix}
            \qmat{F}_1\\ \qmat{F}_2
            \end{bmatrix}} \subseteq \rangemat{\qmat{D}\qmat{V}^{T}\bdOmega}.
    \end{align}
   On the other hand, from   \eqref{eq:DVOMega} we have  
    \begin{align*}
        \begin{bmatrix}
            \qmat{F}_1\\
            \qmat{F}_2
        \end{bmatrix}    & =\begin{bmatrix}
            \boldsymbol{\Sigma}_1\bdPhi_1\bdPhi_1^{T}\boldsymbol{\Sigma}_1^2\bdOmega_1+\boldsymbol{\Sigma}_1\bdPhi_1\bdPhi_2^{T}\boldsymbol{\Sigma}_2^2\bdOmega_2\\
            \boldsymbol{\Sigma}_2\bdPhi_2\bdPhi_1^{T}\boldsymbol{\Sigma}_1^2\bdOmega_1+\boldsymbol{\Sigma}_2\bdPhi_2\bdPhi_2^{T}\boldsymbol{\Sigma}_2^2\bdOmega_2
        \end{bmatrix}\bdOmega_1^\dagger\boldsymbol{\Sigma}_1^{-2}\left(\bdPhi_1\bdPhi_1^{T}\right)^{-1}\boldsymbol{\Sigma}_1^{-1}\\
        &=\begin{bmatrix}
            \qmat{I}_{\varrho}+\boldsymbol{\Sigma}_1\bdPhi_1\bdPhi_2^{T}\boldsymbol{\Sigma}_2^2\bdOmega_2\bdOmega_1^\dagger\boldsymbol{\Sigma}_1^{-2}\left(\bdPhi_1\bdPhi_1^{T}\right)^{-1}\boldsymbol{\Sigma}_1^{-1}\\
            \boldsymbol{\Sigma}_2\bdPhi_2\bdPhi_1^\dagger\boldsymbol{\Sigma}_1^{-1}+\boldsymbol{\Sigma}_2\bdPhi_2\bdPhi_2^{T}\boldsymbol{\Sigma}_2^2\bdOmega_2\bdOmega_1^\dagger\boldsymbol{\Sigma}_1^{-2}\left(\bdPhi_1\bdPhi_1^{T}\right)^{-1}\boldsymbol{\Sigma}_1^{-1}
        \end{bmatrix}\\
       & =\begin{bmatrix}
        \qmat{I}_{\varrho}+\boldsymbol{\Sigma}_1\bdPhi_1\bdPhi_2^{T}\qmat{H}\left(\bdPhi_1\bdPhi_1^{T}\right)^{-1}\boldsymbol{\Sigma}_1^{-1}\\
            \boldsymbol{\Sigma}_2\bdPhi_2\bdPhi_1^\dagger\boldsymbol{\Sigma}_1^{-1}+\boldsymbol{\Sigma}_2\bdPhi_2\bdPhi_2^{T}\qmat{H}\left(\bdPhi_1\bdPhi_1^{T}\right)^{-1}\boldsymbol{\Sigma}_1^{-1}
        \end{bmatrix}
        ,
    \end{align*}
   where the second equality holds because $\bdOmega_1$ has full row rank. 
   Observe that $\Expectation{ \bdPhi_2^{T}}=0$ and $\boldsymbol{\Sigma}_2$ is small,   indicating that $\qmat{F}_1$ is close to $\qmat{I}_{\varrho}$. 
      The assumption and the discussion in step 1 shows that $\qmat{F}_1$ is invertible.  
Right multiply   $\qmat{F}_1^{-1}$ to $\begin{bmatrix}
        \qmat{F}_1\\
        \qmat{F}_2
    \end{bmatrix}$ to obtain 
     \begin{align}\label{eq:defZ}\qmat{N}:=\begin{bmatrix}
        \qmat{F}_1\\
        \qmat{F}_2
    \end{bmatrix}\qmat{F}_1^{-1}=\begin{bmatrix}
        \qmat{I}_{\varrho}\\
        \qmat{F}
    \end{bmatrix}\in\bbR^{n\times \varrho },\end{align} 
     where $\qmat{F}:=\qmat{F}_2\qmat{F}_1^{-1}$.   \eqref{eq:range_F1F2_DVOmega} gives
     $
        \rangemat{\qmat{N}}=\rangemat{\begin{bmatrix}
            \qmat{F}_1\\ \qmat{F}_2
            \end{bmatrix}}\subseteq \rangemat{\qmat{D}\qmat{V}^{T}\bdOmega},
     $
     and so  \begin{align}\label{eq:bound_ZSigma} \normP{(\qmat{I}-P_{\qmat{D}\qmat{V}^{T}\bdOmega})\boldsymbol{\Sigma}}\leq \normP{(\qmat{I}-P_{\qmat{N}})\boldsymbol{\Sigma}} 
     .\end{align}
Thus it suffices to estimate $\normP{(\qmat{I}-P_{\qmat{N}})\boldsymbol{\Sigma}} $ to obtain the assertion \eqref{eq:projection_error_sketchPower}. 
    
 {\bf Step 4: Decoupling.}  
 We first provide some intuition regarding $\qmat{N}$ in \eqref{eq:defZ}.  We have previously argued that $\qmat{F}_1$ is close to $\qmat{I}_{\varrho}$, suggesting that $\qmat{F}\approx \qmat{F}_2$. In expectation, the first term in $\qmat{F}_2$ vanishes, while the second term is of order $O(\boldsymbol{\Sigma}_2^3)$, implying that $\qmat{F}$ is also small. Consequently,   $\qmat{N}$ is close to $\begin{bmatrix}
    \qmat{I}_{\varrho}\\
    0
\end{bmatrix}$. Note that $\rangemat{\begin{bmatrix}
    \qmat{I}_{\varrho}\\
    0
\end{bmatrix}}$ primarily reflects the information of   $\boldsymbol{\Sigma}_1$. Therefore,   we expect that $\rangemat{\qmat{N}}$ can do  it similarly.  

Considering   $\normP{(\qmat{I}-P_{\qmat{N}})\boldsymbol{\Sigma}} $, the following inequalities hold:
\begin{align} \label{eq:1_deterministic_3}
    \normP{(\qmat{I}-P_{\qmat{N}})\boldsymbol{\Sigma}}&=\normP[p/2]{(\qmat{I}-P_{\qmat{N}})\boldsymbol{\Sigma}^2(\qmat{I}-P_{\qmat{N}})}\nonumber\\
    &\leq \normP[p/2]{(\qmat{I}-P_{\qmat{N}})\begin{bmatrix}
        0& \\
        &\boldsymbol{\Sigma}_2^2
    \end{bmatrix}(\qmat{I}-P_{\qmat{N}})}\!+\normP[p/2]{(\qmat{I}-P_{\qmat{N}})\begin{bmatrix}
        \boldsymbol{\Sigma}_1^2& \\
        &0
    \end{bmatrix}(\qmat{I}-P_{\qmat{N}})} \nonumber\\
        &\leq \normP{\boldsymbol{\Sigma}_2}^2+\normP[p/2]{\begin{bmatrix} \boldsymbol{\Sigma}_1&\\
            &0
        \end{bmatrix}(\qmat{I}-P_{\qmat{N}})\begin{bmatrix} \boldsymbol{\Sigma}_1&\\
            &0
        \end{bmatrix}},
\end{align}
where the last inequality comes from  $\normSpectral{\qmat{I}-P_{\qmat{N}}}\leq 1$ and the relation between $\normP[p/2]{\cdot}$ and $\normP[p]{\cdot}$ as $\normP[p]{\qmat{M}}^2 = \normP[p/2]{\qmat{M}^{T}\qmat{M}} = \normP[p/2]{\qmat{M}\qmat{M}^{T}}$.  The second term of the last line above can be written as a parallel sum:
\begin{align}\label{eq:1_deterministic_2}
    &\normP[p/2]{\begin{bmatrix} \boldsymbol{\Sigma}_1& \nonumber\\
        &0
    \end{bmatrix}(\qmat{I}-P_{\qmat{N}})\begin{bmatrix} \boldsymbol{\Sigma}_1& \nonumber\\
        &0
    \end{bmatrix}}=\normP[p/2]{\begin{bmatrix} \boldsymbol{\Sigma}_1& \nonumber\\
        &0
    \end{bmatrix}(\qmat{I}-P_{\qmat{N}\boldsymbol{\Sigma}_1})\begin{bmatrix} \boldsymbol{\Sigma}_1&\\
        &0
    \end{bmatrix}}\\
    &=\normP[p/2]{\boldsymbol{\Sigma}_1^2-\boldsymbol{\Sigma}_1^2(\boldsymbol{\Sigma}_1^2+\boldsymbol{\Sigma}_1\qmat{F}^{T} \qmat{F}\boldsymbol{\Sigma}_1)^\dag\boldsymbol{\Sigma}_1^2}=\normP[p/2]{\boldsymbol{\Sigma}_1^2:\boldsymbol{\Sigma}_1\qmat{F}^{T} \qmat{F}\boldsymbol{\Sigma}_1}\nonumber\\
    &\leq \normP[p/2]{\boldsymbol{\Sigma}_1^2}:\normP[p/2]{\boldsymbol{\Sigma}_1\qmat{F}^{T} \qmat{F}\boldsymbol{\Sigma}_1}= \normP[p/2]{\boldsymbol{\Sigma}_1\qmat{F}^{T} \qmat{F}\boldsymbol{\Sigma}_1}:\normP[p/2]{\boldsymbol{\Sigma}_1^2}\leq\normP{\qmat{F}\boldsymbol{\Sigma}_1}^2,    
\end{align}
where the first equality is due to the invertibility of $\boldsymbol{\Sigma}_1$; the two inequalities   are from   Lemma \ref{lem:parallel_sum}.
Then 
    \begin{align*}
        &\normP{\qmat{F}\boldsymbol{\Sigma}_1}^2=\normP{\qmat{F}_2\qmat{F}_1^{-1}\boldsymbol{\Sigma}_1}^2\\
        &\!=\!\normP{ \!\left(\boldsymbol{\Sigma}_2\bdPhi_2\bdPhi_1^\dagger\boldsymbol{\Sigma}_1^{-1}\!\!+\!\boldsymbol{\Sigma}_2\bdPhi_2\bdPhi_2^{T}\qmat{H}( \bdPhi_1\bdPhi_1^{T})^{-1}\boldsymbol{\Sigma}_1^{-1}\!\right)\!\!\left(\qmat{I}\!+\!\boldsymbol{\Sigma}_1\bdPhi_1\bdPhi_2^{T}\qmat{H} (\!\bdPhi_1\bdPhi_1^{T} )^{-1}\boldsymbol{\Sigma}_1^{-1}\!\right)^{-1}\!\!\boldsymbol{\Sigma}_1\!}^2\\
        &=\normP{ \left(\boldsymbol{\Sigma}_2\bdPhi_2\bdPhi_1^\dagger+\boldsymbol{\Sigma}_2\bdPhi_2\bdPhi_2^{T}\qmat{H} (\bdPhi_1\bdPhi_1^{T} )^{-1}\right)\left(\qmat{I}+\bdPhi_1\bdPhi_2^{T}\qmat{H} (\bdPhi_1\bdPhi_1^{T} )^{-1}\right)^{-1}}^2\\
        &\leq \normP{\boldsymbol{\Sigma}_2\bdPhi_2\bdPhi_1^\dagger+\boldsymbol{\Sigma}_2\bdPhi_2\bdPhi_2^{T}\qmat{H} (\bdPhi_1\bdPhi_1^{T} )^{-1}}^2\normSpectral{ (\qmat{I}+\bdPhi_1\bdPhi_2^{T}\qmat{H} (\bdPhi_1\bdPhi_1^{T} )^{-1} )^{-1}}^2.
    \end{align*}
This together with \eqref{eq:proof_q=1_deterministic_1}, \eqref{eq:bound_ZSigma}, \eqref{eq:1_deterministic_3}, and \eqref{eq:1_deterministic_2}     gives the assertion. 
\end{proof}


\begin{proposition} \label{prop:inv_F1}
   For $\bdOmega$ and $\bdPhi$ being standard Gaussian matrices and recall the notations  of $\bdOmega_1,\bdOmega_2,\bdPhi_1,\bdPhi_2$   in \eqref{eq:notations_Omega_Psi_D}, $\qmat{F}_1 = \qmat{I}+\boldsymbol{\Sigma}_1\bdPhi_1\bdPhi_2^{T}\boldsymbol{\Sigma}_2^2\bdOmega_2\bdOmega_1^\dagger\boldsymbol{\Sigma}_1^{-2} (\bdPhi_1\bdPhi_1^{T} )^{-1}\boldsymbol{\Sigma}_1^{-1}\in\bbR^{\varrho \times \varrho }$ is invertible with probability $1$. 
\end{proposition}
\begin{proof}
     It suffices to prove that each eigenvalue of $\bdPhi_1\bdPhi_2^{T}\boldsymbol{\Sigma}_2^2\bdOmega_2\bdOmega_1^\dagger\boldsymbol{\Sigma}_1^{-2} (\bdPhi_1\bdPhi_1^{T} )^{-1}$ equals   $-1$ with probability $0$.  Let $\lambda_i(\cdot)$ denote an eigenvalue of a matrix.   For a function
    \begin{align*}
        f(\qmat{X})=\lambda_i(\qmat{S}\qmat{X}\qmat{T}),
    \end{align*}
    it is continuous with respect to the entries of $\qmat{X}$, where $\qmat{S},\qmat{T},\qmat{X}$ are   matrices of proper sizes. For a matrix, the zeros of the characteristic polynomial exhibit continuous dependence on the coefficients, and their quantitative description is provided in \cite{BoundsVariationMatrix1994}. Thus $f(\qmat{X})$   is   continuous. Therefore, either $f(\qmat{X})$ is a constant, or its image contains uncountably many points in  $\bbC$.

Now take $\qmat{X}=\bdOmega_2$, $\qmat{S}=\bdPhi_1\bdPhi_2^{T}\boldsymbol{\Sigma}_2^2$, and $\qmat{T}=\bdOmega_1^\dagger\boldsymbol{\Sigma}_1^{-2} (\bdPhi_1\bdPhi_1^{T} )^{-1}$. It follows from that $\bdOmega$ and $\bdPhi$ are standard Gaussian and the unitary invariance property that $\bdOmega_1,\bdOmega_2,\bdPhi_1,\bdPhi_2$  are independent standard Gaussian matrices; in particular, $\qmat{X}=\bdOmega_2$ is independent of $\qmat{S}$ and $\qmat{T}$. Now consider $\bdOmega_1,\bdPhi_1,\bdPhi_2$ as   events in their related probability spaces and $\qmat{X}$ a   variable. Clearly, $f(\qmat{X})$ can take value $0$ provided $\qmat{X}=0$; as a result, if $f(\qmat{X})$   takes value $-1$, then $f(\qmat{X})$ is not a constant,  and the argument above shows that the image of $f(\cdot)$ with respect to the probability space of $\bdOmega_2$ must contain uncountable many points in $\bbC$. 
 Therefore, with respect to the product probability space of $(\bdOmega_1,\bdOmega_2,\bdPhi_1,\bdPhi_2)$, the image of $f(\cdot)$ also has uncountable many points in $\bbC$,
which means that $f(\qmat{X})=\lambda_i(\boldsymbol{\Sigma}_1\bdPhi_1\bdPhi_2^{T}\boldsymbol{\Sigma}_2^2\bdOmega_2\bdOmega_1^\dagger\boldsymbol{\Sigma}_1^{-2} (\bdPhi_1\bdPhi_1^{T} )^{-1}\boldsymbol{\Sigma}_1^{-1})=-1$ with probability $0$. This completes the proof. 
\end{proof}

\subsection{Averaged projection error}\label{sec:prob_qb_err}

We derive the bound of Theorem \ref{thm:projection_error_sketchPower} in expectation. 
We first estimate $\mathbb E \|\boldsymbol{\Sigma}_2\bdPhi_2\bdPhi_1^\dagger+\boldsymbol{\Sigma}_2\bdPhi_2\bdPhi_2^{T}\qmat{H} (\bdPhi_1\bdPhi_1^{T} )^{-1} \|_p^2 $, which can be upper bounded by $  2 \mathbb E \|\boldsymbol{\Sigma}_2\bdPhi_2\bdPhi_1^\dagger\|_p^2 + 2\mathbb E\|\boldsymbol{\Sigma}_2\bdPhi_2\bdPhi_2^{T}\qmat{H} (\bdPhi_1\bdPhi_1^{T} )^{-1}\|_p^2$. The first can be bounded by   Corollary \ref{col:SigmaPhi2Phi1_second_moment_bound}. For the second one we have:
\color{black}
\begin{theorem}\label{thm:expectation_projectionError_partOne}
  Let   $\bdOmega,\bdPhi$ be standard Gaussian. If $l>\varrho +4$ and $s\geq \varrho+2$, then 
    \begin{small}
    \begin{align*} 
        &~~~~\mathbb E{\normP{ \boldsymbol{\Sigma}_2\bdPhi_2\bdPhi_2^{T}\qmat{H} (\bdPhi_1\bdPhi_1^{T} )^{-1}}^2} 
        \leq  \frac{\delta}{2} \frac{\sum_{j=\varrho +1}\sigma_j^6}{\sigma_{\varrho}^4} + \frac{\mu}{2} \tau_{\varrho +1}^2(\qmat{A})\frac{\sum_{j=\varrho +1}\sigma_j^4}{\sigma_{\varrho}^4},
    \end{align*}
\end{small}
where 
\begin{small}
\begin{equation}\label{eq:parameters_eps_delta_mu}  \begin{split}
    \delta =\frac{e^2(s+\varrho )\varrho (l-1)(l^2+2l)\cdot }{ (s-\varrho )^2(l-\varrho )(l-\varrho -1)(l-\varrho -3)} ,~
    \mu =\frac{ e^2(s+\varrho )\varrho (l-1)l\cdot  }{(s-\varrho )^2(l-\varrho )(l-\varrho -1)(l-\varrho -3)}.
\end{split}
\end{equation}
\end{small}
\end{theorem}


\begin{proof}
    The Schatten $p$-norm is decreasing with respect to $p$, so we can bound it by that with $p=2$, i.e, the Frobenius norm.    
 Recall that $\qmat{H}=\boldsymbol{\Sigma}_2^2\bdOmega_2\bdOmega_1^\dagger\boldsymbol{\Sigma}_1^{-2}$; by the independence of $\bdOmega_1,\bdOmega_2,\bdPhi_1,\bdPhi_2$ and using   Lemma \ref{lem:SGT}, 
 \begin{small}
 \begin{align*}
 \!\!   \Expectation[\bdPhi,\bdOmega]{\normF{\boldsymbol{\Sigma}_2\bdPhi_2\bdPhi_2^{T}\qmat{H}(\bdPhi_1\bdPhi_1^{T})^{-1}}^2}
 &=\Expectation[\bdPhi,\bdOmega_1]{\Expectation[\bdOmega_2]{\normF{\boldsymbol{\Sigma}_2\bdPhi_2\bdPhi_2^{T}\qmat{H}(\bdPhi_1\bdPhi_1^{T})^{-1}}^2|\bdOmega_1,\bdPhi}}\\
 &= \Expectation[\bdPhi,\bdOmega_1]{\normF{\boldsymbol{\Sigma}_2\bdPhi_2\bdPhi_2^{T}\boldsymbol{\Sigma}_2^2}^2\cdot \normF{\bdOmega_1^\dagger\boldsymbol{\Sigma}_1^{-2}(\bdPhi_1\bdPhi_1^{T})^{-1}}^2}\\
 &\leq \Expectation[\bdPhi,\bdOmega_1]{\normF{\boldsymbol{\Sigma}_2\bdPhi_2\bdPhi_2^{T}\boldsymbol{\Sigma}_2^2}^2\!\cdot\!\normSpectral{\bdOmega_1^\dagger}^2\normSpectral{\boldsymbol{\Sigma}_1^{-2}}^2 \normF{(\bdPhi_1\bdPhi_1^{T})^{-1}}^2\!}\\
 &=  \sigma_{\varrho}^{-4}   \mathbb E \normF{\boldsymbol{\Sigma}_2\bdPhi_2\bdPhi_2^{T}\boldsymbol{\Sigma}_2^2}^2  \cdot \mathbb E    \normSpectral{\bdOmega_1^\dagger}^2  \cdot \mathbb E   \normF{(\bdPhi_1\bdPhi_1^{T})^{-1}}^2. 
 \end{align*}
\end{small}
The three expectations above are computed in the sequel. 


      Let $\qmat{W}=\bdPhi_2\bdPhi_2^{T}$; then $\qmat{W}_{i,j}=\sum_{k=1}^{l}\phi_{i,k}{\phi_{j,k}}$ where $\phi_{i,j}$ are independent   Gaussian variables. We have:
    \begin{align*}
        \mathbb E{\qmat{W}_{i,i}^2}&=\mathbb E{(\sum_{k=1}^{l}\phi_{i,k}{\phi_{i,k}})^2}=\mathbb E{\sum_{k=1}^{l}(\phi_{i,k}{\phi_{i,k}})^2+\mathbb E\sum_{k\neq p}(\phi_{i,k}{\phi_{i,k}})(\phi_{i,p}{\phi_{i,p}})}\\
        &=
        \sum_{k=1}^{l}\mathbb E{\phi_{i,k}^4}+\sum_{k\neq p}\mathbb E{\phi_{i,k}^2}\mathbb E{\phi_{i,p}^2}=3l+ {l^2-l} = {l^2+2l} ;\\
        \mathbb E{ \qmat{W}_{i,j}^2}&=\mathbb E{( \sum_{k=1}^{l}\phi_{i,k}\phi_{j,k})^2}=
        l,~i\neq j. 
    \end{align*}
    Thus   $\mathbb E {\normF{\boldsymbol{\Sigma}_2\bdPhi_2\bdPhi_2^{T}\boldsymbol{\Sigma}_2^2}^2}=\sum_{i,j}\sigma_i^2 \mathbb E \qmat{W}_{i,j}^2\sigma_j^4= {(l^2+2l)} \normF{\boldsymbol{\Sigma}_2^3}^2+l\normF{\boldsymbol{\Sigma}_2}^2\normF{\boldsymbol{\Sigma}_2^2}^2$. 
    Lemma \ref{lem:inverse_moment_spectral} gives 
  $
        \mathbb E\normSpectral{\bdOmega_1^\dagger}^2 \leq\frac{e^2(s+\varrho )}{2(s-\varrho )^2}.
    $ 
  Lemma \ref{lem:inverse_moment_F} indicates that  $$\mathbb E{\normF{(\bdPhi_1\bdPhi_1^{T})^{-1}}^2}\leq\frac{\varrho (l - 1)}{(l - \varrho )(l - \varrho  - 1)(l - \varrho  - 3)}$$ when $l-\varrho \geq 4$.
     Combining the pieces yields the assertion. 
\end{proof}
\color{black} 

\color{black} The term $\| (\qmat{I}+\bdPhi_1\bdPhi_2^{T}\qmat{H} (\bdPhi_1\bdPhi_1^{T} )^{-1} )^{-1}\|$ in \eqref{eq:inverseTerm}  cannot be bounded in expectation because   there exists a positive Lebesgue  measure around singularity, facing a similar trouble as that in the analysis of \cite[Thm. 3.1 and Remark 3.3]{nakatsukasa2023randomized}. As an alternative, we  
    show that it is not large  with high probability     when $\qmat{A}$ does not have too flattened spectrum. The  bound is related to   classical results in \cite{FindingStructureHalko,tropp2023RandomizedAlgorithms,martinsson2020RandomizedNumerical}. Define the   events: \color{black}
\begin{align*}
    E_1&=\bigdakuohao{\bdPhi:\normF{(\bdPhi_1^{T})^\dagger\bdPhi_2^{T}\boldsymbol{\Sigma}_2^2}<e_1~\text{and}~ \normSpectral{(\bdPhi_1^{T})^\dagger\bdPhi_2^{T}\boldsymbol{\Sigma}_2^2}<e_2},\\
    E_2&=\bigdakuohao{\bdOmega:\normF{\bdOmega_1^\dagger}<e_3 ~\text{and}~ \normSpectral{\bdOmega_1^\dagger}<e_4},\\
    E_3&=\bigdakuohao{\bdPhi:\kappa({\bdPhi_1\bdPhi_1^{T}})<e_5^2} \quad(\kappa(\cdot)~\text{stands for the condition number}),
\end{align*}
where
\begin{align}\label{eq:parametersE}
   & e_1=t \cdot \sqrt{\frac{3\varrho}{l-\varrho +1}}\normF{\boldsymbol{\Sigma}_2^2} 
    + u t \cdot \frac{e \sqrt{l}}{l-\varrho +1} \cdot \sigma_{\varrho +1}^2,\nonumber\\
   & e_2=\left[
            t \cdot \sqrt{\frac{3\varrho}{l-\varrho +1}}
            \sigma_{\varrho +1}^2 
            + t \cdot \frac{e\sqrt{l}}{l-\varrho +1} 
            \normF{\boldsymbol{\Sigma}_2^2}
            \right]
            + ut \cdot \frac{e\sqrt{l}}{l-\varrho +1} \sigma_{\varrho +1}^2,\\
   & e_3\leq \sqrt{\frac{3\varrho}{s-\varrho +1}}t,\quad e_4\leq \frac{e\sqrt{s}}{s-\varrho +1}t,\quad
    e_5\leq\frac{\sqrt{l}+\sqrt{\varrho }+u}{\sqrt{l}-\sqrt{\varrho }-u},\nonumber
\end{align}
with $u>0,t>1$. The first event $E_1$ refines from Corollary \ref{col:AG2G1_deviation_refined_from_finding}. The second one $E_2$ is from     Lemmas \ref{lem:inverse_moment_F} and \ref{lem:inverse_moment_spectral}, and the third one $E_3$ is from   \cite{vershynin2018high}. The following theorem provides a detailed deviation error bound  with specific coefficients related to these classical results.
\begin{theorem}\label{thm:prob_inverseTerm}
    Let   $\bdOmega,\bdPhi$ be standard Gaussian. If $l>s \geq\varrho + 4$ and $(e_1e_4+e_2e_3+e_2e_4u)e_5\sigma_{\varrho}^{-2}<1$, then  for any $\mu>0,t>1$,
    \begin{align}
        &&\probleftright{\normSpectral{\left(\qmat{I}+\bdPhi_1\bdPhi_2^{T}\qmat{H} (\bdPhi_1\bdPhi_1^{T} )^{-1}\right)^{-1}}>\frac{1}{1-(e_1e_4+e_2e_3+e_2e_4u)e_5^2\sigma_{\varrho}^{-2}}}\nonumber\\
        &&< p_f:= 5e^{-u^2/2}+2t^{-(l-\varrho )}+2t^{-(s-\varrho )}. \label{eq:p_f}
    \end{align}
\end{theorem}
\begin{proof}
    The main idea is from the   Neumann series: if $\normSpectral{\qmat{X}}<\alpha<1$, then   $\normSpectral{(\qmat{I}-\qmat{X})^{-1}}\leq ({1-\alpha})^{-1}$. Thus we  estimate the spectral norm of   $\bdPhi_1\bdPhi_2^{T}\qmat{H} (\bdPhi_1\bdPhi_1^{T} )^{-1}$ and recall $\qmat{H}=\boldsymbol{\Sigma}_2^2\bdOmega_2\bdOmega_1^\dagger\boldsymbol{\Sigma}_1^{-2}$:
    \begin{align*}
        \normSpectral{\bdPhi_1\bdPhi_2^{T}\qmat{H} (\bdPhi_1\bdPhi_1^{T} )^{-1}} &= \normSpectral{(\bdPhi_1\bdPhi_1^{T})(\bdPhi_1^\dagger)^{T}\bdPhi_2^{T}\qmat{H} (\bdPhi_1\bdPhi_1^{T} )^{-1}}\nonumber\\
        &\leq \normSpectral{(\bdPhi_1^\dagger)^{T}\bdPhi_2^{T}\boldsymbol{\Sigma}_2^2\bdOmega_2\bdOmega_1^\dagger}\normSpectral{\boldsymbol{\Sigma}_1^{-2}}\normSpectral{\bdPhi_1\bdPhi_1^{T}}\normSpectralnoleftright{ (\bdPhi_1\bdPhi_1^{T} )^{-1}}\nonumber\\
        &= \normSpectral{(\bdPhi_1^\dagger)^{T}\bdPhi_2^{T}\boldsymbol{\Sigma}_2^2\bdOmega_2\bdOmega_1^\dagger}\normSpectral{\boldsymbol{\Sigma}_1^{-2}}\kappa({\bdPhi_1\bdPhi_1^{T}}).
    \end{align*} 
 By   Corollary \ref{col:SGT_deviation} with $\qmat{S}= (\bdPhi_1^\dagger)^{T}\bdPhi_2^{T}\boldsymbol{\Sigma}_2^2,\qmat{G}=\bdOmega_2,\qmat{T}=\bdOmega_1^\dagger$ and   conditioned on $ E_3$:
 \begin{small}
    \begin{align*}
        \probleftright{\normSpectral{\bdPhi_1\bdPhi_2^{T}\qmat{H} (\bdPhi_1\bdPhi_1^{T} )^{-1}}>(e_1e_4+e_2e_3+e_2e_4u) \sigma_{\varrho}^{-2}e_5^2 \mid E_3}\leq e^{-u^2/2} + \probleftright{E_1^c} + \probleftright{E_2^c},
    \end{align*} 
\end{small}
    where $e_1,e_2,e_3,e_4 $ were defined in \eqref{eq:parametersE}. On the other hand, we have
    \begin{align}\label{eq:E1E2E3_complement}
        \probleftright{E_1^c}\leq 2t^{-(l-\varrho )}+2e^{-u^2/2},\quad \probleftright{E_2^c}\leq 2t^{-(s-\varrho )},\quad \probleftright{E_3^c}\leq 2e^{-u^2/2},
    \end{align}
    for $l-\varrho \geq4$ and $u>0,t>1$, where the first one can be derived from Corollary \ref{col:AG2G1_deviation_refined_from_finding}, the second from Lemmas \ref{lem:inverse_moment_F} and \ref{lem:inverse_moment_spectral}, and the last follows the spectral estimation of Gaussian matrices  \cite{vershynin2018high}.
    We finally remove the condition on $E_3$ by \eqref{eq:E1E2E3_complement} and   obtain the assertion.
\end{proof}

 $(e_1e_4+e_2e_3+e_2e_4u)e_5^2\sigma_{\varrho}^{-2}$   can be smaller than a constant $\alpha<1$ with proper choice of $u,t$ and $l,\varrho ,s$. To simplify,   assume that $s=2\varrho=2r$,   $e_5\leq 2$, which is not a strong assumption in practice. It follows   $e_3\leq\sqrt{3}t$ and $e_4\leq\frac{e\sqrt{2}}{\sqrt{\varrho }}t$. Then $(e_1e_4+e_2e_3+e_2e_4u)e_5^2\sigma_{\varrho}^{-2}$ is simplified as:
\begin{small}
\begin{equation*}
    2et^2\!\!\left(\sqrt{\frac{6 }{l-\varrho +1}}+\frac{\sqrt{3l}}{l-\varrho +1}\right)\!\!\frac{\normF{\boldsymbol{\Sigma}_2^2}}{\sigma_{\varrho}^2}+
    2t^2\!\left(\frac{3\sqrt{\varrho }}{l-\varrho +1}+u\cdot\frac{e\sqrt{3l}}{l-\varrho +1}+u\cdot\frac{e^2\sqrt{2l/\varrho }}{l-\varrho +1}\right)\!\frac{\sigma_{\varrho +1}^2}{\sigma_{\varrho}^2}.
\end{equation*}
\end{small}
\noindent Clearly, for a reasonably large $l$, the above can be smaller than a constant $\alpha<1$.
\color{black}
These discussions  together with Theorem \ref{thm:prob_inverseTerm} show that the event
\begin{align}\label{eq:event_EF}
    E_F:=\left\{\bdOmega,\bdPhi:\normSpectral{\left(\qmat{I}+\bdPhi_1\bdPhi_2^{T}\boldsymbol{\Sigma}_2^2\bdOmega_2\bdOmega_1^\dagger\boldsymbol{\Sigma}_1^{-2} (\bdPhi_1\bdPhi_1^{T} )^{-1}\right)^{-1}}<\xi\right\}.
\end{align}
holds with   probability at least $1-p_f$
for a constant $\xi= 1/(1-\alpha)$. 

As a consequence,        the averaged projection error is upper bounded as follows.
\begin{theorem}\label{thm:expectation_ProjectionError} 
 Let $\qmat{Q}$ be computed by Algorithm \ref{alg:psa-sps} with $q=1$  and   $\bdPhi,\bdOmega$ be independent standard Gaussian. If $l>s \geq \varrho+ 4$, then 
    \begin{small}
    \begin{equation*}   
             \Expectation{\normF{\qmat{A}-P_{\qmat{Q}}\qmat{A}}^2\mid E_F}\leq 
        (1+\epsilon\hat{\xi}^2)\tau_{\varrho +1}^2(\qmat{A})+\delta\hat{\xi}^2\frac{\sum_{j=\varrho +1}^{n}\sigma_j^6}{\sigma_{\varrho}^4}+\mu\hat{\xi}^2\tau_{\varrho +1}^2(\qmat{A})\frac{\sum_{j=\varrho +1}^{n}\sigma_j^4}{\sigma_{\varrho}^4},
    \end{equation*}
\end{small}
  where $\hat{ \xi}:=\xi/\sqrt{\mathbb P(E_F)}$ with     $E_F$  defined in \eqref{eq:event_EF} and  $\xi>1$,        $\delta,\mu$   are as in \eqref{eq:parameters_eps_delta_mu}, and 
       $ \epsilon=\frac{2\varrho   }{(l-\varrho -1)}$. 
       
       In particular,   for   given $\xi>1,u>0,t>1$,  if  $l$ is   large enough so that $(e_1e_4+e_2e_3+e_2e_4u)e_5\sigma_{\varrho}^{-2}< 1-\xi^{-1}$ ($e_i$'s are defined in \eqref{eq:parametersE}), then     $\mathbb P(E_F)\geq (1-p_f)$ with $p_f$ as in \eqref{eq:p_f}, hence $\hat{\xi} \leq \xi/\sqrt{1-p_f}\approx \xi$.
\end{theorem}
\begin{proof}
    Note that $P_{\qmat{Q}}=P_{\qmat{\hat{\qmat{Y}}}}$. Conditioned on $E_F$, Theorem \ref{thm:projection_error_sketchPower} shows that the expectation is upper bounded by $\tau_{\varrho +1}^2(\qmat{A})+\xi^2\mathbb E [Z|E_F] $, where $Z:=\|\boldsymbol{\Sigma}_2\bdPhi_2\bdPhi_1^\dagger+\boldsymbol{\Sigma}_2\bdPhi_2\bdPhi_2^{T}\qmat{H} (\bdPhi_1\bdPhi_1^{T} )^{-1} \|_F^2 $. By the definition of conditional expectation  for nonnegative random variables, $ \mathbb E [Z|E_F] \leq \mathbb E [Z]/\mathbb P(E_F)$. Hence the upper bound becomes $\tau_{\varrho +1}^2(\qmat{A})+\hat{\xi}^2\mathbb E [Z] $ with $\hat{\xi} = \xi/\sqrt{\mathbb P(E_F)}$. Then, decompose $Z\leq   2    \|\boldsymbol{\Sigma}_2\bdPhi_2\bdPhi_1^\dagger\|_F^2 + 2  \|\boldsymbol{\Sigma}_2\bdPhi_2\bdPhi_2^{T}\qmat{H} (\bdPhi_1\bdPhi_1^{T} )^{-1}\|_F^2$ and  apply Corollary \ref{col:SigmaPhi2Phi1_second_moment_bound} and Theorem \ref{thm:expectation_projectionError_partOne} to bound each term, completing the proof of the main statement. 

If, in addition,    $l$ is   large enough such that   $(e_1e_4+e_2e_3+e_2e_4u)e_5\sigma_{\varrho}^{-2}\leq 1-\xi^{-1}$, then Theorem \ref{thm:prob_inverseTerm} ensures that $\mathbb P(E_F)>1-p_f$. Consequently, $\hat{\xi} \leq \xi/\sqrt{1-p_f}$. 
\end{proof}
\color{black}




\color{black}

\subsection{Averaged oblique projection error}\label{sec:onepass_err_1}
Previous subsections have established the orthogonal projection error. To obtain the main bound stated in Sect. \ref{sec:error_bound_q=1_stated}, we still need a results established in \cite{Practical_Sketching_Algorithms_Tropp}   that characteristizes the connection between oblique and orthogonal projections. The proof is given in Appendix \ref{sec:proofs_technical_lemmas}.
\begin{lemma}\label{lem:GNErrorFrobenius}
    Let $\qmat{Q},\qmat{B}$ be generated by Algorithm \ref{alg:psa-sps} and $\bdPsi$ be standard Gaussian and independent from arbitrarily random matrices $\bdOmega,\bdPhi$. If $d>s+1$, then 
    \begin{gather*}\small
        \ExpectationNoBracket[\bdPsi,\bdOmega,\bdPhi]{\normF{\qmat{A}-\qmat{Q}\qmat{B}}^2} \leq \frac{d}{d-s-1}\ExpectationNoBracket[\bdOmega,\bdPhi]{\normF{\qmat{A}-P_{\qmat{Q}}\qmat{A}}^2};\\
       \small \ExpectationNoBracket[\bdPsi,\bdOmega,\bdPhi]{\normSpectral{\qmat{A}-\qmat{Q}\qmat{B}}} \leq \left(1+\sqrt{\frac{s}{d-s-1}}\right)\ExpectationNoBracket[\bdOmega,\bdPhi]{\normSpectral{\qmat{A}-P_{\qmat{Q}}\qmat{A}}}+\frac{e\sqrt{d+s}}{\sqrt{2}(d-s)}\ExpectationNoBracket[\bdOmega,\bdPhi]{\normF{\qmat{A}-P_{\qmat{Q}}\qmat{A}}}.
    \end{gather*}
\end{lemma}

\begin{proof}[Proof of Theorem \ref{thm:oblique_proj_error_q=1}]
   Since  Lemma \ref{lem:GNErrorFrobenius} holds for any random $\bdPhi,\bdOmega$,   it also applies when they are Gaussian and further restricted to $E_F$. Therefore, the assertion follows by   combining Theorem \ref{thm:expectation_ProjectionError} and Lemma \ref{lem:GNErrorFrobenius}.  
\end{proof}

\begin{remark} We only present the bound in Frobenius norm. That in   spectral norm can be derived similarly, which is omitted   for simplicity.
\end{remark}
\color{black}

\section{Analysis for the \texorpdfstring{$q\geq 1$}{q \geq 1} Bound} \label{sec:error_anal_q>1}


To prove the bound in Theorem \ref{thm:error_oblique_bound_q_gt_1},  the analysis in the previous section might not yield reasonable bounds. As an alternative, we adapt the decomposition idea     in \cite{martinssonRandomizedAlgorithmDecomposition2011,rokhlin2010randomized,woolfe2008Fast},   with several refinements to keep the   bound   independent of the   ambient matrix size. Throughout,
 we   assume $\bdPsi\in\bbR^{d\times n},\bdPhi\in\bbR^{m\times l}$ are still independent standard Gaussian, while $\bdOmega=\bdPhi\bdtildeOmega$ with $\bdtildeOmega\in\bbR^{l\times s}$   standard Gaussian independent of $\bdPsi,\bdPhi$. The analysis still works when $\bdOmega$ is standard Gaussian, but the bound will be worse. Thus we only present the case $\bdOmega=\bdPhi\bdtildeOmega$. This section is organized as follows:
  Sect. \ref{sec:q>1_pre} presents technical lemmas;  Sect. \ref{sec:deter_proj_error_qgt1} derives the deterministic error,   with probabilistic error bounds   established in Sect. \ref{sec:prob_proj_error_qgt1} and \ref{sec:prob_obl_proj_error_qgt1}. \color{black}


\subsection{Technical lemmas}\label{sec:q>1_pre}

For some positive integer $k$ satisfying $ k<l\ll\min\{m,n\}$, partition   $\qmat{A}$ via SVD as 
\begin{align}\label{eq:A_SVD_k_s_l}
    \begin{array}{c@{}c@{}c@{}c@{}c}
        &   & \begin{array}{cc} \scriptstyle k & \scriptstyle n-k \end{array} &   \begin{array}{c}
               \scriptstyle n \end{array} & \\[-2pt]  
               \qmat{A}=\qmat{U}\boldsymbol{\Sigma}\qmat{V}^{T}= &      \qmat{U} &
           \left[
           \begin{array}{cc}
               \boldsymbol{\Sigma}_1 &  \\ 
                & \boldsymbol{\Sigma}_2
           \end{array}
           \right] &
               \left[
               \begin{array}{c}
                   \mathbf{V}_1^{T}  \\
                   \mathbf{V}_2^{T}
               \end{array}
               \right] &     
           \begin{array}{c}
               \scriptstyle k \\[2pt]
               \scriptstyle n-k
           \end{array}
       \end{array},
\end{align}
where 
 $\boldsymbol{\Sigma}_1\in\bbR^{k\times k}$, $\boldsymbol{\Sigma}_2 \in\bbR^{(n-k)\times (n-k)}$,   $\qmat{U} \in\bbR^{m\times n}$,   $\qmat{V}_1\in \bbR^{n\times k}$, and $\qmat{V}_2\in\bbR^{n\times (n-k)}$. The singular values   are arranged in a descending order.
Then define the block partition 
\begin{align}\label{eq:blockOfV*Phi}
    \qmat{V}^{T}\bdPhi=\begin{bmatrix}
        \qmat{V}_1^{T}\bdPhi\\
        \qmat{V}_2^{T}\bdPhi
    \end{bmatrix}=
    \overset{l}{\begin{bmatrix}
        \bdPhi_1\\
        \bdPhi_2
    \end{bmatrix} }
    \begin{array}{c}
        \scriptstyle k \\[-2pt]
        \scriptstyle n-k
    \end{array} .
\end{align}

The following  estimates the upper bound of   $\sigma_{\varrho+1}(\qmat{Z}) $. 
 Martinsson et al. derived  $\sigma_{\varrho+1}(\qmat{Z})\leq O(\sqrt{\max\{\varrho+j,l  \} })\sigma_{\varrho+1}(\qmat{A})+O(\sqrt{\max\{m-\varrho-j,l  \} } )\sigma_{\varrho+j+1}(\qmat{A})$ \cite{martinssonRandomizedAlgorithmDecomposition2011}. Similar to their analysis but credited to  Lemma \ref{lem:SGT}, we   obtain a bound   depending on sketch sizes and tail energy, which is crucial for our final results. 
\begin{lemma}\label{lem:sigma/tau_r+1_C}
    With the notations in \eqref{eq:A_SVD_k_s_l} and \eqref{eq:blockOfV*Phi} and $\qmat{Z}=\qmat{A}\bdPhi$ with $\bdPhi$ standard Gaussian, for any $\varrho$ with $0\leq \varrho\leq k-1<l-1$,   we have 
    \begin{align*}
        \sigma_{\varrho+1}&\left(\qmat{Z}\right)\leq\sigma_{\varrho+1}\left(\qmat{A}\right)\left(\sqrt{l}+\sqrt{k}+\beta\right)+\sigma_{k+1}(\qmat{A})\left(\sqrt{l}+\beta\right)+\tau_{k+1}(\qmat{A});\\
        \tau_{\varrho+1}&\left(\qmat{Z}\right)\leq\tau_{\varrho+1}\left(\qmat{A}\right)\left(\sqrt{l}+\sqrt{k}+\beta\right)+\sigma_{k+1}(\qmat{A})\beta+\tau_{k+1}(\qmat{A})\sqrt{l}      
    \end{align*}
    with exception probability at most $2e^{-\beta^2/2}$ perspectively,   $\beta>0$.
\end{lemma}
\begin{proof} 
     With the block representation of \eqref{eq:A_SVD_k_s_l} and \eqref{eq:blockOfV*Phi}, we have:
    \begin{align}
        \sigma_{\varrho+1}\left(\qmat{Z}\right)&\leq\sigma_{\varrho+1}\left(\boldsymbol{\Sigma}\begin{bmatrix}
        \bdPhi_1\\
        0
        \end{bmatrix}\right)+\normSpectral{\boldsymbol{\Sigma}\begin{bmatrix}0\\ \bdPhi_2\end{bmatrix}}=\sigma_{\varrho+1}\left(\boldsymbol{\Sigma}_1\bdPhi_1\right)+\normSpectral{\boldsymbol{\Sigma}_2\bdPhi_2}; \label{eq:lem:sig_APhi_1} \\
        \tau_{\varrho+1}^2\left(\qmat{Z}\right)&\leq \tau_{\varrho+1}^2\left(\boldsymbol{\Sigma}\begin{bmatrix}
            \bdPhi_1\\
            0
            \end{bmatrix}\right)+\normFSquare{\boldsymbol{\Sigma}\begin{bmatrix}0\\ \bdPhi_2\end{bmatrix}}=\tau_{\varrho+1}^2\left(\boldsymbol{\Sigma}_1\bdPhi_1\right)+\normFSquare{\boldsymbol{\Sigma}_2\bdPhi_2}, \label{eq:lem:sig_APhi_2} 
    \end{align}
   where the inequality in the first line uses that $\qmat{Z}$ and $\boldsymbol{\Sigma}\qmat{V}^{T}\bdPhi$ share the same singular values and uses Weyl's inequality.  For the second line,   observe      $\|\qmat{Z}\|_F^2 =\|\boldsymbol{\Sigma}\qmat{V}^{T}\bdPhi \|_F^2=\|\boldsymbol{\Sigma}_1\bdPhi_1\|_F^2+\|\boldsymbol{\Sigma}_2\bdPhi_2\|_F^2$, and $\sigma_i(\boldsymbol{\Sigma}\qmat{V}^{T}\bdPhi)\geq \sigma_i(\boldsymbol{\Sigma} \begin{bmatrix}
    \bdPhi_1\\
    0
    \end{bmatrix})$ for any $i$ (see e.g., \cite[Remark 3.1]{gu2015subspace}).    Thus,
    \begin{align*}
        \normFnoleftright{\boldsymbol{\Sigma}_2\bdPhi_2 }^2 &= \sum_{i=1}^{\varrho}\sigma_i(\boldsymbol{\Sigma}\qmat{V}^{T}\bdPhi) + \tau_{\varrho+1}^2(\boldsymbol{\Sigma}\qmat{V}^{T}\bdPhi) - \sum_{i=1}^{\varrho}   \sigma^2_i\left(\boldsymbol{\Sigma} \begin{bmatrix}
            \bdPhi_1\\
            0
            \end{bmatrix}\right) - \tau_{\varrho+1}^2\left(\boldsymbol{\Sigma}\begin{bmatrix}
                \bdPhi_1\\
                0
                \end{bmatrix}\right) \\
 &\geq \tau_{\varrho+1}^2(\boldsymbol{\Sigma}\qmat{V}^{T}\bdPhi) - \tau_{\varrho+1}^2\left(\boldsymbol{\Sigma}\begin{bmatrix}
    \bdPhi_1\\
    0
    \end{bmatrix}\right).
    \end{align*}

    Next,  the terms on the right hand-side of \eqref{eq:lem:sig_APhi_1} and \eqref{eq:lem:sig_APhi_2}    satisfy:
    \begin{gather*}
        \sigma_{\varrho+1}\left(\boldsymbol{\Sigma}_1\bdPhi_1\right)\leq\sigma_{\varrho+1}\left(\boldsymbol{\Sigma}_1\right)\normSpectral{\bdPhi_1}\leq\sigma_{\varrho+1}\left(\boldsymbol{\Sigma}_1\right)\left(\sqrt{l}+\sqrt{k}+\beta\right);\\
        \tau_{\varrho+1}\left(\boldsymbol{\Sigma}_1\bdPhi_1\right)\leq\tau_{\varrho+1}\left(\boldsymbol{\Sigma}_1\right)\normSpectral{\bdPhi_1}\leq\tau_{\varrho+1}\left(\boldsymbol{\Sigma}_1\right)\left(\sqrt{l}+\sqrt{k}+\beta\right);\\
        \normSpectral{\boldsymbol{\Sigma}_2\bdPhi_2}=\normSpectral{\boldsymbol{\Sigma}_2\bdPhi_2\qmat{I}_l}\leq\normSpectral{\boldsymbol{\Sigma}_2}\left(\sqrt{l}+\beta\right)+\normF{\boldsymbol{\Sigma}_2};~~
        \normF{\boldsymbol{\Sigma}_2\bdPhi_2}\leq\normF{\boldsymbol{\Sigma}_2}\sqrt{l}+\normSpectral{\boldsymbol{\Sigma}_2}\beta,
    \end{gather*}
where the  first two lines uses the spectral estimation of Gaussian matrices \cite{vershynin2018high}, while the last line can be obtained by combining     Lemmas \ref{lem:SGT} and \ref{lem:LipschitzConcentrationIneq},   all with exception probability at most $e^{-\beta^2/2}$. By definitions of $\boldsymbol{\Sigma}_1$ and $\boldsymbol{\Sigma}_2$, $\sigma_{\varrho +1}(\boldsymbol{\Sigma}_1)=\sigma_{\varrho+1}(\qmat{A})$, $\|\boldsymbol{\Sigma}_1\| = \sigma_{k+1}(\qmat{A})$, and $\|\boldsymbol{\Sigma}_2\|_F = \tau_{k+1}(\qmat{A})$. 
    Combining the above, we have the assertion.
\end{proof}

\begin{lemma}[Proposition 8.6 of \cite{FindingStructureHalko}]\label{lem:power_iter_spectral_ineq}
    Let \( P \) be an orthogonal projector, and let \(\qmat{M} \) be a matrix. For each positive number \( q \),
    \[
        \normSpectral{P\qmat{M}} \leq \normSpectral{P\left(\qmat{M}\qmat{M}^{T}\right)^q \qmat{M}}^{1/(2q+1)}.
    \] 
\end{lemma}

\begin{lemma}[c.f. \cite{FindingStructureHalko}]\label{lem:rsvdDrivationBound}
    For any matrix $\qmat{M}\in\bbR^{m\times n}$ and standard Gaussian matrix $\bdOmega\in\bbR^{n\times s}$, for any $\varrho\leq s-4$ and $  t \geq 1$, $u>0$, 
    \begin{small}
    \begin{align*}
        \| (\qmat{I} - P_{\qmat{M}\bdOmega}) \qmat{M} \|_F &\leq 
        \left( 1 + t \cdot \sqrt{\frac{3\varrho}{s-\varrho+1}} \right) 
        \tau_{\varrho+1}(\qmat{M})
        + u t \cdot \frac{e \sqrt{s}}{s-\varrho+1} \cdot \sigma_{\varrho+1}\left(\qmat{M}\right);\\
        \| (\qmat{I} - P_{\qmat{M}\bdOmega})\qmat{M} \| &\!\leq\! \left(\! 1 + t \cdot \sqrt{\frac{3\varrho}{s-\varrho + 1}}+ut \cdot \frac{e \sqrt{s}}{s-\varrho+1}  \right) \sigma_{\varrho+1}\left(\qmat{M}\right) + t \cdot \frac{e \sqrt{s}}{s-\varrho+1}\tau_{\varrho+1}\left(\qmat{M}\right),
    \end{align*}
\end{small}
    with failure probability at most \( 2t^{-(s-\varrho)} + e^{-u^2 / 2} \).
\end{lemma}

\subsection{Deterministic projection error}\label{sec:deter_proj_error_qgt1}
We first analyze the behavior of the projection error without randomization.
Denote $$\qmat{F}:=\begin{bmatrix}
    \bdPhi_1^\dagger & 0
\end{bmatrix} \in \bbR^{l\times n};\quad (\text{recall}~\bdPhi_1^\dagger\in\bbR^{l\times k},~l\geq k).$$ We have the following result:
\begin{theorem}\label{thm:determined_projection_error}
    With the notations in \eqref{eq:A_SVD_k_s_l} and \eqref{eq:blockOfV*Phi}; assume that $\bdPhi_1$ has full row rank. Then the following estimation holds:
    \begin{align*}
        \normP{\qmat{A}-P_{\left(\qmat{Z}\qmat{Z}^{T}\right)^q\qmat{A}\bdOmega}\qmat{A}}\leq 2\bigxiaokuohao{\normP{\boldsymbol{\Sigma}_2}+ \|\boldsymbol{\Sigma}_2\bdPhi_2\bdPhi_1^\dagger\|_p } +\normP{\left(\qmat{I}-P_{\left(\qmat{Z}\qmat{Z}^{T}\right)^q\qmat{A}\bdOmega}\right)\qmat{Z}}\normSpectral{\qmat{F}}.
    \end{align*}
\end{theorem}
\begin{proof} 
    We have the following estimation by triangle inequality:
    \begin{align*}
     \!\!   \normP{\qmat{A}-P_{\left(\qmat{Z}\qmat{Z}^{T}\right)^q\qmat{A}\bdOmega}\qmat{A}}\!\leq\! &\normP{\qmat{A}-\qmat{Z}\qmat{F}}+\normP{\qmat{Z}\qmat{F}-P_{\left(\qmat{Z}\qmat{Z}^{T}\right)^q\qmat{A}\bdOmega}\qmat{Z}\qmat{F}}\\
        &+\normP{P_{\left(\qmat{Z}\qmat{Z}^{T}\right)^q\qmat{A}\bdOmega}\qmat{Z}\qmat{F}-P_{\left(\qmat{Z}\qmat{Z}^{T}\right)^q\qmat{A}\bdOmega}\qmat{A}}\\
        = &\!\normP{\qmat{A}-\qmat{Z}\qmat{F}}\!\!+\!\normP{\left(\qmat{I}\!-\!P_{\left(\qmat{Z}\qmat{Z}^{T}\right)^q\qmat{A}\bdOmega}\right)\qmat{Z}\qmat{F}}\!\!+\!\normP{P_{\left(\qmat{Z}\qmat{Z}^{T}\right)^q\qmat{A}\bdOmega}\left(\qmat{Z}\qmat{F}-\qmat{A}\right)}\\
        \leq &2\normP{\qmat{A}-\qmat{Z}\qmat{F}}+\normP{\left(\qmat{I}-P_{\left(\qmat{Z}\qmat{Z}^{T}\right)^q\qmat{A}\bdOmega}\right)\qmat{Z}}\normSpectral{\qmat{F}},
    \end{align*}
    where the last line uses the nonexpansiveness of the projection operator. We then have
\begin{align*}
    \normP{\qmat{A}-\qmat{Z}\qmat{F}} = \normP{\qmat{A}-\qmat{A}\bdPhi\qmat{F}}=\normP{\qmat{U}\boldsymbol{\Sigma}\qmat{V}^{T}-\qmat{U}\boldsymbol{\Sigma}\qmat{V}^{T}\bdPhi\qmat{F}} 
     =\normP{\boldsymbol{\Sigma}-\boldsymbol{\Sigma}\qmat{V}^{T}\bdPhi\qmat{F}}
\end{align*}
by the   rotational invariance of any Schatten-$p$ norm. This yields
\begin{align*}
    \normP{\qmat{A}-\qmat{Z}\qmat{F}}&= \normP{\boldsymbol{\Sigma}\left(\qmat{I}- \begin{bmatrix}
        \bdPhi_1\\
        \bdPhi_2
    \end{bmatrix} [\bdPhi_1^\dagger ,0]  \right)} =  \normP{\boldsymbol{\Sigma}\left(\qmat{I}-\begin{bmatrix}
        \qmat{I} & 0\\
        \bdPhi_2\bdPhi_1^\dagger & 0
    \end{bmatrix}\right)}\\
    &=\normP{\boldsymbol{\Sigma}\begin{bmatrix}
        0 & 0\\
        -\bdPhi_2\bdPhi_1^\dagger & \qmat{I}
    \end{bmatrix}}\leq\normP{\boldsymbol{\Sigma}_2}+\normP{\boldsymbol{\Sigma}_2\bdPhi_2\bdPhi_1^\dagger}.
\end{align*}
Combining the above yields the assertion.
\end{proof}
The above framework, adapted from \cite{martinssonRandomizedAlgorithmDecomposition2011,rokhlin2010randomized,woolfe2008Fast}, uses $\qmat{Z}$ as a bridge to decompose the error into two terms $\|\qmat{A}-\qmat{Z}\qmat{F}\|_p$ and $\|(\qmat{I}- P_{\qmat{Z}\qmat{Z}^{T})^q\qmat{A}\bdOmega}\qmat{Z}\|$, which are relatively easier to analyze. We will estimate the second term in the next subsection.  
Here for the Frobenius norm error we further have
\begin{align} \label{eq:deterministic_error_q_gt_1_Fnorm}
    \normF{\qmat{A}-P_{\left(\qmat{Z}\qmat{Z}^{T}\right)^q\qmat{A}\bdOmega}\qmat{A}}\leq &2\normF{\qmat{A}-\qmat{Z}\qmat{F}}+\normF{\left(\qmat{I}-P_{\left(\qmat{Z}\qmat{Z}^{T}\right)^q\qmat{A}\bdOmega}\right)\qmat{Z}}\normSpectral{\qmat{F}}\nonumber\\
    \leq & 2\normF{\qmat{A}-\qmat{Z}\qmat{F}}+\sqrt{l}\normSpectral{\left(\qmat{I}-P_{\left(\qmat{Z}\qmat{Z}^{T}\right)^q\qmat{A}\bdOmega}\right)\qmat{Z}}\normSpectral{\qmat{F}},
\end{align}
where the inequality is due to the relation between Frobenius and spectral norm, and that   $\qmat{Z}\in\bbR^{n\times l}$. 

\subsection{Probabilistic projection error}\label{sec:prob_proj_error_qgt1}
We complete the randomized projection error analysis. We first give a lemma.

\begin{lemma}\label{lem:I-P_DOmega_C}
    Let $\qmat{Z}=\qmat{A}\bdPhi$ and $\bdOmega=\bdPhi\bdtildeOmega$ with $\bdPhi\in\bbR^{n\times l}$ and $\bdtildeOmega\in\bbR^{l\times s}$     independent standard Gaussian matrices. For any $\varrho,k$ with $\varrho<k<l$ and $\varrho<s-4$, we have 
    \begin{small}
    \begin{align*}
        \normSpectral{\left(\qmat{I}-P_{(\qmat{Z}\qmat{Z}^{T})^q\qmat{A}\bdOmega}\right)\qmat{Z}}&\leq \alpha\cdot\left(\sigma_{\varrho+1}\left(\qmat{A}\right)\left(\sqrt{l}+\sqrt{k}+\beta\right)+\sigma_{k+1}\left(\qmat{A}\right)\left(\sqrt{l}+\beta\right)+\tau_{k+1}\left(\qmat{A}\right)\right) 
    \end{align*}
\end{small}
with exception probability at most \( 2t^{-(s-\varrho)} + e^{-u^2 / 2}+e^{-\beta^2/2} \),   where the parameter 
    \begin{align}
        \alpha=\left( 1 + t \cdot \sqrt{\frac{3\varrho}{s-\varrho + 1}}+t \cdot \frac{e\cdot \left(\sqrt{sl}+u\sqrt{s}\right)}{s-\varrho+1}\right)^{\frac{1}{2q+1}}.
    \end{align}

\end{lemma}
\begin{proof}
    For notational convenience, we denote $\qmat{D}:= (\qmat{Z}\qmat{Z}^{T})^q\qmat{Z}\in\bbR^{m\times l}$. It follows from $\bdOmega=\bdPhi\bdtildeOmega$ that  $P_{(\qmat{Z}\qmat{Z}^{T})^q\qmat{A}\bdOmega} = P_{\qmat{D}\bdtildeOmega}$.   Consider the following fact from Lemma \ref{lem:power_iter_spectral_ineq}:
    \begin{align}\label{eq:I-P_DOmegaCDeter}
        \normSpectral{\left(\qmat{I}-P_{\qmat{D}\bdtildeOmega}\right)\qmat{Z}}\leq\normSpectral{\left(\qmat{I}-P_{\qmat{D}\bdtildeOmega}\right)\left(\qmat{Z}\qmat{Z}^{T}\right)^q\qmat{Z}}^{\frac{1}{2q+1}}
        =\normSpectral{\left(\qmat{I}-P_{\qmat{D}\bdtildeOmega}\right)\qmat{D}}^{\frac{1}{2q+1}};
    \end{align}
  note that $\bdtildeOmega\in\bbR^{l\times s}$ is Gaussian; applying   Lemma \ref{lem:rsvdDrivationBound} we have for any $\varrho<s-4$:
  \begin{small}
    \begin{align*}
        \normSpectral{\left(\qmat{I}-P_{\qmat{D}\bdtildeOmega}\right)\qmat{D}}\leq  \left( 1 + t \cdot \sqrt{\frac{3\varrho}{s-\varrho + 1}} + ut \cdot \frac{e \sqrt{s}}{s-\varrho+1}\right) \sigma_{\varrho+1}(\qmat{D}) + t \cdot \frac{e \sqrt{s}}{s-\varrho+1} \tau_{\varrho+1}\left(\qmat{D}\right) 
    \end{align*}
\end{small}
 \noindent   with failure probability at most \( 2t^{-(s-\varrho)} + e^{-u^2 / 2} \). Next,   $\sigma_{\varrho+1}(\qmat{D})$ and $\tau_{\varrho+1}(\qmat{D})$ satisfy
\begin{align}\label{eq:sigma_r1_D_and_tau_r1_D}
    \tau_{\varrho+1}\left(\qmat{D}\right)
    \leq \sqrt{l-\varrho}\sigma_{\varrho+1}\left(\qmat{D}\right),~  \sigma_{\varrho+1}\left(\qmat{D}\right)=\sigma_{\varrho+1}\left(\left(\qmat{Z}\qmat{Z}^{T}\right)^q\qmat{Z}\right)=\sigma_{\varrho+1}^{2q+1}\left(\qmat{Z}\right)
   ;
\end{align}
  combining the above pieces yields:
\begin{align*}
    \normSpectral{\left(\qmat{I}-P_{\qmat{D}\bdtildeOmega}\right)\qmat{Z}}&\leq\left( 1 + t \cdot \sqrt{\frac{3\varrho}{s-\varrho + 1}}+t \cdot \frac{e\cdot \left(\sqrt{sl}+u\sqrt{s}\right)}{s-\varrho+1}\right)^{\frac{1}{2q+1}} \sigma_{\varrho+1}^{\frac{1}{2q+1}}(\qmat{D})\\
    & =\alpha \sigma_{\varrho+1}(\qmat{Z})\\
    &\leq\alpha\cdot\left(\sigma_{\varrho+1} (\qmat{A} ) (\sqrt{l}+\sqrt{k}+\beta )+\sigma_{k+1} (\qmat{A} ) (\sqrt{l}+\beta )+\tau_{k+1} (\qmat{A} )\right),
\end{align*}
where the second inequality follows from Lemma \ref{lem:sigma/tau_r+1_C}. 
This completes the proof.
\end{proof}

\begin{remark}
For $\tau_{\varrho+1}(\qmat{D})$ in \eqref{eq:sigma_r1_D_and_tau_r1_D},     we could have used $\tau_{\varrho+1}(\qmat{D})\leq \tau^{2q+1}_{\varrho+1}(\qmat{Z})$   and   the second relation of Lemma \ref{lem:sigma/tau_r+1_C}  to obtain a slightly better bound. However, we prefer to keep the current form for clarity and simplicity.
\end{remark}

With the preparations, we have the following theorem.
\begin{theorem}\label{thm:prob_projection_error_q_gt_1}
    Let $\qmat{Z}=\qmat{A}\bdPhi$ and $\bdOmega=\bdPhi\bdtildeOmega$ with $\bdPhi\in\bbR^{n\times l},\bdtildeOmega\in\bbR^{l\times s}$     independent standard Gaussian matrices.  For any $\varrho,k$ with $\varrho<k\leq l-4$ and $\varrho\leq s-4$, we have  
    \begin{small}
    \begin{align*}
        \normSpectral{\qmat{A}-P_{\left(\qmat{Z}\qmat{Z}^{T}\right)^q\qmat{A}\bdOmega}\qmat{A}}&\leq 2 \left( 1 + t \cdot \sqrt{\frac{3k}{l-k + 1}}+  \frac{ut\cdot e \sqrt{l}}{l-k+1} \right) \sigma_{k+1}\left(\qmat{A}\right) +  \frac{2t \cdot e \sqrt{l}}{l-k+1} \tau_{k+1}\left(\qmat{A}\right)\\
        +&
        \frac{t\alpha e \sqrt{l}}{l-k+1}\cdot\left(\sigma_{\varrho+1}\left(\qmat{A}\right) (\sqrt{l}+\sqrt{k}+\beta )+\sigma_{k+1} (\qmat{A} ) (\sqrt{l}+\beta )+\tau_{k+1}\left(\qmat{A}\right)\right)\\
      \!\!  \normF{\qmat{A}-P_{\left(\qmat{Z}\qmat{Z}^{T}\right)^q\qmat{A}\bdOmega}\qmat{A}}&\leq 2\left(1+t\sqrt{\frac{3k}{l-k+1}}\right)\tau_{k+1}\left(\qmat{A}\right)+ut\frac{e\sqrt{l}}{l-k+1}\sigma_{k+1}\left(\qmat{A}\right)\\
        +&
        \frac{t\alpha e l}{l-k+1}\cdot\left(\sigma_{\varrho+1}\left(\qmat{A}\right) (\sqrt{l}+\sqrt{k}+\beta )+\sigma_{k+1} (\qmat{A} ) (\sqrt{l}+\beta )+\tau_{k+1}\left(\qmat{A}\right)\right),
    \end{align*}
\end{small}
    where $\alpha=\left( 1 + t \cdot \sqrt{\frac{3\varrho}{s-\varrho + 1}}+t \cdot \frac{e\cdot \left(\sqrt{sl}+u\sqrt{s}\right)}{s-\varrho+1}\right)^{\frac{1}{2q+1}}$
    with failure  probability at most 
    \begin{align*}
        p_e=2e^{-u^2 / 2} + 2t^{-(l-k)}+2t^{-(s-\varrho)}+e^{-\beta^2/2}\label{eq:orthogonal_projection_error_probability}.
    \end{align*}
\end{theorem}
\begin{proof}
    By Theorem \ref{thm:determined_projection_error}, we have:
    \begin{align*}
        \normP{\qmat{A}-P_{\left(\qmat{Z}\qmat{Z}^{T}\right)^q\qmat{A}\bdOmega}\qmat{A}}\leq 2\bigxiaokuohao{\normP{\boldsymbol{\Sigma}_2}+ \|\boldsymbol{\Sigma}_2\bdPhi_2\bdPhi_1^\dagger\|_p } +\normP{\left(\qmat{I}-P_{\left(\qmat{Z}\qmat{Z}^{T}\right)^q\qmat{A}\bdOmega}\right)\qmat{Z}}\normSpectral{\qmat{F}}.
    \end{align*}
    When $p=2,\infty$, the   norms of $\boldsymbol{\Sigma}_2\bdPhi_2\bdPhi_1^\dagger$   are given by Corollary \ref{col:AG2G1_deviation_refined_from_finding}:
    \begin{gather*}
        \normSpectral{\boldsymbol{\Sigma}_2\bdPhi_2\bdPhi_1^\dagger} \leq \left(t \cdot \sqrt{\frac{3k}{l-k + 1}}+ut \cdot \frac{e \sqrt{l}}{l-k+1} \right) \sigma_{k+1}\left(\qmat{A}\right) + t \cdot \frac{e \sqrt{l}}{l-k+1} \tau_{k+1}\left(\qmat{A}\right);\\
        \normF{\boldsymbol{\Sigma}_2\bdPhi_2\bdPhi_1^\dagger}\leq t\cdot\sqrt{\frac{3k}{l-k+1}}\tau_{k+1}\left(\qmat{A}\right)+ut\frac{e\sqrt{l}}{l-k+1}\sigma_{k+1}\left(\qmat{A}\right);
    \end{gather*}
on the other hand, Lemma \ref{lem:inverse_moment_spectral} shows that $\normSpectralnoleftright{\bdPhi_1^\dagger}\leq \frac{e \sqrt{l}}{l-k+1}t$; these three inequalities simultaneously hold      with failure  probability at most $2t^{-(l-k)}+2e^{-u^2/2}$. Note also that $\normSpectral{\qmat{F}}=\| \bdPhi_1^\dagger\|$.
    Combining the above deductions with Lemma  \ref{lem:I-P_DOmega_C},  we obtain the assertion (note that for $\normFnoleftright{\qmat{A}-P_{\left(\qmat{Z}\qmat{Z}^{T}\right)^q\qmat{A}\bdOmega}\qmat{A}}$, we additionally use \eqref{eq:deterministic_error_q_gt_1_Fnorm}). 
\end{proof}

\subsection{Probabilistic oblique projection error}\label{sec:prob_obl_proj_error_qgt1}
We first need a lemma. 
\begin{lemma}\label{lem:GNError_deviation_bound}
     Let $\qmat{Q},\qmat{B}$ be generated by Algorithm \ref{alg:psa-sps} and $\bdPsi$ be standard Gaussian and independent from $\bdOmega,\bdPhi$. If $d\geq s+4$, then 
    \begin{small}
    \begin{gather*} 
        \probleftright{\!\normSpectral{\qmat{A}-\qmat{Q}\qmat{B}}\!>\!\frac{e\sqrt{d}t}{d-s+1}e_2\!+\!\left(\!1+t\sqrt{\frac{3s}{d-s+1}} +u \frac{e\sqrt{d}t}{d-s+1}\!\right)e_1 \mid E\!} \leq  e^{-u^2/2}\!+\!2t^{-(d-s)},
     \end{gather*}
\end{small}
    where $E=\left\{\bdOmega,\bdPhi|\normSpectral{\qmat{A}-P_{\qmat{Q}}\qmat{A}}<e_1,\normF{\qmat
    A-P_{\qmat{Q}}\qmat{A}}<e_2\right\}$.
\end{lemma}
\begin{proof}According to \eqref{eq:proof:lemma_QB:1}, it suffices to estimate $\|\qmat{Q}^{T}\qmat{A}-\qmat{B}\|$, which is equivalent to $\|\bdPsi_2^\dagger\bdPsi_1\qmat{Q}^{T}_{\bot}\qmat{A}\|$; here $\bdPsi_1=\bdPsi\qmat{Q}_{\bot}$ and $\bdPsi_2 = \bdPsi\qmat{Q}$ with $\qmat{Q}_\bot$ the orthogonal complement of $\qmat{Q}$. 
 Define  $$E_1=\left\{\bdPsi_2\mid \|\bdPsi_2^\dagger\|< e_3:=\frac{e\sqrt{d}}{d-s+1}t,\|\bdPsi_2^\dagger\|_F<e_4:=\sqrt{\frac{3s}{d-s+1}}t   \right\}$$ with some $t>1$. The deviations bounds in Lemmas \ref{lem:inverse_moment_F} and \ref{lem:inverse_moment_spectral} show that $\mathbb P(E_1^C)\leq 2t^{-(d-s)}$ for any fixed $\qmat{Q}$.
 
 We next use    Corollary \ref{col:SGT_deviation} to bound $\|\bdPsi_2^\dagger\bdPsi_1\qmat{Q}^{T}_{\bot}\qmat{A}\|$. However, $\qmat{Q}$ itself randomly depends on $\bdPhi,\bdOmega$, making the independence between $\bdPsi_1$   and $\qmat{Q}^{T}_{\bot}\qmat{A}$ questionable, such that Corollary \ref{col:SGT_deviation} cannot be directly applied.  Nevertheless, we can use the law of total probability to resolve it. To this end, fixed a pair  $(\bdOmega,\bdPhi)= (\bdOmega_{\gamma},\bdPhi_{\gamma})$; then $\qmat{Q}$ is also fixed, so $\bdPsi_1,\bdPsi_2,\qmat{Q}^{T}_{\bot}\qmat{A}$ are independent. Using Corollary \ref{col:SGT_deviation} we get
\begin{small}
\begin{align*}
   & \probleftright{\normSpectralnoleftright{\qmat{B}-\qmat{Q}^{T}\qmat{A}}> e_3\normFnoleftright{\qmat{Q}_{\bot}^{T}\qmat{A}}+(e_4+ue_3)\normSpectralnoleftright{\qmat{Q}_{\bot}^{T}\qmat{A}} \mid  (\bdOmega,\bdPhi)= (\bdOmega_{\gamma},\bdPhi_{\gamma})}\\
    \leq & e^{-u^2/2} +  \mathbb P(E_1^C) \leq  e^{-u^2/2} + 2t^{-(d-s)},
\end{align*}
\end{small}
\noindent which  uniformly holds for any fixed pair $(\bdOmega,\bdPhi)= (\bdOmega_{\gamma},\bdPhi_{\gamma})$.   This allows us to use the law of total probability to remove the condition:
\begin{small}
\begin{align*}
   & \probleftright{\normSpectralnoleftright{\qmat{B}-\qmat{Q}^{T}\qmat{A}}> e_3\normFnoleftright{\qmat{Q}_{\bot}^{T}\qmat{A}}+(e_4+ue_3)\normSpectralnoleftright{\qmat{Q}_{\bot}^{T}\qmat{A}}  } \\
    = &\! \int \!\!\probleftright{\normSpectralnoleftright{\qmat{B}-\qmat{Q}^{T}\qmat{A}}\!>e_3 \normFnoleftright{\qmat{Q}_{\bot}^{T}\qmat{A}}\! +\!(e_4+ue_3) \normSpectralnoleftright{\qmat{Q}_{\bot}^{T}\qmat{A}} \mid  (\bdOmega,\bdPhi)\!= \!(\bdOmega_{\gamma},\bdPhi_{\gamma})} d{\mathbb P}_{(\bdOmega,\bdPhi)}(\bdOmega_{\gamma},\bdPhi_{\gamma}) \\
    \leq & \int (e^{-u^2/2} +  2t^{-(d-s)})d{\mathbb P}_{(\bdOmega,\bdPhi)}(\bdOmega_{\gamma},\bdPhi_{\gamma})  = e^{-u^2/2} +  2t^{-(d-s)}.
\end{align*}
\end{small}
 Finally conditioned on $E$ gives the required probabilistic bound. 
\end{proof}


The previous preparations yield the main bound stated in Sect. \ref{sec:error_bound_q>=1_stated}.
\begin{proof}[Proof of Theorem \ref{thm:error_oblique_bound_q_gt_1}]
Use Theorem \ref{thm:prob_projection_error_q_gt_1}  and Lemma \ref{lem:GNError_deviation_bound}.
\end{proof}
\color{black}

\section{Related Approaches}\label{sec:related_approaches}
We discuss some relavant approaches in this section.

\emph{First},   \cite[Remark 4.2]{yu2018efficient} (see also \cite[Sect. 6.8]{bjarkason2019PassEfficientRandomizedAlgorithms}) suggests a ``half'' power iteration: If $\qmat{A}$ is given column-wisely, then one can compute  $\qmat{Y}=\qmat{A}\qmat{A}^{T}\bdOmega = (a_1a_1^{T} + a_2a_2^{T} + \cdots + a_na_n^{T})\bdOmega$ with only a single pass of $\qmat{A}$, with $a_i$ being  the $i$-th column of $\qmat{A}$. 

In contrast, SPI supports   any linear update of $\qmat{A}$  (see \cite{Practical_Sketching_Algorithms_Tropp} for a detail of linear update). Moreover, if $\qmat{A}$ is indeed given column-wisely, then the idea of \cite{yu2018efficient} can still be enhanced by SPI :   sketch  $\qmat{Z}=\qmat{A}\qmat{A}^{T}\bdPhi$ and $\qmat{Y}=\qmat{A}\qmat{A}^{T}\bdOmega$ with only a single pass of $\qmat{A}$, and   compute  $\hat{\qmat{Y}}=(\qmat{Z}\qmat{Z}^{T})^q\qmat{Y}$ as well.

\emph{Second}, the power of a sketch has been used to estimate the Schatten-$p$ norm of a matrix  \cite{martinsson2020RandomizedNumerical,kong2017spectrum}.  Let  $\qmat{Z}=\qmat{A}\bdPhi$ with $\bdPhi$ having isotropic columns and denote $\mathcal T(\qmat{Z}^T\qmat{Z})$ the strict upper triangle of   $\qmat{Z}^T\qmat{Z}$. Then $c\cdot \text{trace}(\mathcal T(\qmat{Z}^T\qmat{Z})^{p-1}\qmat{Z}^T\qmat{Z})$ gives an unbias estimation of  $\normP[2p]{\qmat{A}}^{2p}$, where $c$  only depends on $p$ and the sketch size. Following   their work, it would be interesting to perform SPI as $\hat{\qmat{Y}}=\qmat{Z}\mathcal T(\qmat{Z}^T\qmat{Z})^{q-1}\qmat{Z}^T\qmat{Y}$ to accelerate the computation. On the other hand, the classical estimation introduced in \cite[Sect. 5.3]{martinsson2020RandomizedNumerical} also   motivates a possible variant of SPI $\qmat{Z}_1\qmat{Z}_1^{T}\qmat{Z}_2\qmat{Z}_2^{T}\cdots \qmat{Z}_p\qmat{Z}_p^{T} \qmat{Y}$, whose   advantage   is that it is an unbias estimation of $(\qmat{A}\qmat{A}^{T})^p\qmat{Y}$.  
  \cite{li2014sketching} estimates $\normP[2p]{\qmat{A}}^{2p}$ using several bilinear sketches. 


 \color{black}

\emph{Third},   SPI has a subtle connection to the block-stochastic power method for streaming PCA \cite{mitliagkas2013memory}. This algorithm takes the following  spiked covariance   model: Assume that one has samples $x_1,\ldots,x_n$, with $x_i = \qmat{A}z_i + w_i\in\bbR^p$, where $z_i \sim N(0,I_k)$, $w_i\sim N(0,\qmat{I}_p)$, while $\qmat{A}\in\bbR^{p\times k}$ is unknown. The goal is to estimate the orthonormal basis of $\qmat{A}$, given that the samples $x_i$'s come  streamingly.  Given a block size $B$,   \cite{mitliagkas2013memory} (in a simplified form) takes the following recursion from $j=1,\ldots,n/B $:
\[
\qmat{S}_{j} = \texttt{QR}(\qmat{X}_j\qmat{X}_j^{T}\qmat{S}_{j-1},0  ),
\]
where $\qmat{S}_0\in\bbR^{p\times k}$ is a standard Gaussian matrix, and  $\qmat{X}_j$ takes the $j$-th block of the samples, i.e., $\qmat{X}_j = [x_{(j-1)B+1},\ldots,x_{jB} ]\in\bbR^{p\times B}$. Note that the setting of \cite{mitliagkas2013memory} has some connections  to SPI:    $\qmat{A}$ is not visible to the algorithms and the storage is constrained. However, the algorithms are different, and SPI only performs very few steps to enhance the rangefinder, while this algorithm requires several iterations       to approximate the orthonormal basis of $\qmat{A}$; it also requires an additional noise $w_i$; see \cite{mitliagkas2013memory} for more details. 


\color{black}
\emph{Fourth},   letting $q\rightarrow\infty$ in SPI algebraically recovers the subspace spanned by the leading $s$ singular vectors of $\qmat{Z}$. 
 This case has certain connections to the truncated-SVD‑based rangefinder approaches \cite{rokhlin2010randomized,martinssonRandomizedAlgorithmDecomposition2011,woolfe2008Fast},  but key distinctions are outlined below:

    1) Goals differ. In \cite{rokhlin2010randomized,martinssonRandomizedAlgorithmDecomposition2011,woolfe2008Fast}, the truncated-SVD-based rangefinder replaces the QR step and   final truncated   in TYUC17 (or TYUC17-SPI). Specifically,    this means applying SVD to  the rangefinder $\qmat{Y}\in\bbR^{m\times s}$ and keeping the  leading $r$  singular vectors as an orthonormal basis, where  $s-r$ is   small (e.g., $8,12$, and $20$  in \cite{woolfe2008Fast,rokhlin2010randomized,martinssonRandomizedAlgorithmDecomposition2011}). By contrast, SPI   derives  $\hat{\qmat{Y}}\in\bbR^{m\times s}$ from a considerably larger sketch $\qmat{Z}\in\bbR^{m\times l}$ (with $l-s=d>s$; see Sect. \ref{sec:sketch_size}) with only iterations $q=1,2$, mimicking the way of power iteration.  
    
    2) Even if one applies truncated SVD directly to $\qmat{Z}$ (to obtain its leading $s$ left singular subspace to play a similar role as $\hat{\qmat{Y}}$), key differences remain. First, the SVD-based approach  requires storage at least comparable to SPI---potentially larger with high-precision SVD algorithms. Second, in the mixed-precision strategy, one may use a half-double precision model instead of a single-double   one to obtain even larger sketches in half-precision (FP16), in which   
    FP16 has become ubiquitous in modern machine learning \cite{narang2017mixed}.    However, mainstream libraries (e.g.  LAPACK, MKL, Matlab) currently offer no native FP16 QR or SVD routines,  making SVD-based approach inapplicable. In contrast, in the same context,  SPI uses only FP16 matrix-multiplication kernels, which has been supported by  NVIDA's cuBLAS,   OpenCL BLAS, Tensorflow, and Matlab. \color{black}
   Third,  computing the SVD of  $\qmat{Z}$ can be more expensive. 
    \color{black} Meanwhile, empirical results in Sect. \ref{sec:numer_exp} imply that SPI   with only $q=1$ nearly attains the same subspace with substantially less overhead.

\color{black}

\color{black}

\emph{Fifth}, \cite{che2025efficient} recently   advocates replacing unfolding matrices in power iteration by column-subset based sketches to accelerate their multiple-pass algorithm for tensor approximation. Several essential points differ their work and ours:

1) Objective: Tensor unfoldings   are typically very fat, so   \cite{che2025efficient}  focuses on  \emph{accelerating} power iteration via approximate matrix multiplications, which naturally leads to column sampling for sketch construction \cite{drineas2006FastMonte}. By contrast, SPI is designed to \emph{enhance} approximation accuracy of one-pass algorithms via mimicking power iteration, in which subspace embedding techniques are therefore the appropriate choices. To that end, we also have  to consider momery-efficient strategies to reduce memory cost and derive sketch sizes a priori via theoretical guidance, which play a central role in   SPI.

2) Sketch: Even though one could use  column-subset based sketches in SPI, only uniform sampling is feasible, as others are data-dependent that require multiple passes over the data. Unfortunately,   uniform sampling demands the strong incoherence assumption to work well, in which SRHT is more preferred \cite{avron2010BlendenpikSupercharging}.

3) Error bound:   \cite{che2025efficient}'s bound (\cite[Corollary 3]{che2025efficient})  depends explicitly on the unfolding matrix's  sizes, rank, and Frobenius norm. In contrast, by leveraging   recent advances in subspace embedding techniques \cite{FindingStructureHalko,tropp2023RandomizedAlgorithms,martinsson2020RandomizedNumerical}, our bound  depends only on the sketch sizes and tail energy, decoupling accuracy guarantees from the ambient matrix size.

\color{black}

Finally, with storage budget constraints, it is also possible to implement TYUC17 in mixed-precision to enlarge sketches. The case is as follows.  $\qmat{W}$ may indeed be obtained   in lower precision to accommodate a larger size. However,   $\qmat{Y}$ must still be sketched and maintained in higher precision: If $\qmat{Y}$ were stored in, say, single-precision, then $\qmat{Q}$ would necessarily also be single-precision (since otherwise $\qmat{Q}$   requires twice the storage of $\qmat{Y}$, which is infeasible under  budget constraints; see Fig. \ref{fig:TYUC17_vs_SPI_tikz}). However, a single-precision $\qmat{Q}$ suffers from a marked  loss of orthonormality ($\normSpectralnoleftright{\qmat{Q}^{T}\qmat{Q}-\qmat{I}} \approx  1e-6 $, compared to $10^{-15}$ for double-precision).  If one is willing to accept a single-precision 
$\qmat{Q}$, TYUC17-SPI may still achieve improved performance by replacing the single-double precision model in Sect. \ref{sec:same_storage_as_PSA} with a half-single precision model (i.e., using half-precision sketches).  We  will describe the formal algorithm in Sect. \ref{sec:performance_SPI_variant}, and  compare it with the SPI variant (Sect. \ref{sec:reduce_storage_cost_q_gt_1}).

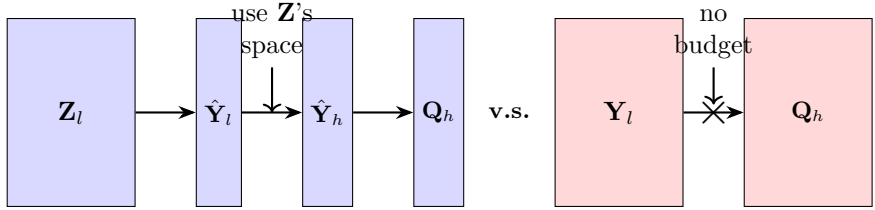
\begin{figure}[htbp]
    \centering
\begin{tikzpicture}[
    node distance=0.8cm and 0.8cm, 
    box_top/.style={
        rectangle,
        draw,
        fill=blue!15,
        align=center,
        minimum height=2.5cm,
        minimum width=1.7cm,
        inner sep=5pt
        },
    box_bottom/.style={
        rectangle,
        draw,
        fill=red!15,
        align=center,
        minimum height=2.5cm,
        minimum width=1.7cm,
        inner sep=5pt
        },
    narrowbox/.style={
        rectangle,
        draw,
        fill=blue!15,
        align=center,
        minimum height=2.5cm,
        minimum width=0.6cm,
        inner sep=3pt
        },
    arrow/.style={->, >=Stealth, thick},
    impossibility/.style={
        font=\LARGE,
        inner sep=0pt
    },
        vs_label/.style={ 
        font=\small\bfseries 
    }
]

\node (Z_low) [box_top] {$\qmat{Z}_{{l}}$};
\node (Yhat_low) [narrowbox, right=of Z_low] {\small$\hat{\qmat{Y}}_{{l}}$};
\node (Yhat_high) [narrowbox, right=of Yhat_low] {\small$\hat{\qmat{Y}}_{{h}}$};
\node (ell_high) [narrowbox, right=of Yhat_high] {\small$\qmat{Q}_{{h}}$};

\draw [arrow] (Z_low.east) -- (Yhat_low.west);
\draw [arrow] (Yhat_high.east) -- (ell_high.west);
\draw [arrow] (Yhat_low.east) -- (Yhat_high.west);

\coordinate (mid_top) at ($(Yhat_low.east)!0.5!(Yhat_high.west)$);
\node (z_space_label) [above=0.6cm of mid_top, align=center] {use $\qmat{Z}$'s \\ space};
\draw [->, thick] (z_space_label.south) -- (mid_top);

\node (vs) [vs_label, right=0.2cm of ell_high] {v.s.};

\node (Y_low) [box_bottom, right=0.2cm of vs] {$\qmat{Y}_{{l}}$};

\node (Y_high) [box_bottom, right=of Y_low] {\small$\qmat{Q}_{{h}}$};

\path (Y_low.east) -- (Y_high.west) node [midway, impossibility] {$\times$};

\draw [arrow] (Y_low.east) -- (Y_high.west);

\coordinate (mid_top1) at ($(Y_low.east)!0.5!(Y_high.west)$);
\node (not_space_label) [above=0.6cm of mid_top1, align=center] {no  \\budget};
\draw [->, thick] (not_space_label.south) -- ([yshift=3pt]mid_top1);
\end{tikzpicture}
\caption{Comparison of generating $\qmat{Q}$ by SPI (left) and by ordinary method (right) in mixed-precision.  Subscripts ``$l$'' and ``$h$''   denote   lower and higher precision.  SPI exploits $\qmat{Z}$'s storage to convert from low to high precision. The right part shows that directly obtaining a higher precision $\qmat{Q}$ form a lower precision $\qmat{Y}$ is impossible due to storage dudget constraint.}\label{fig:TYUC17_vs_SPI_tikz}
\end{figure}


\color{black}

\section{Conclusions}\label{sec:conclusion} 
In this work, we introduced the sketch-power iteration (SPI) as an effective approach to enhance one-pass   algorithms for low-rank   approximation. By leveraging sketches, SPI substantially improves accuracy without requiring multiple passes over the data matrix. Mixed-precision and two hardware-dependent strategies were proposed to reduce the storage cost. Theoretical analysis established rigorous error bounds, which   provide practical guidance on setting   sketch sizes a priori for different singular spectrum decay.   Numerical results demonstrate the improved accuracy of TYUC17 enhanced by SPI in approximating the dominant singular values and singular vectors. We also observed  that the deduced sketch sizes well match empirical optima, enabling efficient parameter selection for better approximation accuracy. Future research may  better understand the theoretical behavior of SPI and explore its  potential combinations     into one-pass algorithms for eigenvalue decompositions and tensor approximations. 

\color{black}
{\footnotesize \section*{Acknowledgments} We are very grateful to Professor Joel A. Tropp for his constructive feedback---particularly for pointing out the connection between SPI and the moment method for estimating Schatten-$p$ norms. We   also       thank  the associate editor and the anonymous referees for their   careful reading and insightful suggestions, which have helped us substantially  improve  the quality and clarity of this manuscript.
}
\color{black}

\appendix

\section{Proofs of Technical Lemmas} \label{sec:proofs_technical_lemmas}
 \begin{proof}[Proof of Corollary \ref{col:SigmaPhi2Phi1_second_moment_bound}]
        The first relation   comes from the proof of \cite[Thm. 10.5]{FindingStructureHalko}. Here we prove the second relation:
        \begin{small}
        \begin{align*}
            \ExpectationNoBracket[\qmat{G}]{\|\qmat{A}\qmat{G}_2\qmat{G}_1^\dagger \|^2}&=\ExpectationNoBracket[\qmat{G}_1 ]{\ExpectationNoBracket[\qmat{G}_2]{\|\qmat{A}\qmat{G}_2\qmat{G}_1^\dagger \|^2}}
            \leq \mathbb E_{\qmat{G}_1}{\left(\normSpectral{\qmat{A}}\|\qmat{G}_1^\dagger\|_F+\normF{\qmat{A}}\| \qmat{G}_1^\dagger\|\right)^2}\\
           &\leq \normSpectral{\qmat{A}}^2\ExpectationNoBracket{\normFnoleftright{\qmat{G}_1^\dagger}^2}+\normF{\qmat{A}}^2\ExpectationNoBracket{\normSpectralnoleftright{\qmat{G}_1^\dagger}^2}+2\normSpectral{\qmat{A}}\normF{\qmat{A}} (\ExpectationNoBracket{\normFnoleftright{\qmat{G}_1^\dagger}^2})^{1/2}(\ExpectationNoBracket{\normSpectralnoleftright{\qmat{G}_1^\dagger}^2})^{1/2}\\
            &=\left(\normSpectral{\qmat{A}}(\ExpectationNoBracket{\normFnoleftright{\qmat{G}_1^\dagger}^2})^{1/2}+\normF{\qmat{A}}(\ExpectationNoBracket{\normSpectralnoleftright{\qmat{G}_1^\dagger}^2})^{1/2}\right)^2,
        \end{align*}  
    \end{small}
   \noindent     where the first inequality is from   Lemma \ref{lem:SGT} and the second one is due to H\"older inequality. We then use   Lemmas \ref{lem:inverse_moment_F} and \ref{lem:inverse_moment_spectral} to complete the proof.
    \end{proof}

     \begin{proof}[Proof of Corollary \ref{col:SGT_deviation}]
        Conditioned on $E$, Lemma \ref{lem:SGT} together with Jensen inequality shows that 
        \begin{align*}
            \Expectation[\qmat{G}]{\normSpectral{\qmat{S}\qmat{G}\qmat{T}}\mid E}<\normF{\qmat{S}}\normSpectral{\qmat{T}}+\normSpectral{\qmat{S}}\normF{\qmat{T}}\leq e_1e_4+e_2e_3,
        \end{align*}
        which in connection with Lemma \ref{lem:LipschitzConcentrationIneq}   gives 
        \begin{align*}
            \probleftright{\normSpectral{\qmat{S}\qmat{G}\qmat{T}}>\Expectation[\qmat{G}]{\normSpectral{\qmat{S}\qmat{G}\qmat{T}}\mid E}+\normSpectral{\qmat{S}}\normSpectral{\qmat{T}}u\mid E}<e^{-u^2/2},
        \end{align*}
       and so
         $
            \probleftright{\normSpectral{\qmat{S}\qmat{G}\qmat{T}}>e_1e_4+e_2e_3+e_2e_4u\mid E}<e^{-u^2/2} 
         $.
        Use   $\probleftright{E^C}$ to remove the condition and obtain the final result.
     \end{proof}  

 \begin{proof}[Proof of Lemma \ref{lem:GNErrorFrobenius}]
   The Frobenius case is similar to that of \cite[Thm. 4.3]{Practical_Sketching_Algorithms_Tropp}. For  the spectral case, first, we use the triangle inequality to obtain:
    \begin{small}
    \begin{align}\label{eq:proof:lemma_QB:1}
        \normSpectral{\qmat{A}-\qmat{Q}\qmat{B}}&\leq\normSpectral{\qmat{A}-\qmat{Q}\qmat{Q}^{T}\qmat{A}}+\normSpectral{\qmat{Q}\qmat{Q}^{T}\qmat{A}-\qmat{Q}\qmat{B}} \leq\normSpectral{\qmat{A}-\qmat{Q}\qmat{Q}^{T}\qmat{A}}+\normSpectral{\qmat{Q}^{T}\qmat{A}-\qmat{B}}.
    \end{align}
\end{small}
        \cite[Lemma A.4]{Practical_Sketching_Algorithms_Tropp} shows that $
        \qmat{B}-\qmat{Q}^{T}\qmat{A}=\bdPsi_2^\dagger\bdPsi_1\qmat{Q}_{\bot}^{T}\qmat{A}$, 
    where $\bdPsi_1=\bdPsi\qmat{Q}_{\bot}$ and $\bdPsi_2=\bdPsi\qmat{Q}$ are independent and standard Gaussian, with $\qmat{Q}_{\bot}$ the orthonormal complement of $\qmat{Q}$. It then follows from Lemma \ref{lem:SGT} that
    \begin{small}
    \begin{align}\label{eq:proof:lemma_QB:2}
        \mathbb E_{\bdPsi} {\normSpectralnoleftright{\qmat{B}-\qmat{Q}^{T}\qmat{A}}}&=\mathbb E_{\bdPsi} {\normSpectral{\bdPsi_2^\dagger\bdPsi_1\qmat{Q}_{\bot}^{T}\qmat{A}}} \leq \mathbb E {\normFnoleftright{\bdPsi_2^\dagger}}\normSpectral{\qmat{Q}_{\bot}^{T}\qmat{A}}+\mathbb E {\normSpectralnoleftright{\bdPsi_2^\dagger}}\normF{\qmat{Q}_{\bot}^{T}\qmat{A}}. 
    \end{align}
\end{small}
It follows $\normSpectralnoleftright{\qmat{Q}_{\bot}^{T}\qmat{A}}=\normSpectralnoleftright{\qmat{Q}_{\bot}\qmat{Q}_{\bot}^{T}\qmat{A}}$ and  $\normFnoleftright{\qmat{Q}_{\bot}^{T}\qmat{A}}=\normFnoleftright{\qmat{Q}_{\bot}\qmat{Q}_{\bot}^{T}\qmat{A}}$, while $\qmat{Q}_{\bot}\qmat{Q}_{\bot}^{T}\qmat{A}=\qmat{A}-P_\qmat{Q}\qmat{A}$. On the other hand,  by Lemmas \ref{lem:inverse_moment_F}, \ref{lem:inverse_moment_spectral}, and Jensen's inequality, we have:
\begin{gather*}
    \ExpectationNoBracket{\normSpectralnoleftright{\bdPsi_2^\dagger}}\leq\sqrt{\ExpectationNoBracket{\normSpectralnoleftright{\bdPsi_2^\dagger}^2}}\leq\frac{e\sqrt{d+s}}{\sqrt{2}(d-s)}\\
    \ExpectationNoBracket{\normFnoleftright{\bdPsi_2^\dagger}}\leq\sqrt{\ExpectationNoBracket{\trace{\bdPsi_2\bdPsi_2^{T}}^{-1}}}=\sqrt{\trace{\ExpectationNoBracket{ (\bdPsi_2\bdPsi_2^{T} )^{-1}}}}=\sqrt{\frac{s}{d-s-1}},
\end{gather*}
which together with \eqref{eq:proof:lemma_QB:1} and \eqref{eq:proof:lemma_QB:2} yields the result. 
\end{proof}

\section{Sketch Size Deduction via Theorem \ref{thm:oblique_proj_error_q=1}}\label{sec:sketch_size_guidance}
This  section details how to deduce the sketch sizes in   Sect. \ref{sec:sketch_size} under the mixed-precision strategy. Recall that  
  the storage budget is parameterized by $T=(c(l+s)+d)/2$. For simplicity we only consider $m=n$, i.e., $c=1$  case.  To make the derivation go through,  we have to    simplify the   bound in Theorem \ref{thm:oblique_proj_error_q=1}. Assume that the last two terms are higher-order ones. 
 For the first term,   relax the constant $1$ to $2\hat{\xi}^2$ and assume that $\hat{\xi} <C$ for some   $C>0$. The bound can be simplified as 
\begin{align}\label{eq:simplified_error_bound_q_equals_1}
    \Expectation{\normF{\qmat{A}-\qmat{Q}\qmat{B}}^2|E_F}\lesssim 
     O(1)\cdot \min_{s\leq d, s+d=T/2,r\leq\varrho\leq s}\frac{d}{d-s-1} \frac{l}{l-\varrho - 1}\tau_{\varrho +1}^2.
\end{align}
Recall  that Sect. \ref{sec:same_storage_as_PSA} implies $l=d+s=T$ when $m=n$; thus we only need to determine $s$.    Assume  that $T> 2r$ for convenience in the sequel. 


\subsection{Flat spectrum} In this case, the tail energy $\tau_{\varrho +1}^2$ is the dominant term, for which $\varrho\approx r$ to minimize the bound. Thus   \eqref{eq:simplified_error_bound_q_equals_1} remains to minimize
\[
    \min_{s\leq d, s+d=T ,r\leq   s}\frac{d}{d-s-1} =\min_{r\leq s}\frac{T-s}{T-2s-1}, 
\]
which is an increasing function of $s$, whose minimizer is ottained at $s=r$.

\subsection{Polynomial spectrum decay}
In this case, write $\sigma_i = i^{-\alpha}$ for some $\alpha>0$. When $\alpha\neq 1/2$, $\tau_{\varrho +1}^2= \sum^n_{i=\varrho+1}i^{-2\alpha}\approx \int_{\varrho+1}^{n}x^{-2\alpha}{\rm d}{x}= C_{2\alpha}(n^{1-2\alpha}-(\varrho+1)^{1-2\alpha}) $ with $C_{2\alpha} = (1-2\alpha)^{-1}$. Due to the different behavior of $\tau_{\varrho +1}^2$ for different $2\alpha$, we consider the argument separately.

When $\alpha <1/2$,   $\tau_{\varrho +1}^2\approx n^{1-2\alpha}$ for large enough $n$, and so the problem reduces to $\min_{  s+d=T ,r\leq \varrho\leq s\leq d} \frac{d}{d-s-1} \frac{l}{l-\varrho - 1} n^{1-2\alpha}$. The minimum occurs at $s=\varrho=r$.

When  $\alpha> 1/2$,   $\tau_{\varrho +1}^2\approx 1/(\varrho+1)^{2\alpha-1}-1/n^{2\alpha-1}\approx  1/(\varrho+1)^{2\alpha-1}$, and so the problem reduces to 
\begin{align}\label{eq:poly_decay_min_s_rho_alpha_geq_1}
    \min_{ s+d=T ,r\leq \varrho\leq s\leq d} \frac{d}{d-s-1} \frac{l}{(l-\varrho - 1)(\varrho+1)^{2\alpha-1}} .
\end{align}
We first minimize $\varrho$ without constraints, which takes $\varrho=\frac{(2\alpha-1)(l+1)-1}{2\alpha}$. Consider two cases: 1) $s\geq\frac{(2\alpha-1)(l+1)-1}{2\alpha}$: the minimizer of \eqref{eq:poly_decay_min_s_rho_alpha_geq_1} with respect to $\varrho$ is $\varrho =\frac{(2\alpha-1)(l+1)-1}{2\alpha}$, and after some calculations, solving  \eqref{eq:poly_decay_min_s_rho_alpha_geq_1}  gives $s=\varrho={\rm P}_{[r,T/2]}(\frac{(2\alpha-1)(l+1)-1}{2\alpha})$; here ${\rm P}_{[r,T/2]}(\cdot)$ means the projection   onto   $[r,T/2]$. 2)   $s< \frac{(2\alpha-1)(l+1)-1}{2\alpha}$:    the minimizer of \eqref{eq:poly_decay_min_s_rho_alpha_geq_1} with respect to $\varrho$ is $\varrho=s$, and \eqref{eq:poly_decay_min_s_rho_alpha_geq_1}  reduces to
\begin{small}
\begin{align*}
    \min_{  s+d={T},r\leq  s\leq \min\{ \frac{(2\alpha-1)(l+1)-1}{2\alpha}, \frac{T}{2} \} } \frac{d}{d-s-1} \frac{l}{(l-s - 1)(s+1)^{2\alpha-1}} \approx   \frac{l}{(T-2s-1)(s+1)^{2\alpha-1}} ,
\end{align*}
\end{small}
where the approximation is due to $d=l-s\approx l-s-1$, whose minimizer occurs at   $s= \max\{r, \frac{(2\alpha-1)({T+3})-2}{4\alpha} \}$. After some simple comparisons, this is the minimizer of \eqref{eq:poly_decay_min_s_rho_alpha_geq_1}.


When $\alpha\approx 1/2$,   $\tau_{\varrho +1}^2\approx \log n- \log (\varrho+1)$, and so the problem reduces to $$\min_{s\leq d, s+d=T,r\leq \varrho\leq s} \frac{d}{d-s-1} \frac{l}{l-\varrho - 1} (\log n-\log (\varrho+1)).$$  We first minimize $\varrho$ without constraints, which satisfies  $    \log (\frac{n}{\varrho+1}) = \frac{l- \varrho - 1 }{\varrho+1}$. It has no closed-form solution; however, its solution can be represented as $\varrho =  -\frac{l}{W(-{l}/{ne})}-1$, where $W(\cdot)$  is the   Lambert W function.   Then, similar  to the discussions in the $\alpha>1/2$ case, we find that   $s \approx  {\rm P}_{[r,T/2]}(-\frac{T+1 }{2W(- {(T+1) }/{2ne})}-1) $.  





\subsection{Exponential spectrum decay}
In this case, write $\sigma_i = e^{-\alpha i}$ for some $\alpha>0$. Then $\tau_{\varrho +1}^2\approx\int_{\varrho+1}^{n}e^{-2\alpha x}{\rm d}x\approx C_{\alpha}(e^{-2\alpha(\varrho +1)} -e^{-2\alpha(n+1)})\approx C_{\alpha}e^{-2\alpha(\varrho +1)}$ where $C_{\alpha} = (1-e^{-2\alpha})^{-1}$. Substituting this into   \eqref{eq:simplified_error_bound_q_equals_1} yields:
\begin{small}
\begin{align}\label{eq:exp_decay_min_s_rho}
    \min_{s\leq d, s+d=T,r\leq \varrho\leq  s}  \frac{d}{d-s-1} \frac{l}{l-\varrho - 1}  e^{-2\alpha(\varrho +1)}  .
\end{align}
\end{small}
Its minimum with respect to $\varrho$     occurs at $\varrho=l- (2\alpha)^{-1}$. Then consider three cases:

1. $r\leq l- (2\alpha)^{-1}\leq s$: Now minimizing \eqref{eq:exp_decay_min_s_rho} with respect to $\varrho$   gives  $\varrho=l- (2\alpha)^{-1}$, and so the problem reduces to $\min_{s\leq d,s+d=T,l-(2\alpha)^{-1}\leq s}\frac{d}{d-s-1}$; thus $s=\min\{l-(2\alpha)^{-1},T/2   \}=\min\{T-(2\alpha)^{-1},T/2   \} $.

2. $s\leq l- (2\alpha)^{-1}$: Now minimizing \eqref{eq:exp_decay_min_s_rho} respect to $\varrho$   takes the minimum  $\varrho=s$, and so the problem reduces to 
\begin{align*}
    \min_{s\leq d,s+d=T }\frac{d}{d-s-1}\frac{l}{l-s - 1}e^{-2\alpha(s +1)}\approx \frac{l\cdot e^{-2\alpha(s +1)}}{d-s-1},
\end{align*}
where in the approximation, we used that  $d=T-s\approx l-s-1$.  The minimum occurs at $s\approx\max\{T/2 - (2\alpha)^{-1},r\} $.

3. $r\geq l- (2\alpha)^{-1}$: Now minimizing \eqref{eq:exp_decay_min_s_rho} with respect to $\varrho$   takes the minimum  $\varrho=r$, and so the problem reduces to $\min_{s\leq d,s+d=T }\frac{d}{d-s-1}$; thus $s=r$.

In summary, if $\alpha$ is very small (say, $\alpha\leq \frac{1}{2T}$),  then $s=r$; Otherwise, $s\approx T/2$. 

Finally, we   remark that although the bound in Theorem \ref{thm:oblique_proj_error_q=1} is for $q=1$ case, empirically we observe that the derived parameters work well for $q\geq 1$.


\clearpage
\section*{Supplementary Materials}
\addcontentsline{toc}{section}{Supplementary Materials}
\setcounter{section}{0}
\setcounter{subsection}{0}
\setcounter{equation}{0}
\setcounter{figure}{0}
\setcounter{table}{0}
\setcounter{algorithm}{0}
\renewcommand{\thesection}{S\arabic{section}}
\renewcommand{\thesubsection}{\thesection.\arabic{subsection}}
\renewcommand{\theequation}{S\arabic{equation}}
\renewcommand{\thefigure}{S\arabic{figure}}
\renewcommand{\thetable}{S\arabic{table}}
\renewcommand{\thealgorithm}{S\arabic{algorithm}}

\section*{Outline of the Supplementary Materials} In Section \ref{SMsec:para_guid_m_neq_n}, we illustrate the parameter guidance for TYUC17-SPI when $m\neq n$. In Section \ref{SMsec:insensitivity_test_matrix}, we show that TYUC17-SPI is insensitive to the test matrices, using Gaussian, Rademacher, Count-Sketch, and sparse  random test matrices as examples.  Section \ref{SMsec:JL_property_SparseRademacher} shows that the sparse Rademacher test matrices used in the main manuscript exhibits JL behavior empirically.   In Section \ref{sec:performance_SPI_variant}, we recall  the TYUC17-SPI variant in Section \ref{sec:reduce_storage_cost_q_gt_1} and the mixed-precision version of TYUC17 mentioned at the end of Section \ref{sec:related_approaches}, and compare their performance. In Section \ref{sec:TYUC19_SPI_main}, we show how to integrate SPI into the streaming algorithm of Tropp et al. \cite{troppStreamingLowRankMatrix2019}, preliminarily study its empirical performance, and provide a proof idea for establishing the deviation bound.

\section{Parameter Guidance of TYUC17-SPI when $m\neq n$} \label{SMsec:para_guid_m_neq_n}

Table \ref{SMtab:sketch_size_guidance_m_neq_n} presents the parameter of TYUC17-SPI when  $m=cn$ where $c\neq 1$. This is an extension of that in Section \ref{sec:sketch_size} in the main manuscript. The derivation is similar to those in Section \ref{sec:sketch_size_guidance} and is omitted. 

  \setlength{\tabcolsep}{4pt}
\setlength{\cmidrulewidth}{0.5pt}
  \begin{table}[htbp] 
    \centering
    \caption{Sketch size $s$ selection in different spectrum decay;   $r\leq s\leq T/(c+1)$}
      \begin{mytabular2}{c!{\vline width 1pt}c|c|c!{\vline width 1pt}c|c}
        \toprule
        \multicolumn{1}{c}{{Flat}} & \multicolumn{3}{c}{Poly ($\sigma_i=i^{-\alpha}$)} & \multicolumn{2}{c}{Exp ($\sigma_i = e^{-\alpha\cdot i}$)} \\
        \cmidrule(lr){1-1} \cmidrule(lr){2-4} \cmidrule(lr){5-6} 
    &  {$\alpha<\frac{1}{2}$} & {$\alpha\approx\frac{1}{2}$} & {$\alpha>\frac{1}{2}$} & {$\alpha<\frac{1}{2T}$} & {$\alpha\geq \frac{1}{2T}$}     \\
      \hline
      $r$ & $r$ &  ${\rm Proj}_{[r,\frac{T}{c+1}]}\bigxiaokuohao{ -\frac{\frac{T+c}{c+1}}{W\!\left(-\frac{T+c}{(c+1)ne }\right)} - 1   }$ & $\max\{r, \frac{(2\alpha-1)({T+3})-(c+1)}{2(c+1)\alpha} \}$ & $r$ & $  \frac{T}{c+1}$\\
      \bottomrule
      \end{mytabular2}%
      \label{SMtab:sketch_size_guidance_m_neq_n}%
  \end{table}%

\begin{figure}[htp]
    \centering
    \captionsetup{font=tiny}
    \begin{subfigure}{0.3\textwidth}
        \centering
        \includegraphics[width=\linewidth]{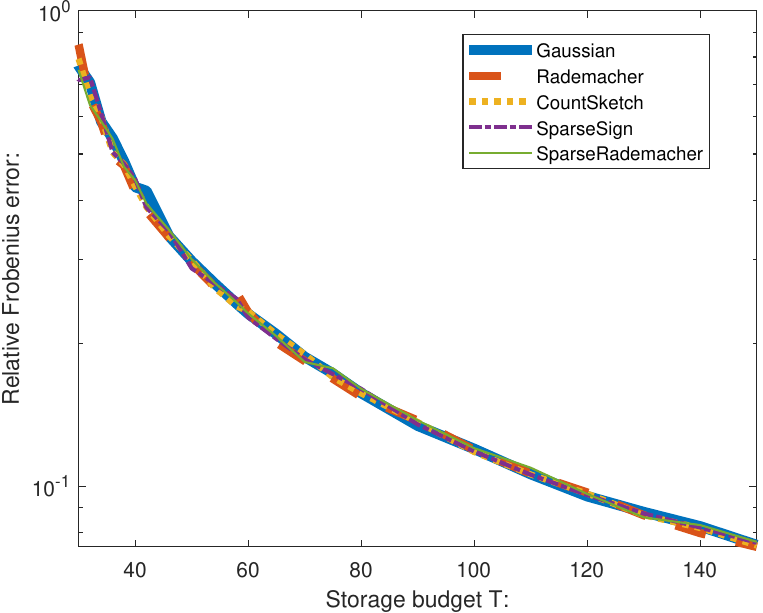}
        \caption{Low Rank with Low Noise}
    \end{subfigure}%
    \begin{subfigure}{0.3\textwidth}
        \centering
        \includegraphics[width=\linewidth]{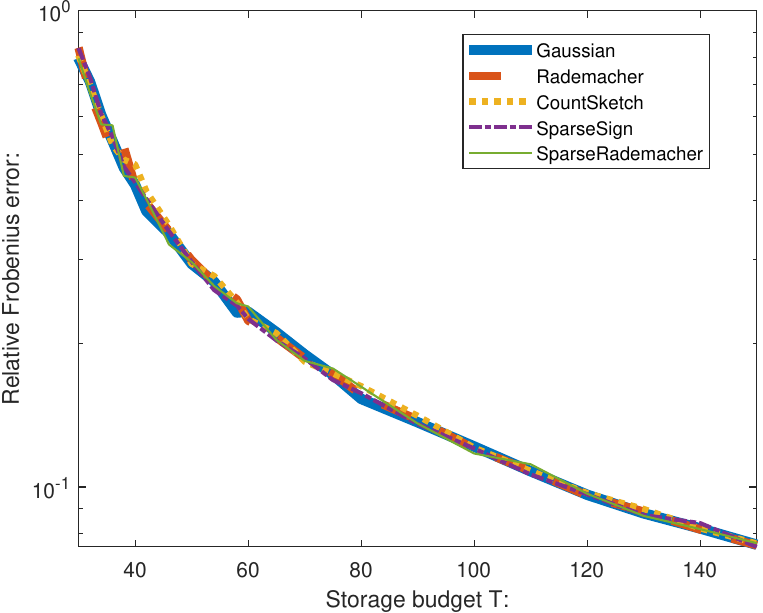}
        \caption{Low Rank with Medium Noise}
    \end{subfigure}%
    \begin{subfigure}{0.3\textwidth}
        \centering
        \includegraphics[width=\linewidth]{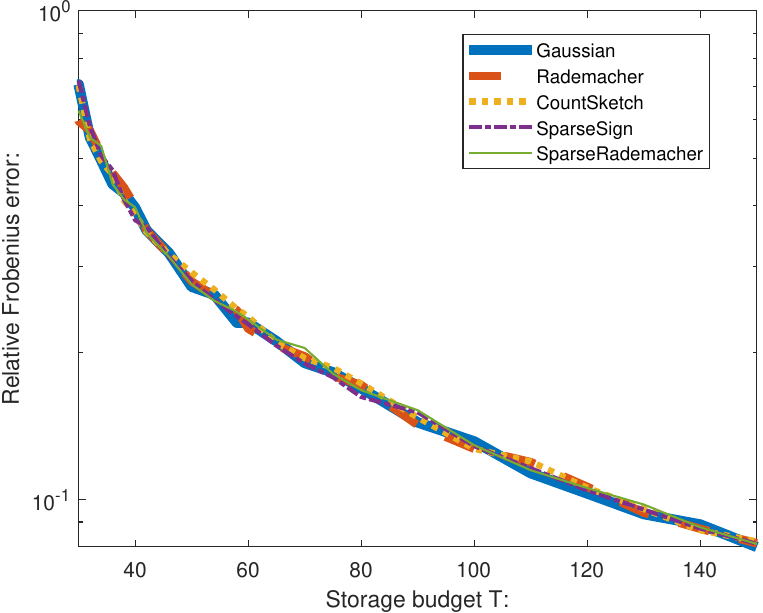}
        \caption{Low Rank with High Noise}
    \end{subfigure}

    \vspace{0.1cm} 

    \begin{subfigure}{0.3\textwidth}
        \centering
        \includegraphics[width=\linewidth]{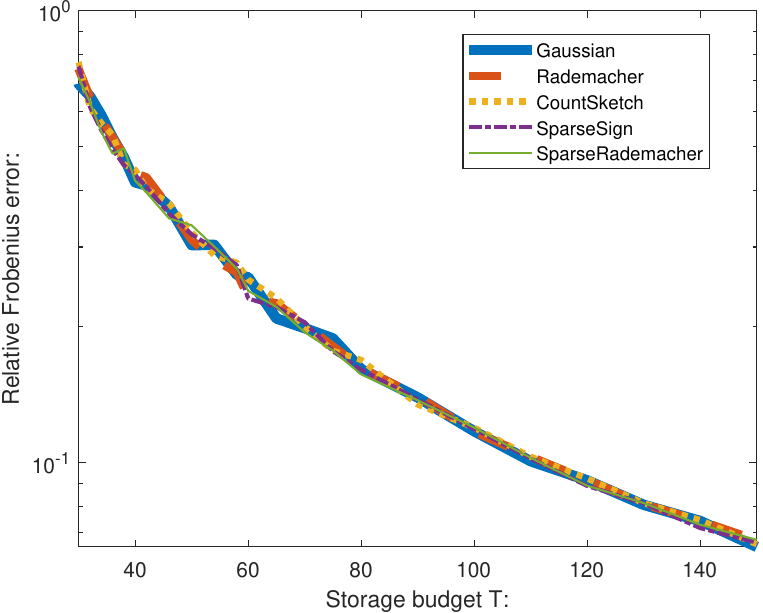}
        \caption{Slow polynomial decay}
    \end{subfigure}%
    \begin{subfigure}{0.3\textwidth}
        \centering
        \includegraphics[width=\linewidth]{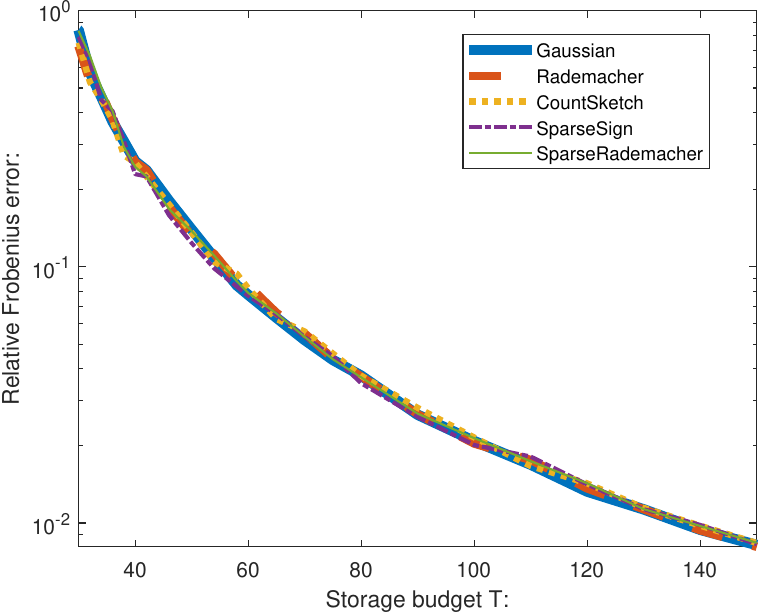}
        \caption{Medium polynomial decay}
    \end{subfigure}%
    \begin{subfigure}{0.3\textwidth}
        \centering
        \includegraphics[width=\linewidth]{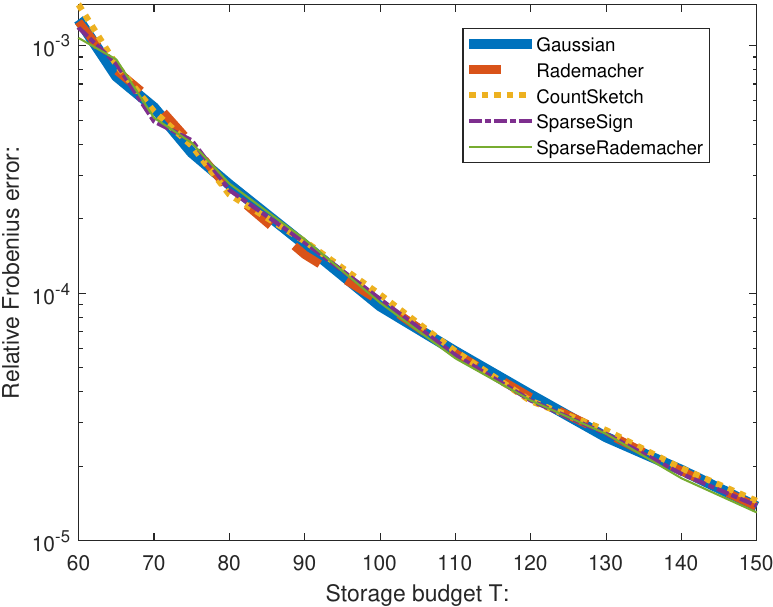}
        \caption{Fast polynomial decay}
    \end{subfigure}

    \vspace{0.1cm} 

    \begin{subfigure}{0.3\textwidth}
        \centering
        \includegraphics[width=\linewidth]{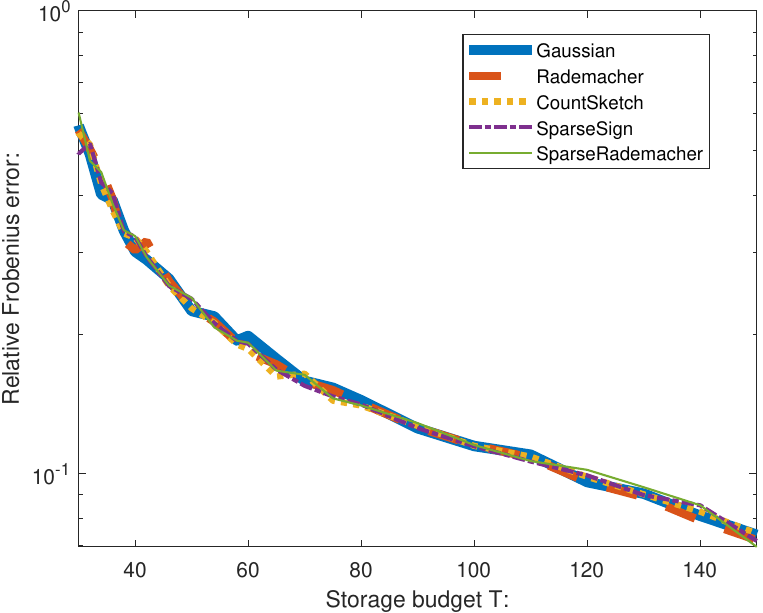}
        \caption{Slow exponentially decay}
    \end{subfigure}%
    \begin{subfigure}{0.3\textwidth}
        \centering
        \includegraphics[width=\linewidth]{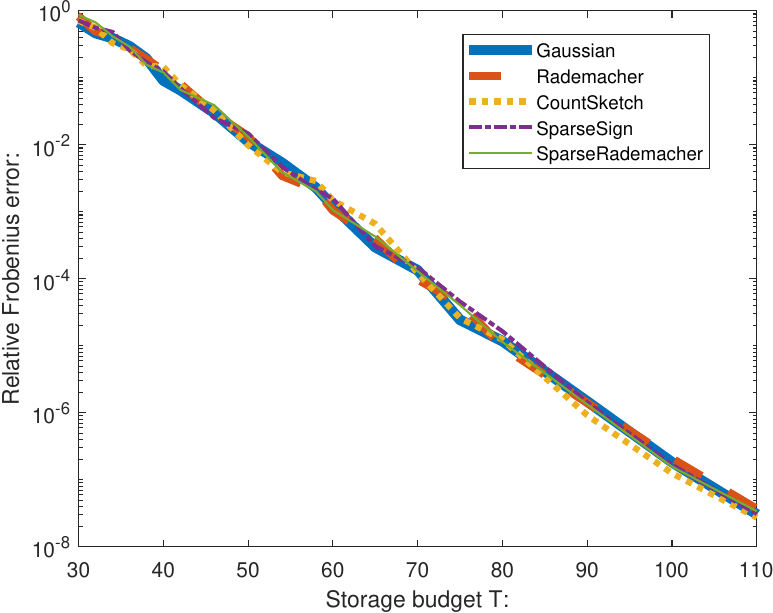}
        \caption{Medium exponentially decay}
    \end{subfigure}%
    \begin{subfigure}{0.3\textwidth}
        \centering
        \includegraphics[width=\linewidth]{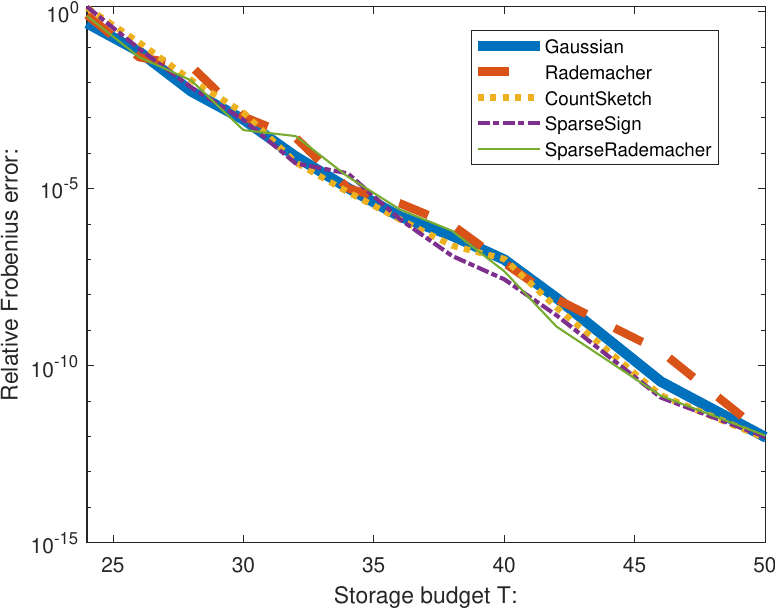}
        \caption{Fast exponentially decay}
    \end{subfigure}

    \caption{Figures of Synthetic data. X-axis means the $T$ represented the storage budget. Y-axis means the relative Frobenius error $S_F$. All results are averaged by 20 independent repeated experiments.}\label{SMfig:testMatrix_Synectic_data}
\end{figure}

\section{Insensitivity of the SPI to Test Matrices}\label{SMsec:insensitivity_test_matrix}
Our parameter guidance is based on the assumption that the test matrices are Gaussian matrices. However, in practice, we can use other test matrices such as countsketches and sparse Rademacher matrices, which costs little extra storage to hold in memory. In Section \ref{sec:numer_exp} of the main manuscript we have used sparse Rademacher test matrices. Here we use a numerical example to show that   TYUC17-SPI is insensitive to the test matrices. 
\begin{example} 
    In this example, we expand the experiment from numerical Example \ref{eg:synthetic_data}. 
    We used the same synthetic data as in Section \ref{sec:synthetic_data_def}  but with different test matrices including Gaussian, CountSketch, sparse sign \cite[Sect. 9.2]{martinsson2020RandomizedNumerical}, and sparse Rademacher matrices ($99\%$ sparsity). To be specific, we only show the case of $q=1$ in SPI.
\end{example}

Figure \ref{SMfig:testMatrix_Synectic_data} shows that the results are almost identical for different test matrices except for perturbation by randomness when data matrix has more flatten spectrum decay. When the decay is more rapid, Gaussian embedding is slightly more stable. This observation allows us to use sparse   matrices as test matrices in TYUC17-SPI in most cases, which costs little storage while succeeds the theoretical advantages for Gaussian matrices, as those obtained in Section \ref{sec:analysis} and \ref{sec:error_anal_q>1}. As a result, we can omit the storage cost of the test matrices   in the derivation of the parameter guidance. 

\section{Numerical Experiments on JL Property of SparseRademacher Matrices}\label{SMsec:JL_property_SparseRademacher} 
In this section, we provide numerical evidence to show that SparseRademacher matrices may satisfy the Johnson-Lindenstrauss (JL) property in high-dimensional spaces in practice. We conduct experiments to demonstrate that SparseRademacher matrices can effectively preserve the distances between points in high-dimensional space when projected into a lower-dimensional space.
We test the JL property by checking the subspace embedding property \cite[Prop. 5.2]{kireeva2024RandomizedMatrix}:
    \begin{proposition}[c.f. Prop. 5.2 of \cite{kireeva2024RandomizedMatrix}]
Suppose that   $\qmat{U} \in \mathbb{R}^{n \times d}$ is a matrix with orthonormal columns. For a matrix $\bdPhi \in \mathbb{R}^{s \times n}$, the subspace embedding property   is equivalent with the condition
\[
1 - \varepsilon \;\leq\; \sigma_{\min}(\bdPhi \qmat{U}) \;\leq\; \sigma_{\max}(\bdPhi \qmat{U}) \;\leq\; 1 + \varepsilon.
\]
\end{proposition}
\begin{example}
 We set $n=2000$, $d=50$ and $s=100$. The sparsity level $\xi$ of sparse test matrices is set from $1e-4$ to $1$ (including the $1\%$ sparsity used in the main manuscript). All experiments are averaged by $1000$ independent repeated experiments. We compare two types of sparse embeddings: (1) SparseRademacher matrix containing $nd\times \xi$ nonzero entries, the locations of nonzero entries are randomly selected, and the values of nonzero entries are independently set to $+1$ or $-1$ with probability $0.5$. (2) SparseIID matrix \cite{troppComparisonTheoremsMinimum2025}, where each entry is independently set to $+1$ with probability $\xi/2$, $-1$ with probability $\xi/2$, and $0$ with probability $1-\xi$. We also include the Gaussian embedding as a benchmark, where each entry is independently drawn from the standard normal distribution. We display the condition number $\kappa(\bdPhi \qmat{U}) = \sigma_{\max}(\bdPhi \qmat{U})/\sigma_{\min}(\bdPhi \qmat{U})$ of $\bdPhi \qmat{U }$ in Figure \ref{SMfig:JL_property_SparseRademacher}.
\end{example}
\begin{figure}
    \centering
    \captionsetup{font=footnotesize}
     \begin{subfigure}[c]{0.49\textwidth}
    \centering
    \adjincludegraphics[width=\linewidth,valign=m]{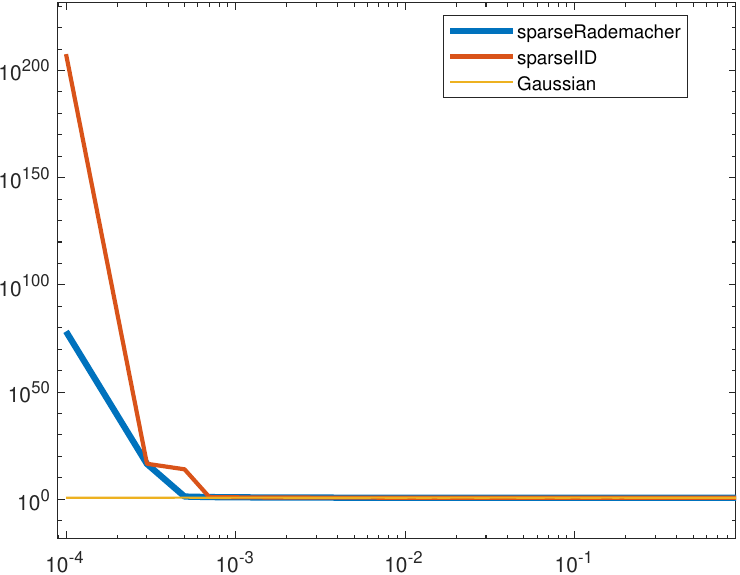} 
  \end{subfigure}\hfill
  \begin{subfigure}[c]{0.49\textwidth}
    \centering
    \adjincludegraphics[width=\linewidth,valign=m]{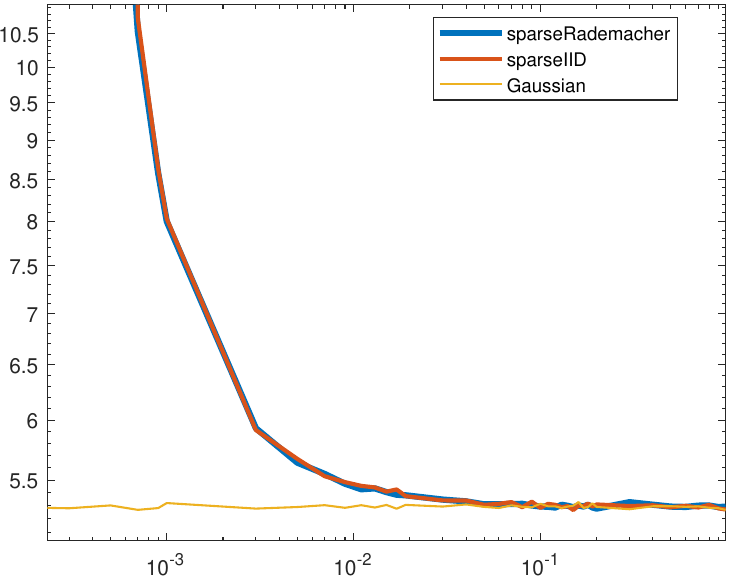}
  \end{subfigure}
    \caption{Condition number $\kappa(\bdPhi \qmat{U})$ of $\bdPhi \qmat{U}$ with respect to the sparsity level of SparseRademacher matrices. The right subfigure zooms out the left one in the interval $[10^{-3},1]$.}\label{SMfig:JL_property_SparseRademacher}
\end{figure}

From   Figure \ref{SMfig:JL_property_SparseRademacher}, we can observe three phenomena:

(1) The subspace embedding property holds for sparseRademacher   when the sparsity level is larger than $1e-3$. When the sparsity level decrease to less than $1e-3$, the condition numbers corresponding to the two test sparse embeddings increase significantly, especially for the sparseIID test matrix.

(2) When the sparsity level is larger than $1e-3$, the condition numbers corresponding to the SparseRademacher embedding are very close to that of the sparseIID embedding, while the latter is proved to satisfy the JL property \cite{troppComparisonTheoremsMinimum2025}.  

(3) When the sparsity level is larger than $1e-2$, the condition numbers corresponding to the two sparse embeddings are relatively stable and close to that of the Gaussian embedding. 

Based on the above observations, we find that when sparsity level is not too small (say, $10^{-3}$), the SparseRademacher matrix  may exhibit JL property in high-dimensional spaces in practice.

\color{black}

\section{TYUC17 in Mixed-Precision and TYUC17-SPI Variant} \label{sec:performance_SPI_variant}

\begin{algorithm}
    \caption{TYUC17-mixed-precision; see Section \ref{sec:related_approaches}}
    \label{SMalg:TYUC17-mixed-precision}
    \begin{algorithmic}[1]
        \Require Single-precision sketch  and test matrix   \(\qmat{W} = \bdPsi\qmat{A}\in\bbR^{d\times n}\), \(\bdPsi\in\bbR^{d \times m}\), double sketch ${\qmat{Y}}=\qmat{A}{\bdOmega}\in\bbR^{m\times s}$
        \Ensure  Rank-\(r\) approximation  \( \qmat{U}\Sigma\qmat{V}^{T}\) with orthonormal \(\qmat{U} \in \mathbb{R}^{m \times r}\) and \(\qmat{V} \in \mathbb{R}^{n \times r}\) and diagonal matrix $\Sigma\in\bbR^{r\times r}$
        \State \([\qmat{Q}, \sim ] = \texttt{QR}(\qmat{Q}, 0)\)
        \State \(\qmat{B} = (\bdPsi \qmat{Q})\backslash  \qmat{W}\) \Comment{Single-precision}
        \State $\qmat{B} = \texttt{double}(\qmat{B}) \in\bbR^{s\times n}$\Comment{Double-precision for avoiding loss of orthogonality in SVD. The space of $\qmat{W}$ can be used for such a conversion if $d\geq 2s$; See Algorithm \ref{alg:single2double}.}
        \State $\qmat{U}\Sigma\qmat{V}^{T}=\texttt{SVD}(\qmat{B},r)$
        \State $\qmat{U} = \qmat{Q}\qmat{U}$
        \State \textbf{return} \((\qmat{U}, \Sigma, \qmat{V})\) 
    \end{algorithmic}
\end{algorithm}

At the end of Section \ref{sec:related_approaches} of the main manuscript, we have discussed that TYUC17  can also be implemented  in mixed-precision, e.g., $\qmat{W}$ in single-precision to allow a larger sketch while $\qmat{Y}$ still in double-precision to avoid loss of orthogonality. A detailed algorithm is depicted in Algorithm \ref{SMalg:TYUC17-mixed-precision}. 

A TYUC17-SPI variant was introduced in Section \ref{sec:reduce_storage_cost_q_gt_1}. The variant generates two sketches only: $\qmat{W}=\bdPsi\qmat{A}$ and  $\qmat{Z}=\qmat{A}\bdPhi \in\bbR^{m\times l}$. Then perform $\hat{\qmat{Y}}= \qmat{Z}(\qmat{Z}^{T}\qmat{Z})^q\tilde{\bdOmega}\in\bbR^{m\times s}$ to obtain the rangefinder matrix. Here $\bdPsi\in\bbR^{d\times m},\bdPhi\in\bbR^{m\times l},\tilde{\bdOmega}\in\bbR^{l\times s}$ are independent random test matrices.  $\qmat{W}$ and $\qmat{Z}$ are in single-precision. After the computation of $\hat{\qmat{Y}}$,  $ {\qmat{Z}}$ can be deleted, and its space can be used   for converting $\hat{\qmat{Y}}$ from single to double-precision. The remaining steps are the same as TYUC17-mixed-precision.

When $l\geq 2s$, the above conversion from single-precision to double-precision is possible. Without loss of generality assume that $\hat{\qmat{Y}}$ is stored in the first $s$ columns of the previous space of $\qmat{Z}$ and is in the column-major format. The extra space for converting is also of size at least $m\cdot s$, just the remainig space of $\qmat{Z}$.  The procedure is depicted in Algorithm \ref{alg:single2double}.  Note that {\color{black}this memory reuse is feasible in programming environments that allow for explicit memory management or fine-grained control over data allocation and deallocation.}

\begin{algorithm}[H]
    \caption{Convert single‐precision matrix to double-precision}
    \label{alg:single2double}
    \begin{algorithmic}[1]
    \Require 
      \Statex $\qmat{Y}_{\mathrm{single}}$: an \(m\times s\) matrix in single precision (column‐major)  
      \Statex $T$: extra space of size \(m\cdot s+2\)  
    \Ensure 
      \Statex $\qmat{Y}_{\mathrm{double}}$ contains the double‐precision copy of \(\qmat{Y}_{\mathrm{single}}\)
    \State virtually view $\qmat{Y}_{\mathrm{single}}$ and $T$  as a \((2m s+2)\times 1\) vector $\qmat{Z}$
    \For{$k \leftarrow m\cdot s $ \textbf{downto} $1$}
      \State $\qmat{Z}(2k+1:2k+2) \gets \texttt{double}(Y_{\mathrm{single}}(k))$
    \EndFor
    \State \Return $\qmat{Z}$
    \end{algorithmic}
    \end{algorithm}

    \begin{algorithm}
        \caption{TYUC17-SPI variant; see Section \ref{sec:reduce_storage_cost_q_gt_1}}
        \label{SMalg:TYUC17-SPI}
        \begin{algorithmic}[1]
            \Require Two single-precision sketches \( {\qmat{Z}}=\qmat{A} {\bdPhi}\in\bbR^{m\times l}\), \(\qmat{W} = \bdPsi\qmat{A}\in\bbR^{d\times n}\), single test matrix \(\bdPsi\in\bbR^{d \times m}\), double test matrix \(\tilde{\bdOmega}\in\bbR^{l\times s}\)
            \Ensure  Rank-\(r\) approximation  \( \qmat{U}\Sigma\qmat{V}^{T}\) with orthonormal \(\qmat{U} \in \mathbb{R}^{m \times r}\) and \(\qmat{V} \in \mathbb{R}^{n \times r}\) and diagonal matrix $\Sigma\in\bbR^{r\times r}$
            \State \(\qmat{M}_{\text{tmp}} =  \texttt{double}(\qmat{Z}^{T}\qmat{Z})   \in\bbR^{l\times l}\) \Comment{Only take  $l^2$ extra storage for conversion if $l$ is small}.
            \State $\hat{\qmat{Y}} = \qmat{Z}\qmat{M}_{\text{tmp}}^q\tilde{\bdOmega}\in\bbR^{m\times s}$  \Comment{Optional: re-orthonormalization on $\qmat{M}_{\text{tmp}}\bdtildeOmega $ }
            \State \([\qmat{Q}, \sim] = \texttt{QR}(\texttt{double}(\tilde{\qmat{Y}}), 0)\)  \Comment{Use the space of $\qmat{Z}$ to double $\hat{\qmat{Y}}$ to avoid loss of orthogonality in $\qmat{Q}$. The remaining steps are the same as TYUC17-mixed-precision (Algorithm \ref{SMalg:TYUC17-mixed-precision})}
            \State \(\qmat{B} = (\bdPsi \qmat{Q})\backslash  \qmat{W}\)
            \State \qmat{B} = \texttt{double}(\qmat{B}) \Comment{The same as that in Algorithm \ref{SMalg:TYUC17-mixed-precision}}
            \State $\qmat{U}\Sigma\qmat{V}^{T}=\texttt{SVD}(\qmat{B},r)$
            \State $\qmat{U}= \qmat{Q}\qmat{U}$
            \State \textbf{return} \((\qmat{U}, \Sigma, \qmat{V})\) 
        \end{algorithmic}
    \end{algorithm}


 

\begin{example}
    In this example, we compare TYUC17-mixed-precision  (Algorithm \ref{SMalg:TYUC17-mixed-precision})  and TYUC17-SPI variant in Section \ref{sec:reduce_storage_cost_q_gt_1} (Algorithm \ref{SMalg:TYUC17-SPI}) on synthetic data as in Section \ref{sec:synthetic_data_def} under the same storage budget. We present the oracle error only. 
\end{example}

\begin{figure}[htp]
    \centering
    \captionsetup{font=tiny}
    \begin{subfigure}{0.3\textwidth}
        \centering
        \includegraphics[width=\linewidth]{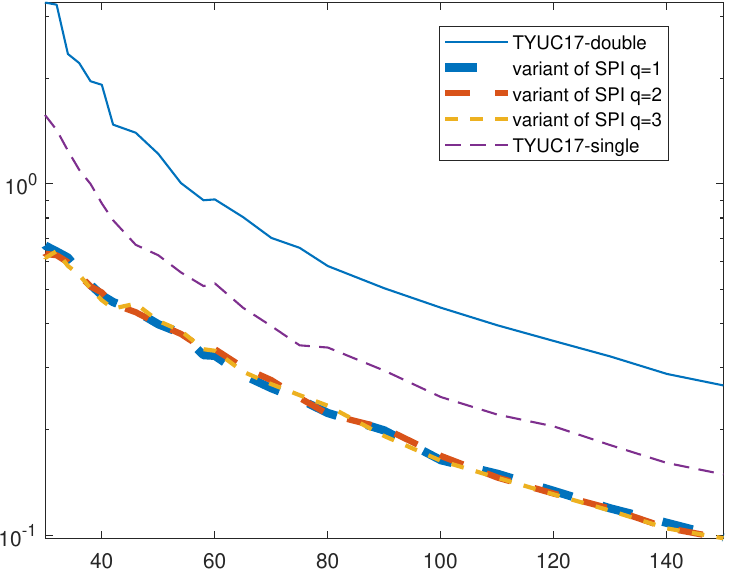}
        \caption{Low Rank with Low Noise}
    \end{subfigure}%
    \begin{subfigure}{0.3\textwidth}
        \centering
        \includegraphics[width=\linewidth]{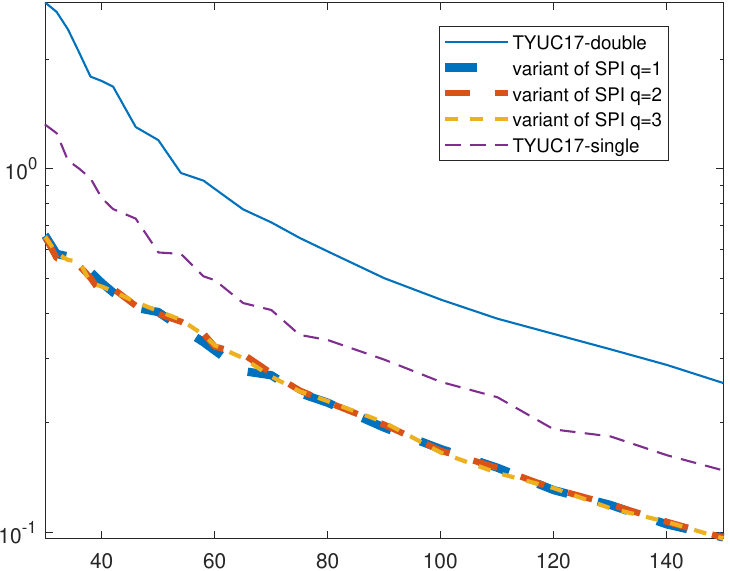}
        \caption{Low Rank with Medium Noise}
    \end{subfigure}%
    \begin{subfigure}{0.3\textwidth}
        \centering
        \includegraphics[width=\linewidth]{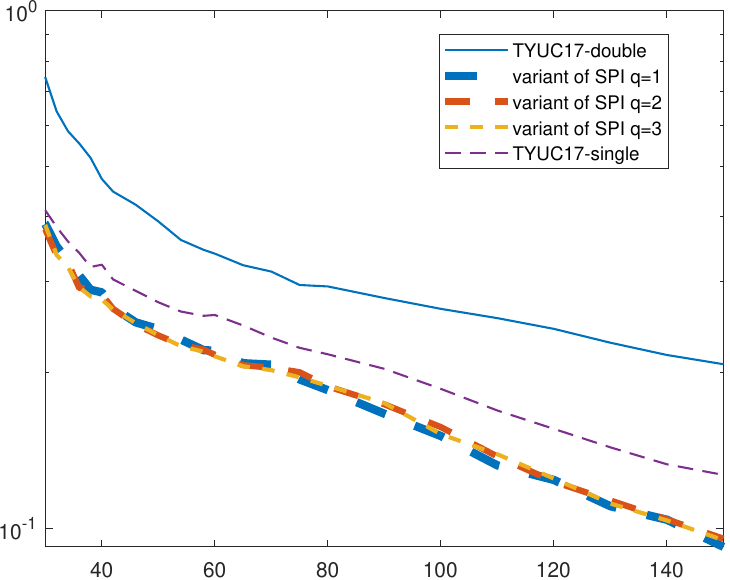}
        \caption{Low Rank with High Noise}
    \end{subfigure}

    \vspace{0.1cm} 

    \begin{subfigure}{0.3\textwidth}
        \centering
        \includegraphics[width=\linewidth]{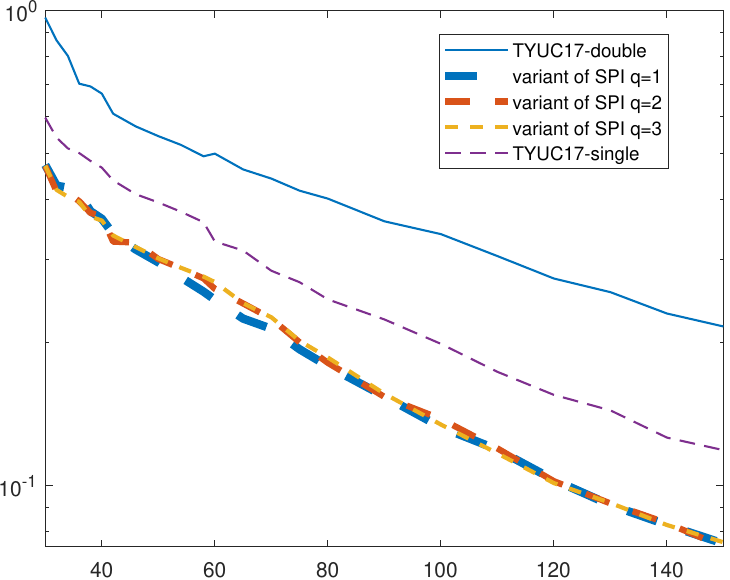}
        \caption{Slow polynomial decay}
    \end{subfigure}%
    \begin{subfigure}{0.3\textwidth}
        \centering
        \includegraphics[width=\linewidth]{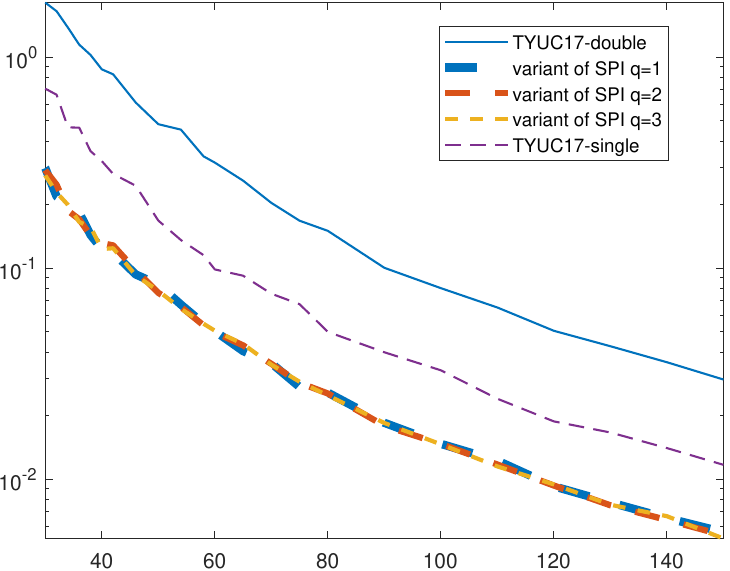}
        \caption{Medium polynomial decay}
    \end{subfigure}%
    \begin{subfigure}{0.3\textwidth}
        \centering
        \includegraphics[width=\linewidth]{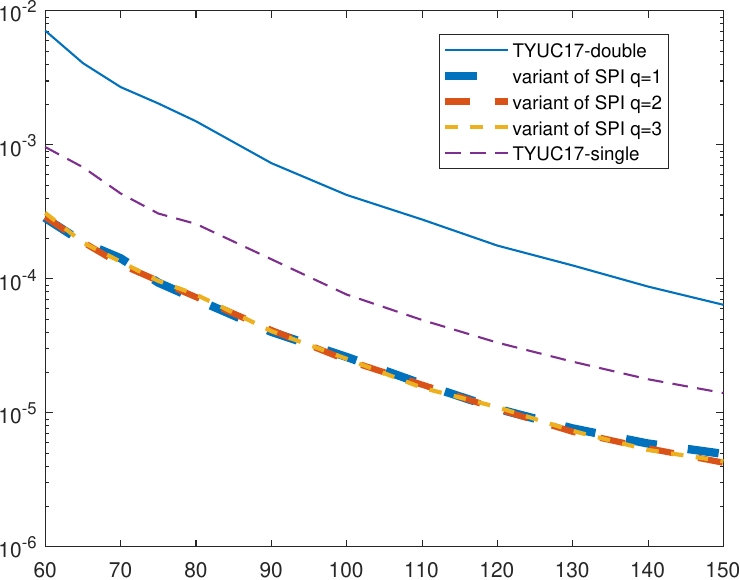}
        \caption{Fast polynomial decay}
    \end{subfigure}

    \vspace{0.1cm} 

    \begin{subfigure}{0.3\textwidth}
        \centering
        \includegraphics[width=\linewidth]{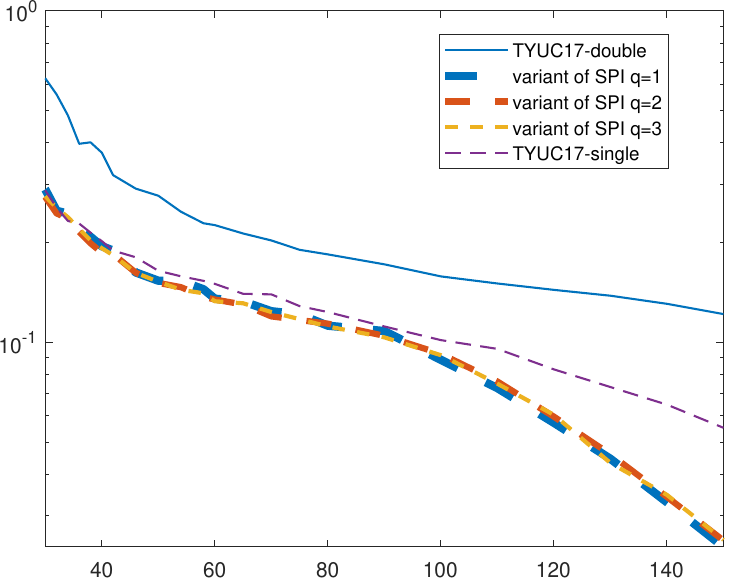}
        \caption{Slow exponentially decay}
    \end{subfigure}%
    \begin{subfigure}{0.3\textwidth}
        \centering
        \includegraphics[width=\linewidth]{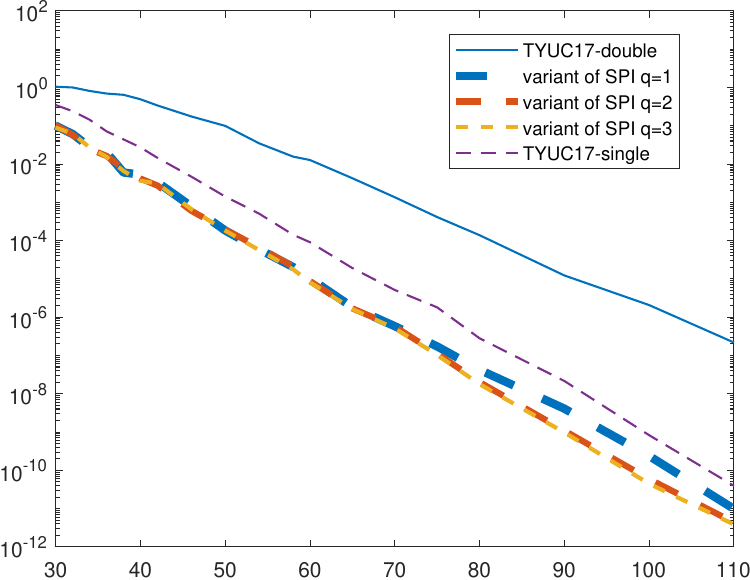}
        \caption{Medium exponentially decay}
    \end{subfigure}%
    \begin{subfigure}{0.3\textwidth}
        \centering
        \includegraphics[width=\linewidth]{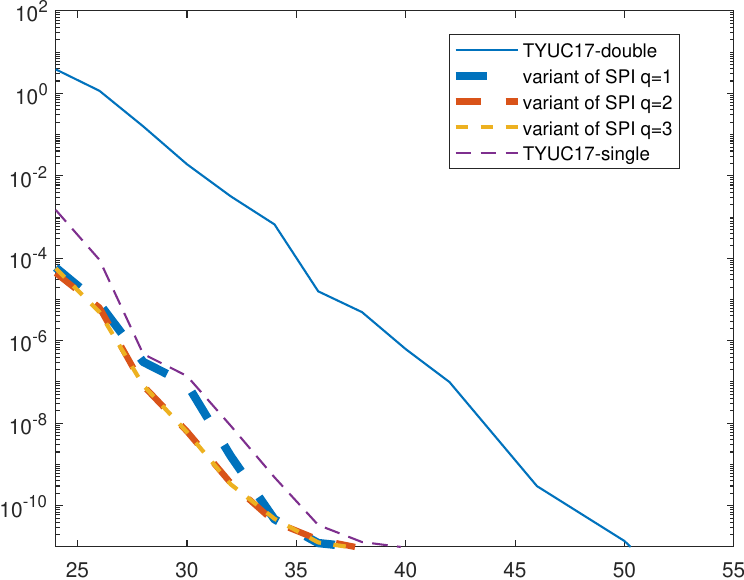}
        \caption{Fast exponentially decay}
    \end{subfigure}
    \caption{Figures of Synectic data. X-axis means the $T$ represented the storage budget. Y-axis means the relative Frobenius error $S_F$. All results are averaged by 20 independent repeated experiments.}
    \label{SMfig:Wsingle_Synectic_data}
\end{figure}

Figure \ref{SMfig:Wsingle_Synectic_data} shows the performance of the prototype TYUC17 (denoted as TYUC17-double), TYUC17-mixed-precision (denoted as TYUC17-single), and TYUC17-SPI variant (denoted as SPI variant). Here TYUC17-double acts as a baseline.  From the figure, we first observe that TYUC17-single outperforms TYUC17-double, in that TYUC17-single allows a larger  $\qmat{W}$, although it is in single-precision.  Next, we can also see that TYUC17-SPI variant still outperforms TYUC17-single, thanks to the  larger single-precision sketch $\qmat{Z}$ to retain more range information of the data matrix.   The  improvement is more significant when the data matrix has a slow to moderate spectrum decay.   The   exception is that  when the spectrum exhibits a fast exponential decay and a larger storage budget is available (in the bottom-right subfigure), both of TYUC17-single and TYUC17-SPI variant are outperformed by TYUC17-double. This   is due to that the former two reach  the precision limit of their mixed-precision strategy, where the round-off errors become  dominant.

\section{Probabilistic Spectral Norm Bound for TYUC17}

Similar to Lemma 5.8 of the main manuscript, when $d\geq s+4$, we have that conditioned on $E=\left\{\bdOmega |\normSpectral{\qmat{A}-P_{\qmat{A}\bdOmega}\qmat{A}}<e_1,\normF{\qmat
    A-P_{\qmat{A}\bdOmega}\qmat{A}}<e_2\right\}$, with failure probability at most $\!e^{-u^2\!/2}+2t^{-(d-s)}$, there holds 
\begin{align*}
\normSpectral{\qmat{A}-\qmat{Q}\qmat{B}}\leq \frac{e\sqrt{d}t}{d-s+1}e_2+\left(1+t\sqrt{\frac{3s}{d-s+1}} +u \frac{e\sqrt{d}t}{d-s+1}\right)e_1.
    \end{align*}
For any $\varrho\leq s-4$ and $t\geq 1, u>0$,  
with failure probability at most  \( 2t^{-(s-\varrho)} + e^{-u^2 / 2} \)  there holds
    \begin{align*}
      &  \| (I - P_{\qmat{A}\bdOmega}) \qmat{A} \|_F  \leq 
        \left( 1 + t \cdot \sqrt{\frac{3\varrho}{s-\varrho+1}} \right) 
        \tau_{\varrho+1}(\qmat{A})
        + u t \cdot \frac{e \sqrt{s}}{s-\varrho+1} \cdot \sigma_{\varrho+1}\left(\qmat{A}\right);\\
     &   \| (I \!-\! P_{\qmat{A}\bdOmega})\qmat{A} \|  \!\leq\! \left(\!\! 1 + t \cdot \sqrt{\frac{3\varrho}{s-\varrho + 1}}\!+\!ut \cdot \frac{e \sqrt{s}}{s-\varrho+1}  \right)\! \sigma_{\varrho+1}\left(\qmat{A}\right) \!+\! t \!\cdot\! \frac{e \sqrt{s}}{s-\varrho+1}\tau_{\varrho+1}\left(\qmat{A}\right).
    \end{align*}
    Thus with failure probability at most  $e^{-u^2/2}+2t^{-(d-s)} + 4t^{-(s-\varrho)} + 2e^{-u^2 / 2}$, $\bignorm{\qmat{A}-\qmat{Q}\qmat{B}}$ with $\qmat{Q}$ and $\qmat{B}$ generated by TYUC17  is upper bounded by
    \begin{align*}
       & \frac{e\sqrt{d}t}{d-s+1} \cdot \bigzhongkuohao{\left( 1 + t \cdot \sqrt{\frac{3\varrho}{s-\varrho+1}} \right) 
        \tau_{\varrho+1}(\qmat{A})
        + u t \cdot \frac{e \sqrt{s}}{s-\varrho+1} \cdot \sigma_{\varrho+1}\left(\qmat{A}\right)  }\\
      &  ~~ +\left(1+t\sqrt{\frac{3s}{d-s+1}} +u \frac{e\sqrt{d}t}{d-s+1}\right)  \left[  \left(\! 1 + t   \sqrt{\frac{3\varrho}{s-\varrho + 1}}+ut  \frac{e \sqrt{s}}{s-\varrho+1}  \right) \sigma_{\varrho+1}\left(\qmat{A}\right) \right.\\
      &~~~~~~~~~~~~~~~~~~~~~~~~~~~~~~~~~~~~~~~~~~~~~~~~~~~~~~~~ \left. + t \cdot \frac{e \sqrt{s}}{s-\varrho+1}\tau_{\varrho+1}\left(\qmat{A}\right) \right].
    \end{align*}

\section{Proof of Oblique Projection Error Decomposition}\label{SMsec:proof_oblique_projection_error_decomposition}

We provide a   proof for the oblique projection error decomposition stated in Sect. \ref{sect:metrics} of the main manuscript:
\begin{align*}
     \normP{\qmat{A}-\qmat{U}\Sigma\qmat{V}^{T}}=\|(\qmat{A}-\qmat{U}\qmat{U}^{T}\qmat{A}) + (\qmat{U}\tilde{\qmat{U}}^T(\qmat{Q}^T\qmat{A}-(\bdPsi\qmat{Q})^{\dagger}\bdPsi\qmat{A}))\|_p.
\end{align*}  

   
\begin{proposition}
    For any column orthogonal matrix \(\qmat{Q} \in \mathbb{R}^{m \times s}\) and matrix $\qmat{A}\in \mathbb{R}^{m \times n}$, let $\tilde{\qmat{U}},\Sigma,\qmat{V}$ be the rank-r truncated SVD of $\qmat{B}=(\bdPsi\qmat{Q})^\dagger \bdPsi\qmat{A}$, $\qmat{U}=\qmat{Q}\tilde{\qmat{U}}$. Then the following holds:
     \begin{align*}
        \bignorm{\qmat{A}-\qmat{U}\Sigma\qmat{V}^{T}}=\bignorm{(\qmat{A} - \qmat{U}\qmat{U}^T\qmat{A})+\qmat{U}\tilde{\qmat{U}}^T(\qmat{Q}^T\qmat{A}-(\bdPsi\qmat{Q})^\dagger \bdPsi\qmat{A})}.
    \end{align*}
\end{proposition}
\begin{proof} By the definitions of $\tilde{\qmat{U}},\Sigma,\qmat{V}$, and $\qmat{B}$,
    \[
    \tilde{\qmat{U}}^T(\bdPsi\qmat{Q})^\dagger \bdPsi\qmat{A} = \tilde{\qmat{U}}^T\qmat{B}=\Sigma\qmat{V}^T.
    \]
    and we   obtain that
    \begin{align*}
        \bignorm{\qmat{A}-\qmat{U}\Sigma\qmat{V}^{T}} &= \bignorm{\qmat{A} - \qmat{U}\qmat{U}^T\qmat{A}+\qmat{U}\qmat{U}^T\qmat{A}-\qmat{U}\tilde{\qmat{U}}^T(\bdPsi\qmat{Q})^\dagger \bdPsi\qmat{A}}\\
        &=\bignorm{(\qmat{A} - \qmat{U}\qmat{U}^T\qmat{A})+\qmat{U}\tilde{\qmat{U}}^T(\qmat{Q}^T\qmat{A}-(\bdPsi\qmat{Q})^\dagger \bdPsi\qmat{A})}.
    \end{align*}
\end{proof}

\section{Application of   SPI to the Streaming Algorithm  TYUC19} \label{sec:TYUC19_SPI_main}
In this section, we demonstrate that SPI can also be applied to other randomized low-rank approximation algorithms. Even without hardware specifications, we can still apply SPI to algorithms under memory-limited scenarios. For example, the streaming algorithm TYUC19 \cite{troppStreamingLowRankMatrix2019}, proposed by Tropp et al. We first recall the algorithm in Section \ref{sec:tyuc19}, then incorporate SPI into TYUC19 in Section \ref{sec:TYUC19_SPI}, and perform numerical comparisons. In Section \ref{sec:deviation_bound_TYUC19_SPI}, we present the proof idea for establishing the deviation bound for TYUC19-SPI.

\subsection{TYUC19 algorithm}\label{sec:tyuc19}

We briefly recall the TYUC19 algorithm here. It uses two sketches: \begin{align*} 
    \qmat{Y}=\qmat{A}\bdOmega\in\bbR^{m\times s}, \quad \qmat{X}=\bdGamma\qmat{A}\in\bbR^{s\times n} 
\end{align*} 
to capture the projection space, and then applies two-sided sketching: 
\begin{align*} 
    \qmat{K}=\bdPhi \qmat{A}\bdPsi^{T}\in\bbR^{d\times d}, 
\end{align*} 
to contains the fresh information, where $\bdOmega\in\bbR^{n\times s}, \bdGamma\in\bbR^{s\times m}, \bdPhi\in\bbR^{d\times m}$, and $\bdPsi\in\bbR^{d\times n}$ are independent random test matrices.

The range and co-range of $\qmat{A}$ are constructed by orthonormalizing the sketches $\qmat{Y}$ and $\qmat{X}$: 
\begin{align*}
     \qmat{Q} = \texttt{QR}(\qmat{Y}) \in\bbR^{m\times s}, \quad \qmat{P} = \texttt{QR}(\qmat{X}^{T}) \in \bbR^{n\times s}. 
\end{align*}
Then, an approximation of the ``core matrix'' is obtained by solving least squares problems: \begin{align*} \qmat{C} = ((\bdPhi\qmat{Q})\backslash \qmat{K}) / (\bdPsi\qmat{P})^{T} \in \bbR^{s\times s}, \end{align*} which returns the approximate triple $(\qmat{Q}, \qmat{C}, \qmat{P})$.

Further steps, such as truncated SVD, can be applied to $\qmat{C}$ to obtain the final result, similar to TYUC17. The algorithm is   shown in Algorithm \ref{SMalg:initial_approximation}.


\color{black}
\begin{algorithm}
    \caption{TYUC19 \cite{troppStreamingLowRankMatrix2019}}
    \label{SMalg:initial_approximation}
    \begin{algorithmic}[1]
        \Require Sketches \(\qmat{Y}\in\bbR^{m\times s}\), \(\qmat{X}\in\bbR^{s\times n}\), \(\qmat{K}\in\bbR^{d\times d}\), \(\bdPhi\in\bbR^{d\times m}\),\(\bdPsi\in\bbR^{n\times d}\).
    \Ensure Rank-\(r\) approximation   in the form \(  \qmat{U}\Sigma\qmat{V}^{T}\) with orthonormal \(\qmat{U} \in \mathbb{R}^{m \times r}\) and \(\qmat{V} \in \mathbb{R}^{n \times r}\) and diagonal \(\Sigma \in \mathbb{R}^{r \times r}\)
        \State \([\qmat{Q}, \sim] = \texttt{QR}(\qmat{Y}, 0)\) 
        \State \([\qmat{P}, \sim] = \texttt{QR}(\qmat{X}^{T}, 0)\)
        \State \(\qmat{C} = ((\bdPhi \qmat{Q})^{T} \backslash \qmat{K}) / ((\bdPsi \qmat{P})^{T})\)
        \State $[\qmat{U},\Sigma,\qmat{V}]=\texttt{SVD}(\qmat{C},r)$
        \State $\qmat{U} = \qmat{Q}\qmat{U}$
        \State $\qmat{V} = \qmat{P}\qmat{V}$
        \State \textbf{return} \((\qmat{U}, \Sigma, \qmat{V})\) 
    \end{algorithmic}
\end{algorithm}


\subsection{TYUC19-SPI}\label{sec:TYUC19_SPI}
In the multi-level caches and distributed systems, SPI can be incorporated into TYUC19, as that for TYUC17 discussed in Section \ref{sec:same_storage_as_PSA}. The details is similar and we omit them. Here we mainly focus on the mixed-precision strategy, and 
we will integrate SPI into TYUC19 in a similar way as the TYUC17-SPI variant in Section \ref{sec:reduce_storage_cost_q_gt_1} (Algorithm \ref{SMalg:TYUC17-SPI}). To be specific, we sketch two   sketches
$$\qmat{Z}=\qmat{A}\bdOmega\in\bbR^{m\times l} ~{\rm and}~ \qmat{W} = \bdGamma\qmat{A}\in\bbR^{l\times n}$$ 
 in single-precision, and sketch   
 \begin{align*} 
    \qmat{K}=\bdPhi \qmat{A}\bdPsi^{T}\in\bbR^{d\times d}, 
\end{align*} 
in double-precision as that in TYUC19. Then we generate two small random test matrices $\tilde{\bdOmega}\in\bbR^{l\times s}$ and $\tilde{\bdGamma}\in\bbR^{s\times l}$, and apply the SPI to get $\hat{\qmat{Y}}\in\bbR^{m\times s}$ and $\hat{\qmat{X}}\in\bbR^{s\times n}$ as follows:
\begin{align*}
    \hat{\qmat{Y}} = \qmat{Z}(\qmat{Z}^{T}\qmat{Z})^q\tilde{\bdOmega}\in\bbR^{m\times s}, \quad \hat{\qmat{X}} = \tilde{\bdGamma} (\qmat{W} \qmat{W}^{T})^q\qmat{W}\in\bbR^{s\times n}.
\end{align*}
As TYUC17-SPI variant, $\hat{\qmat{Y}}$ can be stored in the first $s$ columns of   $\qmat{Z}$, and $\hat{\qmat{X}}$ can be stored in the first $s$ rows of the previous space of $\qmat{W}$. When $l\geq 2s$, we can use the space of $\qmat{Z}$ and $\qmat{W}$ to convert $\hat{\qmat{Y}}$ and $\hat{\qmat{X}}$ from single-precision to double-precision, using Algorithm \ref{alg:single2double}. We then orthogonal $\hat{\qmat{Y}}$ and $\hat{\qmat{X}}^{T}$ to obtain $\qmat{Q}$ and $\qmat{P}$.  The remaining steps are the same as TYUC19. We depict the algorithm in Algorithm \ref{SMalg:TYUC19-SPI}.

\begin{algorithm}
    \caption{TYUC19-SPI}
    \label{SMalg:TYUC19-SPI}
    \begin{algorithmic}[1]
        \Require Two single-precision sketches \( {\qmat{Z}}=\qmat{A} {\bdOmega}\in\bbR^{m\times l}\), \( {\qmat{W}} = ( {\bdGamma}\qmat{A})^{T}\in\bbR^{l\times n}\), double-precision sketch \(\qmat{K}=  {\bdPhi}\qmat{A} {\qmat{\bdPsi}}\in\bbR^{d\times d}\), test matrices \(\bdPhi\in\bbR^{d\times m}\),\(\bdPsi\in\bbR^{n\times d}\), $\tilde{\bdOmega}\in\bbR^{l\times s},\tilde{\bdGamma}\in\bbR^{s\times l}$.
        \Ensure  Rank-\(r\) approximation   in the form \(  \qmat{U}\Sigma\qmat{V}^{T}\) with orthonormal \(\qmat{U} \in \mathbb{R}^{m \times r}\) and \(\qmat{V} \in \mathbb{R}^{n \times r}\) and diagonal \(\Sigma \in \mathbb{R}^{r \times r}\)
        \State \(\qmat{M}_{\text{tmp}} =  \texttt{double}(\qmat{Z}^{T}\qmat{Z})   \in\bbR^{l\times l}\) \Comment{Only take  $l^2$ extra storage for conversion if $l$ is small}.
        \State $\hat{\qmat{Y}} = \qmat{Z}\qmat{M}_{\text{tmp}}^q\tilde{\bdOmega}\in\bbR^{m\times s}$  \Comment{Optional: re-orthonormalization on $\qmat{M}_{\text{tmp}}\bdtildeOmega $ }
        \State \([\qmat{Q}, \sim] = \texttt{QR}(\texttt{double}(\hat{\qmat{Y}}), 0)\)  \Comment{Double $\hat{\qmat{Y}}$ to avoid loss of orthogonality in $\qmat{Q}$}
        \State \(\qmat{N}_{\text{tmp}} =  \texttt{double}(\qmat{W}\qmat{W}^{T})   \in\bbR^{l\times l}\)  
        \State $\hat{\qmat{X}} = \tilde{\bdGamma} \qmat{N}_{\text{tmp}}^q\qmat{W}\in\bbR^{s\times n}$  \Comment{Optional: re-orthonormalization on $\qmat{N}_{\text{tmp}}\tilde{\bdGamma} $ }
        \State \([\qmat{P}, \sim] = \texttt{QR}(\texttt{double}(\hat{\qmat{X}}^{T}), 0)\)  \Comment{Double $\hat{\qmat{X}}^{T}$ to avoid loss of orthogonality in $\qmat{P}$} 
        \State \(\qmat{C} = ((\bdPhi \qmat{Q})^{T} \backslash \qmat{K}) / ((\bdPsi \qmat{P})^{T})\) \Comment{The remaining steps are the same as TYUC19 (Algorithm \ref{SMalg:initial_approximation})}
        \State $[\qmat{U},\Sigma,\qmat{V}]=\texttt{SVD}(\qmat{C},r)$
        \State $\qmat{U} = \qmat{Q}\qmat{U}$
        \State $\qmat{V} = \qmat{P}\qmat{V}$
        \State \textbf{return} \((\qmat{U}, \Sigma, \qmat{V})\)  
    \end{algorithmic}
\end{algorithm}

\begin{figure}[htp]
    \centering
    \captionsetup{font=tiny}
    \begin{subfigure}{0.3\textwidth}
        \centering
        \includegraphics[width=\linewidth]{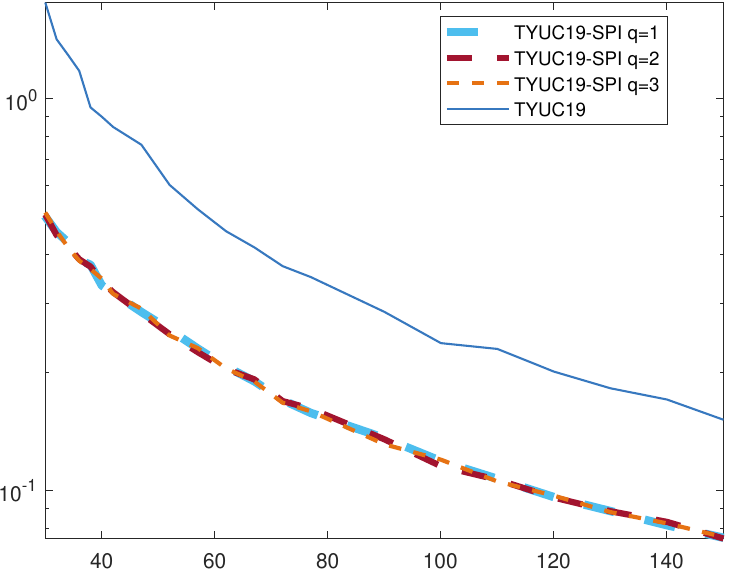}
        \caption{Low Rank with Low Noise}
    \end{subfigure}%
    \begin{subfigure}{0.3\textwidth}
        \centering
        \includegraphics[width=\linewidth]{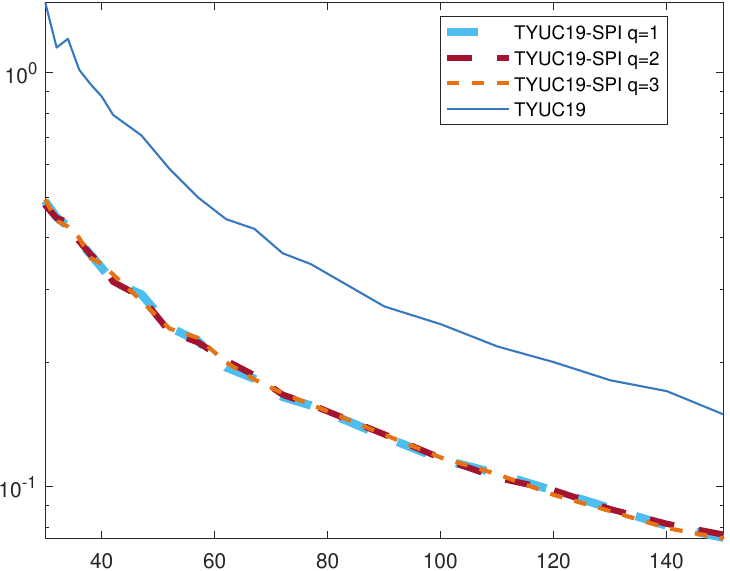}
        \caption{Low Rank with Medium Noise}
    \end{subfigure}%
    \begin{subfigure}{0.3\textwidth}
        \centering
        \includegraphics[width=\linewidth]{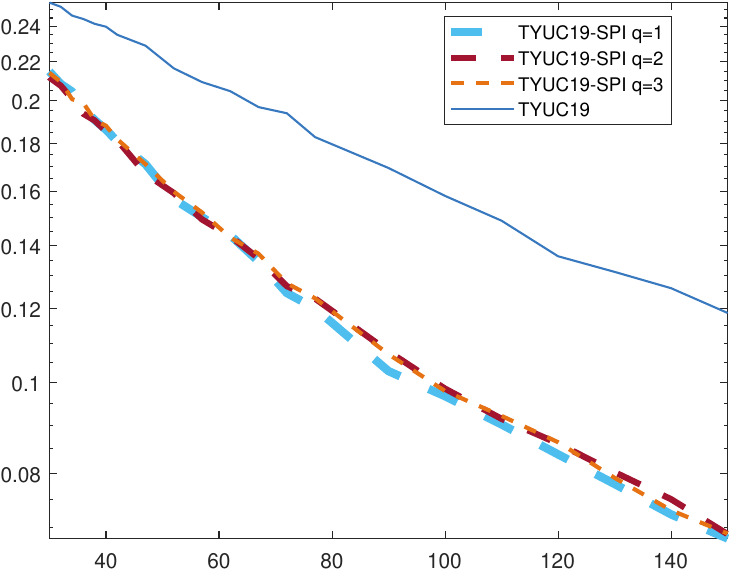}
        \caption{Low Rank with High Noise}
    \end{subfigure}

    \vspace{0.1cm} 

    \begin{subfigure}{0.3\textwidth}
        \centering
        \includegraphics[width=\linewidth]{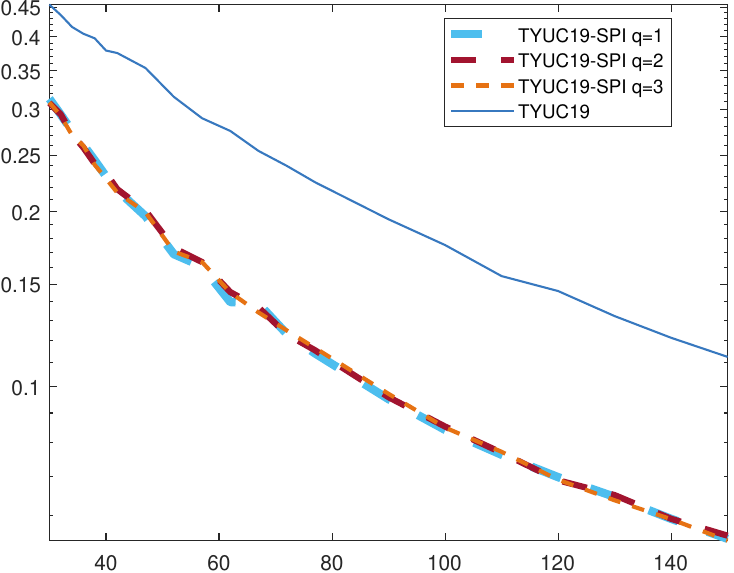}
        \caption{Slow polynomial decay}
    \end{subfigure}%
    \begin{subfigure}{0.3\textwidth}
        \centering
        \includegraphics[width=\linewidth]{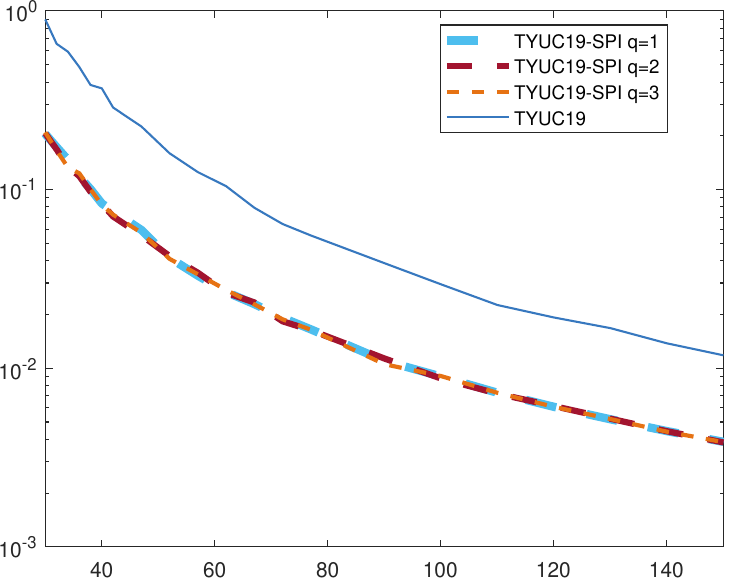}
        \caption{Medium polynomial decay}
    \end{subfigure}%
    \begin{subfigure}{0.3\textwidth}
        \centering
        \includegraphics[width=\linewidth]{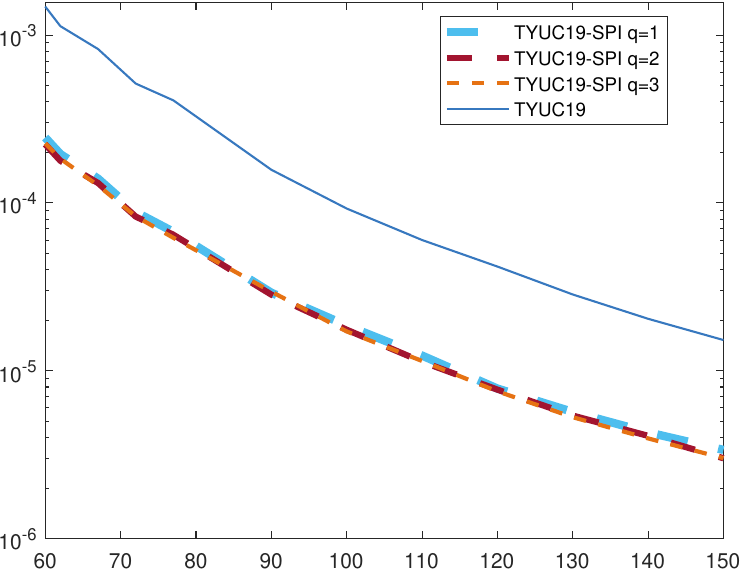}
        \caption{Fast polynomial decay}
    \end{subfigure}

    \vspace{0.1cm} 

    \begin{subfigure}{0.3\textwidth}
        \centering
        \includegraphics[width=\linewidth]{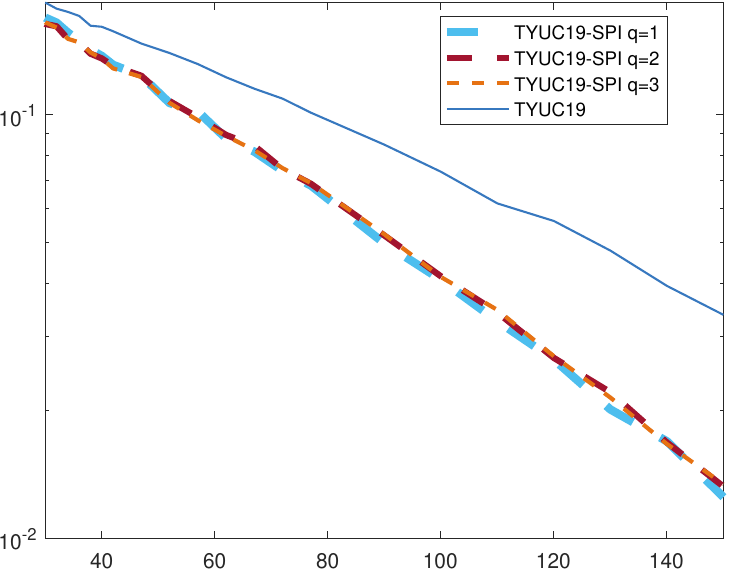}
        \caption{Slow exponentially decay}
    \end{subfigure}%
    \begin{subfigure}{0.3\textwidth}
        \centering
        \includegraphics[width=\linewidth]{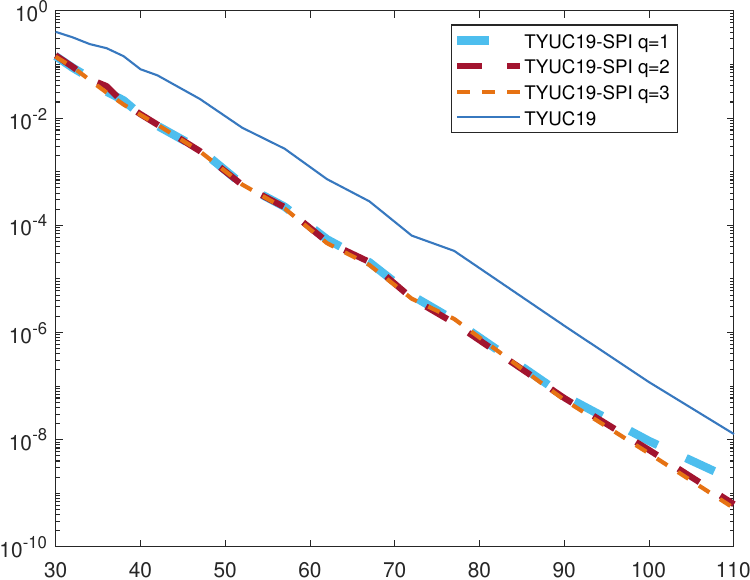}
        \caption{Medium exponentially decay}
    \end{subfigure}%
    \begin{subfigure}{0.3\textwidth}
        \centering
        \includegraphics[width=\linewidth]{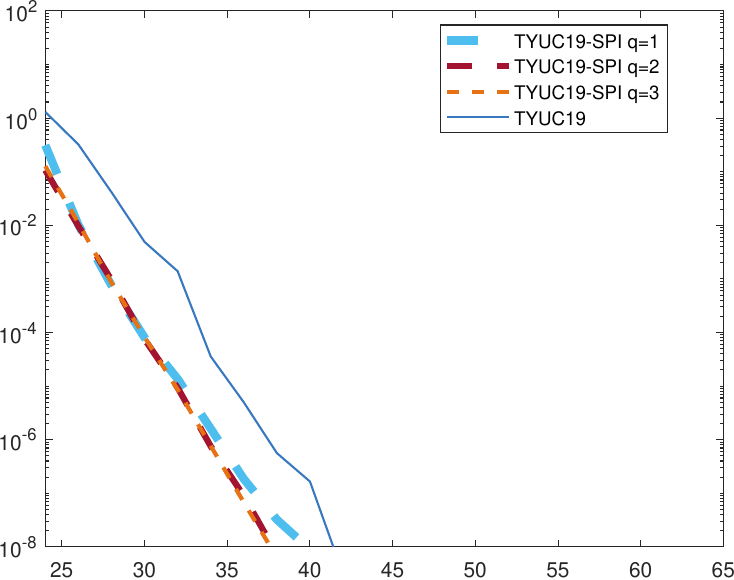}
        \caption{Fast exponentially decay}
    \end{subfigure}

    \caption{Figures of Synectic data. X-axis means the $T$ represented the storage budget. Y-axis means the relative Frobenius error $S_F$. All results are averaged by 20 independent repeated experiments.}\label{SMfig:StreamingSPI_Synectic_data}
\end{figure}

\begin{example}
    In this example, we compare TYUC19  (Algorithm \ref{SMalg:initial_approximation})  and TYUC19-SPI      (Algorithm \ref{SMalg:TYUC19-SPI}) on synthetic data as in Section \ref{sec:synthetic_data_def} under the same storage budget. We present the oracle error only. 
\end{example}

Figure \ref{SMfig:StreamingSPI_Synectic_data} shows the   performance of TYUC19-SPI and TYUC19 under the same storage budget. We can observe that the trend is similar to that of TYUC17-SPI. In summary, SPI can be not only integrated into TYUC17, but also into TYUC19, and possibly other one-pass algorithms, to improve their accuracy, especially when the decay of the data matrix is flat or moderate.    

\subsection{Analysis of deviation bound for   TYUC19-SPI}\label{sec:deviation_bound_TYUC19_SPI}
{We can also obtain a deviation bound for the TYUC19-SPI. Here for notational simplicity, we do not present the final bound; instead we provide the proof idea on how to obtain the final bound.

Let $\hat{\qmat{A}} = \qmat{Q}\qmat{C}\qmat{P}^{T}$, where $\qmat{Q},\qmat{C},\qmat{P}$ are generated in Steps 3, 6, and 7 of Algorithm \ref{SMalg:TYUC19-SPI}.    We first follow \cite[Section SM1.6; Lemma SM1.5]{troppStreamingLowRankMatrix2019} to make a decompose:}
\begin{align*}
    \normFSquare{\qmat{A} -\hat{\qmat{A}}} &= \normFSquare{\qmat{A} - \qmat{Q}\qmat{C}\qmat{P}^{T}} \\
     &= \normFSquare{\qmat{A} - \qmat{Q}\qmat{Q}^{T}\qmat{A}\qmat{P}\qmat{P}^{T} + \qmat{Q}\bigxiaokuohao{\qmat{Q}^{T}\qmat{A}\qmat{P}-\qmat{C}   }\qmat{P}^{T}  }\\
     & = \normFSquare{\qmat{\qmat{A} -\qmat{Q}\qmat{Q}^{T}\qmat{A}\qmat{P}\qmat{P}^{T} }} + \normFSquare{ {\qmat{Q}^{T}\qmat{A}\qmat{P}-\qmat{C}   }   },
\end{align*}
where the equality only uses the orthonormality of $\qmat{Q}$.  Further denote
\begin{align*}
    \qmat{P}_{\bot}\qmat{P}_{\bot}^{T} = \qmat{I}-\qmat{P}\qmat{P}^{T}, \quad \qmat{Q}_{\bot}\qmat{Q}_{\bot}^{T} = \qmat{I}-\qmat{Q}\qmat{Q}^{T}, 
\end{align*}
where $\qmat{P}_{\bot}\in\bbR^{n\times (n-s)}$, $\qmat{Q}_{\bot}\in\bbR^{m\times (m-s)}$.   
Note that
\begin{align*}
    \normFSquare{\qmat{\qmat{A} -\qmat{Q}\qmat{Q}^{T}\qmat{A}\qmat{P}\qmat{P}^{T} }} &= \normFSquare{\qmat{\qmat{A} -\qmat{Q}\qmat{Q}^{T}\qmat{A}\qmat{P}\qmat{P}^{T} } +\qmat{Q}\qmat{Q}^{T}\qmat{A}\qmat{P}_{\bot}\qmat{P}_{\bot}^{T}  - \qmat{Q}\qmat{Q}^{T}\qmat{A}\qmat{P}_{\bot}\qmat{P}_{\bot}^{T}}\\
    &=\normFSquare{ \qmat{\qmat{A} -\qmat{Q}\qmat{Q}^{T}\qmat{A} ^{T} } - \qmat{Q}\qmat{Q}^{T}\qmat{A}\qmat{P}_{\bot}\qmat{P}_{\bot}^{T}}\\
    &= \normFSquare{\bigxiaokuohao{\qmat{I}-\qmat{Q}\qmat{Q}^{T}}\qmat{A} } + \normFSquare{\qmat{A}\bigxiaokuohao{\qmat{I}-\qmat{P}\qmat{P}^{T}}}.
\end{align*}
Thus
\begin{align}\label{eq:SM_decompose_A_Ahat}
    &\normFSquare{\qmat{A} -\hat{\qmat{A}}}  
     = \normFSquare{\bigxiaokuohao{\qmat{I}-\qmat{Q}\qmat{Q}^{T}}\qmat{A} } + \normFSquare{\qmat{A}\bigxiaokuohao{\qmat{I}-\qmat{P}\qmat{P}^{T}}} +   \normFSquare{ {\qmat{Q}^{T}\qmat{A}\qmat{P}-\qmat{C}   }   }.
\end{align}

The deviation bound for the first two terms have been obtained in Theorem \ref{thm:prob_projection_error_q_gt_1} of the main manuscript. The expectation bound of $\normFSquare{ {\qmat{Q}^{T}\qmat{A}\qmat{P}-\qmat{C}   }   }$ has been established in \cite[Lemma SM1.4]{troppStreamingLowRankMatrix2019}.  Here we need a deviation bound. \cite[Lemma SM1.3]{troppStreamingLowRankMatrix2019} shows the decomposition:
\begin{align*}
  \qmat{C} - \qmat{Q}^{T}\qmat{A}\qmat{P}  &=  \bdPhi_1^\dagger\bdPhi_2(\qmat{Q}_\perp^{T}\qmat{A}\qmat{P}) +  (\qmat{Q}^{T}\qmat{A}\qmat{P}_\perp)\bdPsi_2^{T}( \bdPsi_1^\dagger)^{T}  +  \bdPhi_1^\dagger\bdPhi_2(\qmat{Q}_\perp^{T}\qmat{A}\qmat{P}_\perp)\bdPsi_2^{T}( \bdPsi_1^\dagger)^{T} .
\end{align*}
Here
\begin{align*}
    &\bdPhi_1 = \bdPhi\qmat{Q}\in\bbR^{d\times s}, \quad \bdPhi_2 = \bdPhi\qmat{Q}_{\bot} \in\bbR^{d\times (m-s)}, \\
    &\bdPsi_1 = \bdPsi\qmat{P}\in\bbR^{d\times s}, \quad \bdPsi_2 = \bdPsi\qmat{P}_{\bot}\in\bbR^{d\times (n-s)}.
\end{align*}
Then we simply decompose $ \normFSquare{\qmat{C} - \qmat{Q}^{T}\qmat{A}\qmat{P}} $ as
\begin{align*}
    \normFSquare{\qmat{C} - \qmat{Q}^{T}\qmat{A}\qmat{P}}  &\leq  3 \normFSquare{\bdPhi_1^\dagger\bdPhi_2(\qmat{Q}_\perp^{T}\qmat{A}\qmat{P})} +  3\normFSquare{(\qmat{Q}^{T}\qmat{A}\qmat{P}_\perp)\bdPsi_2^{T}( \bdPsi_1^\dagger)^{T}}\\
    &~
      +  3\normFSquare{\bdPhi_1^\dagger\bdPhi_2(\qmat{Q}_\perp^{T}\qmat{A}\qmat{P}_\perp)\bdPsi_2^{T}( \bdPsi_1^\dagger)^{T}} .
\end{align*}
Similar to the arguments in Lemma \ref{lem:GNError_deviation_bound} of the main manuscript we can obtain that
\begin{gather*}
    \probleftright{ \normFSquare{\bdPhi_1^\dagger\bdPhi_2(\qmat{Q}_\perp^{T}\qmat{A}\qmat{P})} > \alpha \normFSquare{\qmat{Q}_\perp^{T}\qmat{A}\qmat{P}}t^2(1+\mu) } \leq   e^{-2\mu^2/2} + t^{-(d-s)};\\
    \probleftright{\normFSquare{(\qmat{Q}^{T}\qmat{A}\qmat{P}_\perp)\bdPsi_2^{T}( \bdPsi_1^\dagger)^{T}} > \alpha \normFSquare{\qmat{Q}^{T}\qmat{A}\qmat{P}_\perp}t^2(1+\mu) } \leq   e^{-2\mu^2/2} + t^{-(d-s)};\\
    \probleftright{\normFSquare{\bdPhi_1^\dagger\bdPhi_2(\qmat{Q}_\perp^{T}\qmat{A}\qmat{P}_\perp)\bdPsi_2^{T}( \bdPsi_1^\dagger)^{T}} \!>\! \alpha^2 \normFSquare{\qmat{Q}_\perp^{T}\qmat{A}\qmat{P}_\perp}t^2(1+\mu) } \!\leq\!   2e^{-2\mu^2/2} \!+\! 2t^{-(d-s)},
\end{gather*}
where $\alpha = \frac{3s}{d-s+1}$. Combining the above yields
\begin{align*}
     \normFSquare{\qmat{C} - \qmat{Q}^{T}\qmat{A}\qmat{P}} &\leq 3 \alpha t^2(1+\mu) \bigxiaokuohao{\normFSquare{\qmat{Q}_\perp^{T}\qmat{A}\qmat{P}}+ \normFSquare{\qmat{Q}^{T}\qmat{A}\qmat{P}_\perp}  }   \\
     &~~+ 3 \alpha^2t^2(1+\mu)   \normFSquare{\qmat{Q}_\perp^{T}\qmat{A}\qmat{P}_\perp} 
\end{align*}
with failure probability $4e^{-2\mu^2/2} + 2t^{-(d-s)}$.

Note that $\normFSquare{\qmat{Q}_\perp^{T}\qmat{A}\qmat{P}} + \normFSquare{\qmat{Q}^{T}\qmat{A}\qmat{P}_\perp} + \normFSquare{\qmat{Q}_\perp^{T}\qmat{A}\qmat{P}_{\bot}} = \normFSquare{\qmat{A} -\qmat{Q}\qmat{Q}^{T}\qmat{A}\qmat{P}\qmat{P}^{T}}   $. Thus the above becomes
\begin{align*}
    \normFSquare{\qmat{C} - \qmat{Q}^{T}\qmat{A}\qmat{P}} &\leq 3 \alpha t^2(1+\mu) \normFSquare{\qmat{A} -\qmat{Q}\qmat{Q}^{T}\qmat{A}\qmat{P}\qmat{P}^{T}}      \\
    &~~+ 3 (\alpha^2-\alpha)t^2(1+\mu)   \normFSquare{\qmat{Q}_\perp^{T}\qmat{A}\qmat{P}_\perp} 
\end{align*}
with failure probability $4e^{-2\mu^2/2} + 2t^{-(d-s)}$. When $d>4s$ the last term becomes non-positive. As a result, when $d>4s$ and conditioned on $\qmat{P},\qmat{Q}$, the above relation together with \eqref{eq:SM_decompose_A_Ahat} yields 
\begin{align*}
    \normFSquare{\qmat{A} -\hat{\qmat{A}}}  
     \leq \bigxiaokuohao{1+ 3 \alpha t^2(1+\mu) }\bigxiaokuohao{   \normFSquare{\bigxiaokuohao{\qmat{I}-\qmat{Q}\qmat{Q}^{T}}\qmat{A} } + \normFSquare{\qmat{A}\bigxiaokuohao{\qmat{I}-\qmat{P}\qmat{P}^{T}}}  }
\end{align*}
with failure probability $4e^{-2\mu^2/2} + 2t^{-(d-s)}$. This in connection with our previous estimations of $\normF{\qmat{A}-\qmat{Q}\qmat{Q}^{T}\qmat{A}}$ in   Theorem \ref{thm:prob_projection_error_q_gt_1}  of the main manuscript yields the final deviation bound for the TYUC19-SPI. We omit the detailed results. 

\bibliographystyle{siamplain}
\bibliography{references.bib}

@article{mitliagkas2013memory,
  title={Memory limited, streaming {PCA}},
  author={Mitliagkas, Ioannis and Caramanis, Constantine and Jain, Prateek},
  journal={Advances in Neural Information Processing Systems},
  volume={26},
  year={2013}
}

@book{vershynin2018high,
  title={High-dimensional probability: An introduction with applications in data science},
  author={Vershynin, Roman},
  volume={47},
  year={2018},
  publisher={Cambridge university press}
}

@article{rokhlin2010randomized,
  title={A randomized algorithm for principal component analysis},
  author={Rokhlin, Vladimir and Szlam, Arthur and Tygert, Mark},
  journal={SIAM J.   Matrix Anal. Appl.},
  volume={31},
  number={3},
  pages={1100--1124},
  year={2010},
  publisher={SIAM}
}

@article{upadhyay2018price,
  title={The price of privacy for low-rank factorization},
  author={Upadhyay, Jalaj},
  journal={Advances in Neural Information Processing Systems},
  volume={31},
  year={2018}
}

@article{higham2022mixed,
  title={Mixed precision algorithms in numerical linear algebra},
  author={Higham, Nicholas J and Mary, Theo},
  journal={Acta Numer.},
  volume={31},
  pages={347--414},
  year={2022},
  publisher={Cambridge University Press}
}

@article{carson2024single,
  title={Single-pass Nystr{\"o}m approximation in mixed precision},
  author={Carson, Erin and Dau{\v{z}}ickait{\.e}, Ieva},
  journal={SIAM J.   Matrix Anal. Appl.},
  volume={45},
  number={3},
  pages={1361--1391},
  year={2024},
  publisher={SIAM}
}

@article{connolly2022randomized,
  title={Randomized low rank matrix approximation: Rounding error analysis and a mixed precision algorithm},
  author={Connolly, Michael P and Higham, Nicholas J and Pranesh, Srikara},
  year={2022},
  url={https://eprints.maths.manchester.ac.uk/2863/1/paper.pdf},
}

@article{bjarkason2019PassEfficientRandomizedAlgorithms,
  title = {Pass-{{Efficient Randomized Algorithms}} for {{Low-Rank Matrix Approximation Using Any Number}} of {{Views}}},
  author = {Bjarkason, Elvar K.},
  year = {2019},
  journal = {SIAM J. Sci. Comput.},
  volume = {41},
  number = {4},
  pages = {A2355-A2383},
}

@techreport{martinsson2006randomizedalg,
  author       = {Per-Gunnar Martinsson and Vladimir Rokhlin and Mark Tygert},
  title        = {A randomized algorithm for the approximation of matrices},
  institution  = {Yale University, Computer Science Department},
  year         = {2006},
  type         = {Technical Report},
  number       = {YALEU/DCS/RR-1361}
}

@article{gu2015subspace,
  title={Subspace iteration randomization and singular value problems},
  author={Gu, Ming},
  journal={SIAM J.   Sci. Comput.},
  volume={37},
  number={3},
  pages={A1139--A1173},
  year={2015},
  publisher={SIAM}
}

@article{li2017algorithm,
  title={Algorithm 971: An implementation of a randomized algorithm for principal component analysis},
  author={Li, Huamin and Linderman, George C and Szlam, Arthur and Stanton, Kelly P and Kluger, Yuval and Tygert, Mark},
  journal={ACM Trans.   Math. Softw.},
  volume={43},
  number={3},
  pages={1--14},
  year={2017},
  publisher={ACM New York, NY, USA}
}

@article{voronin2015rsvdpack,
  title={{RSVDPACK}: An implementation of randomized algorithms for computing the singular value, interpolative, and {CUR} decompositions of matrices on multi-core and GPU architectures},
  author={Voronin, Sergey and Martinsson, Per-Gunnar},
  journal={arXiv:1502.05366},
  year={2015}
}

@article{saibaba2023randomized,
  title={Randomized low-rank approximations beyond {G}aussian random matrices},
  author={Saibaba, Arvind K and Mi{{e}}dlar, Agnieszka},
  journal={arXiv preprint arXiv:2308.05814},
  year={2023}
}

@article{boutsidis2013improved,
  title={Improved matrix algorithms via the subsampled randomized Hadamard transform},
  author={Boutsidis, Christos and Gittens, Alex},
  journal={SIAM J. Matrix Anal.   Appl.},
  volume={34},
  number={3},
  pages={1301--1340},
  year={2013},
  publisher={SIAM}
}

@article{yu2018efficient,
  title={Efficient randomized algorithms for the fixed-precision low-rank matrix approximation},
  author={Yu, W. and Gu, Y. and Li, Y.},
  journal={SIAM J. Matrix Anal. Appl.},
  volume={39},
  number={3},
  pages={1339--1359},
  year={2018},
  publisher={SIAM}
}

@article{BoundsVariationMatrix1994,
title = {Bounds for the variation of matrix eigenvalues and polynomial roots},
journal = {Linear Algebra     Appl.},
volume = {208-209},
pages = {73-82},
year = {1994},
author = {Gerd M. Krause}
}

@inproceedings{cohen2015dimensionalityreduction,
author = {Cohen, M. B. and Elder, S. and Musco, C. and Musco, C. and Persu, M.},
title = {Dimensionality Reduction for $k$-Means Clustering and Low Rank Approximation},
year = {2015},
publisher = {Association for Computing Machinery},
address = {New York, NY, USA},
booktitle = {Proceedings of the Forty-Seventh Annual ACM Symposium on Theory of Computing},
pages = {163-–172},
numpages = {10},
}

@article{martinsson2020RandomizedNumerical,
  title = {Randomized Numerical Linear Algebra: foundations and Algorithms},
  author = {Martinsson, P.-G. and Tropp, J. A.},
  year = {2020},
  journal = {Acta Numer.},
  volume = {29},
  pages = {403--572},
}

@article{murray2023RandomizedNumerical,
  title = {Randomized Numerical Linear Algebra: A Perspective on the Field With an Eye to Software},
  author = {Murray, R. and Demmel, J. and Mahoney, M. W. and Erichson, N. B. and Melnichenko, M. and Malik, O. A. and Grigori, L. and Luszczek, P. and Derezi{\'n}ski, M. and Lopes, M. E. and Liang, T. and Luo, H. and Dongarra, J.},
  year = {2023},
  journal = {arXiv:2302.11474},
}

@article{kireeva2024RandomizedMatrix,
  title = {Randomized Matrix Computations: {{Themes}} and Variations},
  author = {Kireeva, A. and Tropp, J. A.},
  year = {2024},
  journal = {arXiv:2402.17873},
}

@article{kannan2017RandomizedAlgorithms,
  title = {Randomized Algorithms in Numerical Linear Algebra},
  author = {Kannan, R. and Vempala, S.},
  year = {2017},
  journal = {Acta Numer.},
  volume = {26},
  pages = {95--135},
}

@article{tropp2023RandomizedAlgorithms,
  title = {Randomized Algorithms for Low-Rank Matrix Approximation: Design, Analysis, and Applications},
  author = {Tropp, J. A. and Webber, R. J.},
  year = {2023},
  journal = {arXiv:2306.12418},
}

@article{FastMontecarlo,
title={Fast {Monte-Carlo} algorithms for finding low-rank approximations},
author={Frieze, A. and Kannan, R. and Vempala, S.},
journal={J.  ACM},
volume={51},
number={6},
pages={1025--1041},
year={2004},
}

@inproceedings{boutsidis2016OptimalPrincipal,
  title = {Optimal Principal Component Analysis in Distributed and Streaming Models},
  booktitle = {Proceedings of the Forty-Eighth Annual {{ACM}} Symposium on {{Theory}} of {{Computing}}},
  author = {Boutsidis, C. and Woodruff, D. P. and Zhong, P.},
  year = {2016},
  series = {{{STOC}} '16},
  pages = {236--249},
  publisher = {Association for Computing Machinery},
}

@article{woodruff2014SketchingTool,
  title = {Sketching as a Tool for Numerical Linear Algebra},
  author = {Woodruff, D. P.},
  year = {2014},
  journal = {Found. Trends  Theor. Comput. Sci.},
  volume = {10},
  number = {1-2},
  pages = {1--157},
}

@article{mahoney2011RandomizedAlgorithms,
  title = {Randomized Algorithms for Matrices and Data},
  author = {Mahoney, M. W.},
  year = {2011},
  journal = {Found. Trends Mach. Learn.},
  volume = {3},
  number = {2},
  pages = {123--224},
}

@inproceedings{clarkson2009NumericalLinear,
  title = {Numerical Linear Algebra in the Streaming Model},
  booktitle = {Proceedings of the Forty-First Annual {{ACM}} Symposium on {{Theory}} of Computing},
  author = {Clarkson, K. L. and Woodruff, D. P.},
  year = {2009},
  pages = {205--214},
  publisher = {ACM},
}

@article{woolfe2008Fast,
  title = {A Fast Randomized Algorithm for the Approximation of Matrices},
  author = {Woolfe, F. and Liberty, E. and Rokhlin, V. and Tygert, M.},
  year = {2008},
  journal = {Appl.   Comput. Harmon. Anal.},
  volume = {25},
  number = {3},
  pages = {335--366},
}

@article{FindingStructureHalko,
  title = {Finding Structure with Randomness: probabilistic Algorithms for Constructing Approximate Matrix Decompositions},
  author = {Halko, N. and Martinsson, P. G. and Tropp, J. A.},
  year = {2011},
  journal = {SIAM Rev.},
  volume = {53},
  number = {2},
  pages = {217--288},
}

@article{libertyRandomizedAlgorithmsLowrank2007,
  title = {Randomized Algorithms for the Low-Rank Approximation of Matrices},
  author = {Liberty, Edo and Woolfe, Franco and Martinsson, Per-Gunnar and Rokhlin, Vladimir and Tygert, Mark},
  year = {2007},
  journal = {Proc. Natl. Acad. Sci. U. S. A.},
  volume = {104},
  number = {51},
  pages = {20167--20172}
}

@article{martinssonRandomizedAlgorithmDecomposition2011,
  title = {A Randomized Algorithm for the Decomposition of Matrices},
  author = {Martinsson, Per-Gunnar and Rokhlin, Vladimir and Tygert, Mark},
  year = {2011},
  journal = {Appl. Comput. Harmon. Anal.},
  volume = {30},
  number = {1},
  pages = {47--68},
  urldate = {2024-01-08},
  langid = {english}
}

@article{nakatsukasaFastStableRandomized2020,
  title = {Fast and Stable Randomized Low-Rank Matrix Approximation},
  author = {Nakatsukasa, Yuji},
  year = {2020},
  journal = {arXiv:2009.11392},
}

@article{Practical_Sketching_Algorithms_Tropp,
  title = {Practical Sketching Algorithms for Low-Rank Matrix Approximation},
  author = {Tropp, J. A. and Yurtsever, A. and Udell, M. and Cevher, V.},
  year = {2017},
  journal = {SIAM J. Matrix Anal.  Appl.},
  volume = {38},
  number = {4},
  pages = {1454--1485},
}

@article{troppStreamingLowRankMatrix2019,
  title = {Streaming Low-Rank Matrix Approximation with an Application to Scientific Simulation},
  author = {Tropp, J. A. and Yurtsever, A. and Udell, M. and Cevher, V.},
  year = {2019},
  journal = {SIAM J. Sci. Comput.},
  volume = {41},
  number = {4},
  pages = {A2430-A2463},
}

@article{higham2019new,
  title={A new approach to probabilistic rounding error analysis},
  author={Higham, N. J. and Mary, T.},
  journal={SIAM J.   Sci. Comput.},
  volume={41},
  number={5},
  pages={A2815--A2835},
  year={2019},
  publisher={SIAM}
}

@article{nakatsukasa2023randomized,
  title={Randomized low-rank approximation for symmetric indefinite matrices},
  author={Nakatsukasa, Y. and Park, T.},
  journal={SIAM J.  Matrix Anal. Appl.},
  volume={44},
  number={3},
  pages={1370--1392},
  year={2023},
  publisher={SIAM}
}

@inproceedings{narang2017mixed,
  title={Mixed precision training},
  author={Narang, Sharan and Diamos, Gregory and Elsen, Erich and Micikevicius, Paulius and Alben, Jonah and Garcia, David and Ginsburg, Boris and Houston, Michael and Kuchaiev, Oleksii and Venkatesh, Ganesh and others},
  booktitle={Int. Conf.   Learning Representation},
  year={2017}
}

@article{kong2017spectrum,
  title        = {Spectrum Estimation from Samples},
  author       = {Kong, W. and Valiant, G.},
  journal      = {Ann. Statist.},
  volume       = {45},
  number       = {5},
  pages        = {2218--2247},
  month        = oct,
  year         = {2017},
}

@inproceedings{li2014sketching,
  title        = {On sketching matrix norms and the top singular vector},
  author       = {Li, Y. and Nguy{\~{\^e}}n, H. L. and Woodruff, D. P.},
  booktitle    = {Proceedings of the twenty-fifth annual ACM-SIAM symposium on Discrete algorithms},
  pages        = {1562--1581},
  year         = {2014},
  organization = {SIAM},
}

@article{drineas2006FastMonte,
  title = {Fast {M}onte {C}arlo algorithms for  matrices {I}: Approximating Matrix Multiplication},
  author = {Drineas, P. and Kannan, R. and Mahoney, M. W.},
  year = {2006},
  journal = {SIAM J.   Comput.},
  volume = {36},
  number = {1},
  pages = {132--157}
}

@article{che2025efficient,
  author = {Che, M. and Wei, Y.},
  journal = {Adv. Comput.},
  pages = {20},
  title = {Efficient algorithms for {T}ucker decomposition via approximate matrix multiplication},
  volume = {51},
  year = {2025}
}

@article{avron2010BlendenpikSupercharging,
  title = {Blendenpik: Supercharging {LAPACK}'s Least-Squares Solver},
  author = {Avron, H. and Maymounkov, P. and Toledo, S.},
  year = {2010},
  journal = {SIAM J.   Sci. Comput.},
  volume = {32},
  number = {3},
  pages = {1217--1236}
}

@misc{troppComparisonTheoremsMinimum2025,
  title = {Comparison Theorems for the Minimum Eigenvalue of a Random Positive-Semidefinite Matrix},
  author = {Tropp, J. A.},
  year = {2025},
  number = {arXiv:2501.16578},
  eprint = {2501.16578},
}

\end{document}